\documentclass[a4,11pt]{amsart}
\usepackage[top=3cm, bottom=3cm, left=2.5cm, right=2.5cm]{geometry}

\usepackage{amssymb,amsmath,amsthm,txfonts}

\usepackage{color}
\usepackage{hyperref}
\usepackage{cite}
\usepackage{nicefrac}

\usepackage{textcomp}
\usepackage{protosem}

\usepackage{graphicx}
\usepackage{xcolor}

\usepackage{lscape}
\usepackage[hang,small,bf]{caption}
\usepackage[subrefformat=parens]{subcaption}
\newtheorem{thm}{Theorem}[section]
\newtheorem{lem}[thm]{Lemma}

\newtheorem{prop}[thm]{Proposition}
\theoremstyle{definition}
\newtheorem{defn}[thm]{Definition}
\theoremstyle{remark}
\newtheorem{rem}[thm]{\textbf{Remark}}

\newtheorem{q}[thm]{\textbf{Question}}

\DeclareMathOperator{\peye}{\textproto{o}}

 \makeatletter
    
    \@addtoreset{equation}{section}
  \makeatother

\makeatletter
      \def\@makefnmark{%
         \leavevmode
            \raise.9ex\hbox{\check@mathfonts
                \fontsize\sf@size\z@\normalfont%
                            \@thefnmark}%
       }
      \makeatother

\makeatletter
\@namedef{subjclassname@2020}{\textup{2020} Mathematics Subject Classification}
\makeatother

\newcommand{\dd}{\textrm{d}}

\begin{document}

\title[]{Stationary self-similar profiles for the two-dimensional inviscid Boussinesq equations}
\author[]{Ken Abe}
\author[]{Daniel Ginsberg}
\author[]{In-Jee Jeong}
\date{}
\address[Ken Abe]{Department of Mathematics, Graduate School of Science, Osaka Metropolitan University, 3-3-138 Sugimoto, Sumiyoshi-ku Osaka, 558-8585, Japan}
\email{kabe@omu.ac.jp}
\address[Daniel Ginsberg]{Department of Mathematics, Brooklyn College (CUNY), Brooklyn, NY 11210, USA}
\email{daniel.ginsberg@brooklyn.cuny.edu}
\address[In-Jee Jeong]{Department of Mathematical Sciences and RIM, Seoul National University, Seoul 08826, Korea}
\email{injee\_j@snu.ac.kr}

\subjclass[2020]{35Q31, 35Q35}
\keywords{Self-similar solutions, Dubreil-Jacotin--Long equation, Minimax theorems, Singular elliptic problem}
\date{\today}

\begin{abstract}
We consider ($-\alpha$)-homogeneous solutions (stationary self-similar solutions of degree $-\alpha$) to the two-dimensional inviscid Boussinesq equations in a half-plane. We show their non-existence and existence with both regular and singular profile functions. More specifically, we demonstrate:  \\ 

\begin{itemize}
\item Non-existence of rotational ($-\alpha$)-homogeneous solutions with regular profiles\\ $(u,p,\rho)\in C^{1}(\overline{\mathbb{R}^{2}_{+}}\backslash \{0\})$ for $0\leq \alpha\leq 1$ and $(u,p,\rho)\in C^{2}(\overline{\mathbb{R}^{2}_{+}}\backslash \{0\})$ for $-1/2\leq \alpha<0$ \\
\item Existence of rotational ($-\alpha$)-homogeneous solutions with regular profiles\\ $(u,p,\rho)\in C^{2}(\overline{\mathbb{R}^{2}_{+}}\backslash \{0\})$ for $\alpha>1$ and $(u,p,\rho)\in C^{1}(\overline{\mathbb{R}^{2}_{+}})$ for $\alpha<-2$\\
\item Existence of rotational ($-\alpha$)-homogeneous solutions with $x_1$-symmetric singular profiles\\
 $(u,p,\rho) \in C^{\infty}(\overline{\mathbb{R}^{2}_{+}}\backslash \{x_1=0\}\cup \{x_2=0\})\cap C(\overline{\mathbb{R}^{2}_{+}})$ for $-1<\alpha< -1/2$ and $(u,p,\rho)\in C^{\infty}(\overline{\mathbb{R}^{2}_{+}}\backslash \{x_1=0\}\cup \{x_2=0\})$ for $-1/2\leq  \alpha<1$\\
\end{itemize}
The ($-\alpha$)-homogeneous solutions with continuous profiles $(u,p,\rho)\in C^{\infty}(\overline{\mathbb{R}^{2}_{+}}\backslash \{x_1=0\}\cup \{x_2=0\})\cap C(\overline{\mathbb{R}^{2}_{+}})$ for $-1<\alpha<-1/2$ provide examples to self-similar weak solutions in $\mathbb{R}^{2}_{+}$ for the scaling exponent $\alpha\approx -0.617$, at which Wang et al. (2023) numerically discovered the existence of backward self-similar solutions with smooth profile functions.
\end{abstract}

\maketitle

\tableofcontents

\section{Introduction}
The 2D inviscid Boussinesq equations
\begin{equation}
\begin{aligned}
u_t+u\cdot \nabla u+\nabla p&=-\rho e_{2},\\
\nabla \cdot u&=0,\\
\rho_t+u\cdot \nabla \rho&=0,
\end{aligned}
\end{equation}
describe the velocity $u(x,t)=(u^{1}(x,t),u^{2}(x,t))$, the pressure $p(x,t)$, and the density $\rho(x,t)$ of stratified fluids when the density variation is small compared with the mean density. These equations can be derived from the equations of thermodynamics by the Boussinesq approximation \cite{ob}, \cite{oberbeck88}, \cite{Bou}, incorporating the buoyancy force $-\rho e_{2}$ for $e_{2}=(0,1)$ and are widely used to model stratified flows in geophysical fluid dynamics, e.g., \cite[2.4]{Vallis17}, \cite[Chapter 4]{Dav}, \cite[Chapter 2]{DR82}, \cite{MA}. The Boussinesq equations in a half-plane also share the essential structure of the axisymmetric Euler equations
with swirl \cite{MaB}. The problem of singularity formation for these two equations has attracted attention during the last decade; see Kiselev \cite{KiselevICM} and Drivas and Elgindi \cite[5.2]{DE} for surveys.

\subsection{Self-similar solutions}

We consider self-similar solutions to the Boussinesq equations (1.1). We say that $(u,p,\rho)$ is a backward (resp. forward) self-similar solution if 
\begin{align*}
u(x,t)=\lambda^{\alpha}u(\lambda x,\lambda^{1+\alpha}t ),\quad  p(x,t)=\lambda^{2\alpha}u(\lambda x,\lambda^{1+\alpha}t ), \quad  
\rho(x,t)=\lambda^{2\alpha+1}\rho(\lambda x,\lambda^{1+\alpha}t ),
\end{align*}
for some $\alpha\in  \mathbb{R}$ and for all $\lambda>0$, $x\in \mathbb{R}^{2}$ and $t<0$ (resp. $t>0$) \cite{Chae07}. This self-similarity is of the \textit{second} kind \cite{ZR}, \cite{Bare96}, \cite{EF09} in the sense that the scaling exponent $\alpha\in \mathbb{R}$ is free in contrast to the viscous and diffusive Boussinesq equations
where $\alpha=1$ is dictated by the scaling of the problem (self-similarity of the first kind). 

Self-similarity of the second kind plays an important role in the study of
singularity formation in many equations of hydrodynamics:

\begin{itemize}
\item \textbf{One-dimensional models}: Various backward self-similar solutions have been discovered for the gCLM equations in the works of Elgindi and Jeong \cite{EJ20c}, Chen et al. \cite{CHH21},  Chen \cite{Chen20}, Lushnikov et al. \cite{LSS}, Huang et al. \cite{HQWW}, \cite{HQWW}, Zhen \cite{Z23}, and Jia and \v{S}ver\'{a}k \cite{SverakVideo}. They have the same scaling law as the inviscid Boussinesq equations and possess exponents distributing over the range $\alpha\in [-2,2]$; see Table 1 in Appendix A.
\item \textbf{Shock formation}: Backward self-similar solutions to the 1D inviscid Burgers equation with smooth profile functions exist for $\alpha=-1/(2i+1)$ and $i\in \mathbb{N}$, e.g., \cite{EF09}. Higher dimensional anisotropic self-similar solutions for $\alpha=-1/3$ involving additional scaling exponents have been obtained in the work of Collot et al. \cite{CGM22} in the study of the 2D Burgers equations with transverse viscosity and the work of Buckmaster et al. \cite{BVS22} in the study of shock and singularity formation of the 3D isentropic compressible Euler equations, cf. \cite{BSV23}, \cite{BSV23b}.
\item \textbf{Boundary layer separation}: Collot et al. \cite{CGM21} discovered two explicit backward self-similar solutions to the 2D inviscid Prandtl equations for $\alpha=-1/3$ with two different additional exponents modeling generic separation singularity relevant to Van-Dommelen--Shen boundary layer singularity and degenerate singularity relevant to Burgers shock; see \cite{CGIM22} for a stability result in the Prandtl boundary layer equation.
\item \textbf{Implosions in compressible fluids}: Merle et al. \cite{MRRS} discovered backward self-similar solutions to the isentropic compressible Euler equations for exponents $\{\alpha_n\}\subset (0,\alpha_{\peye})$ accumulating at the constant $\alpha_{\peye}$ with smooth spherically symmetric profile functions slowly decaying at infinity. They describe blow-ups of isentropic compressible Navier--Stokes equations \cite{MRRS2} and supercritical defocusing nonlinear Schr\"{o}dinger equations \cite{MRRS3}. Buckmaster et al. \cite{BCG} demonstrated that $3$D imploding solutions exist for all adiabatic exponents $\gamma>1$; see \cite{BCG2} for a review. 
\end{itemize} 


Returning to the Boussinseq equations (1.1), we distinguish the following four cases, cf.
\cite[Definition 1.1]{CHH21}, \cite[Definition 4.8, 4.1.10]{DE}:\\

\noindent
\textbf{Case 1}: $\alpha=-1$ (Scale-invariant). The equation (1.1) is reduced to a system for $1$-homogeneous velocity and density. Elgindi and Jeong \cite{EJ20} constructed finite energy blow-up solutions to the 2D inviscid Boussinesq equations in a corner domain by studying solutions of this type, cf. \cite[Problem 8]{DE}. See also Sarria and Wu \cite{SW15} for infinite energy blow-up solutions in a strip.

\noindent
\textbf{Case 2}: $\alpha>-1$ (Concentrating). (Time-dependent) backward self-similar solutions take the form 
\begin{align*}
u(x,t)=\frac{1}{(-t)^{\frac{\alpha}{\alpha+1}}}u\left(\frac{x}{ (-t)^{\frac{1}{\alpha+1}}} ,-1\right),\quad p(x,t)=\frac{1}{(-t)^{\frac{2\alpha}{\alpha+1}}}p\left(\frac{x}{ (-t)^{\frac{1}{\alpha+1}}} ,-1\right), \quad \rho(x,t)=\frac{1}{(-t)^{\frac{2\alpha+1}{\alpha+1}}}\rho\left(\frac{x}{ (-t)^{\frac{1}{\alpha+1}}} ,-1\right),
\end{align*}
where the profile function $(u,p,\rho)(x,-1)$ satisfies the self-similar equations:
\begin{equation}
\begin{aligned}
\frac{\alpha}{\alpha+1}u+\frac{1}{\alpha+1}x\cdot \nabla u +u\cdot \nabla u+\nabla p&=-\rho e_2,\\
\nabla \cdot u&=0,\\
\frac{2\alpha+1}{\alpha+1}\rho+\frac{1}{\alpha+1}x\cdot \nabla \rho +u\cdot \nabla \rho&=0.
\end{aligned}
\end{equation}
Elgindi \cite[Theorem 1 and Remark 1.3]{Elgindi} provided the first example of backward self-similar solutions to the Euler equations in this class with some axisymmetric profile without swirl,
and satisfying $u(x,-1)\in C^{1,\gamma}(\mathbb{R}^{3})$ for small $\gamma \in (0,1)$. See also Elgindi et al. \cite{EGM2} for a study
of their stability. Chen and Hou \cite{CH21} constructed asymptotically self-similar blow-up solutions to the 2D inviscid Boussinesq equations in a half-plane for some $C^{1,\gamma}$ initial data. The recent works \cite{CH22a}, \cite{CH22b} significantly extended the result for smooth initial data with a computer-assisted proof.

\noindent
\textbf{Case 3}: $\alpha<-1$ (Expanding). Chen et al. \cite{CHH21} discovered this type of backward self-similar solution for the De Gregorio model of the 3D Euler equations. See also Huang et al. \cite{HTW}, cf. \cite{SverakVideo}. They also exist for the generalized Constantin--Lax--Majda (gCLM) equations \cite{HQWW}.

\noindent
\textbf{Case 4}: $\alpha=\infty$ (Separation of variables). 
Chen \cite{Chen21} found this type of self-similar solutions for the gCLM equations inspired by the numerical work of Lushnikov et al. \cite{LSS}, cf. \cite{HQWW}.\\

Wang et al. \cite{WLGB} explored the existence of solutions to the backward self-similar equations (1.2) in a half-plane. The work \cite{WLGB} numerically searched for \textit{smooth} solutions satisfying the conditions $\nabla u$, $\nabla \rho\to 0$ as $|x|\to\infty$ as well as the $x_1$-symmetry 
\begin{equation}
\begin{aligned}
u^{1}(x_1,x_2)&=-u^{1}(-x_1,x_2),\quad
u^{2}(x_1,x_2)=u^{2}(-x_1,x_2), \\
p(x_1,x_2)&=p(-x_1,x_2),\quad
\rho(x_1,x_2)=\rho(-x_1,x_2),
\end{aligned}
\end{equation}
for \textit{unknown} $\alpha>-1$ by using neural networks, and found solutions 
at the exponent $\alpha\approx -0.617$ (in the notation of \cite[p.3]{WLGB}, $\mu=-\alpha/(\alpha+1)\approx 1.917$). The work \cite{WLGB} also found solutions to (1.2) peaked at the origin for the \textit{fixed} exponent $\alpha=-0.75$ ($\mu=3$); see Figure 1.  

In light of the above, it is natural to ask the following question:
\begin{q}
For which exponents $\alpha\in \mathbb{R}$ do self-similar solutions to the problem (1.1) exist? 
\end{q}

\subsection{The statement of the main results}

\subsubsection{Self-similar Boussinesq solutions}
This work aims to answer Question 1.1
in the class of stationary solutions. We consider ($-\alpha$)-homogeneous (stationary self-similar) solutions of the form   
\begin{align*}
u(x)=\frac{1}{|x|^{\alpha}}u\left(\frac{x}{|x|}\right),\quad p(x)=\frac{1}{|x|^{2\alpha}}p\left(\frac{x}{|x|}\right),\quad \rho(x)=\frac{1}{|x|^{2\alpha+1}}\rho\left(\frac{x}{|x|}\right).
\end{align*}
The $(-\alpha)$-homogeneous solutions to (1.1) are also solutions to the self-similar equations (1.2). We identify solutions of (1.1) in a half-plane with those in the whole plane under the $x_2$-symmetry: 
\begin{equation}
\begin{aligned}
u^{1}(x_1,x_2)&=u^{1}(x_1,-x_2),\quad
u^{2}(x_1,x_2)=-u^{2}(x_1,-x_2), \\
p(x_1,x_2)&=p(x_1,-x_2),\quad
\rho(x_1,x_2)=-\rho(x_1,-x_2).
\end{aligned}
\end{equation}


We say that ($-\alpha$)-homogeneous solutions $(u,p,\rho)$ have \textit{regular} profiles if $(u,p,\rho)$ is a classical stationary solution to (1.1) in $\mathbb{R}^{2}\backslash \{0\}$. Otherwise, we say that ($-\alpha$)-homogeneous solutions $(u,p,\rho)$ have \textit{singular} profiles. First, we state one of the main results about the non-existence and existence of rotational ($-\alpha$)-homogeneous solutions with regular profiles. We give a concrete characterization of irrotational ($-\alpha$)-homogeneous solutions with regular profiles in Theorem 1.9.

\begin{thm}[Rotational solutions with regular profiles]
The following holds for rotational ($-\alpha$)-homogeneous solutions to (1.1) and (1.4):

\noindent
(i) For $0\leq \alpha \leq 1$, no solutions $(u,p,\rho) \in C^{1}(\mathbb{R}^{2}\backslash \{0\})$ exist.\\
(ii) For $-1/2\leq \alpha < 0$, no solutions $(u,p,\rho) \in C^{2}(\mathbb{R}^{2}\backslash \{0\})$ exist.\\
(iii) For $\alpha>1$, solutions $(u,p,\rho)\in C^{2}(\mathbb{R}^{2}\backslash \{0\})$ exist.\\
(iv) For $\alpha<-2$, solutions $(u,p,\rho)\in C^{1}(\mathbb{R}^{2})$ exist.
\end{thm}

Figure 1 (A) shows the non-existence and existence ranges of $\alpha\in \mathbb{R}$ for rotational ($-\alpha$)-homogeneous solutions with regular profiles. The main result of this paper is the existence of ($-\alpha$)-homogeneous solutions to (1.1) with singular profiles for $-1<\alpha<1$ satisfying the $x_1$-symmetry (1.3). They provide examples of stationary self-similar weak solutions for $\alpha\approx -0.617$, at which Wang et al. \cite{WLGB} numerically discovered the existence of backward self-similar solutions to (1.1) with smooth profile functions.

\begin{thm}[Rotational solutions with singular profiles]
The following holds for rotational ($-\alpha$)-homogeneous solutions to (1.1), (1.3), and (1.4):

\noindent
(i) For $-1<\alpha< -1/2$, solutions $(u,p,\rho)\in C^{\infty}(\mathbb{R}^{2}\backslash \{x_1=0\}\cup\{x_2=0\})\cap C(\mathbb{R}^{2})$ exist. \\
(ii) For $-1/2\leq \alpha<1$, solutions $(u,p,\rho)\in C^{\infty}(\mathbb{R}^{2}\backslash \{x_1=0\}\cup\{x_2=0\})$ exist.

The solutions in (i) are weak solutions to (1.1) in $\mathbb{R}^{2}$ and the $C^{1}$-norm of $u$ diverges on the axis $\{x_1=0\}$. The velocity $u$ is ($-\alpha$)-H\"older continuous for the radial variable and ($2+2/(\alpha-1)$)-H\"older continuous for the angular variable. The solutions in (ii) satisfy (1.1) in $\mathbb{R}^{2}\backslash \{x_1=0\}\cup \{x_2=0\}$ and the $C^{1}$-norm of $u$ diverges on the axis $\{x_1=0\}$.   
\end{thm}

\begin{figure}[h]
     \begin{minipage}[b]{0.45\linewidth}
\hspace{-128pt}
\includegraphics[scale=0.23]{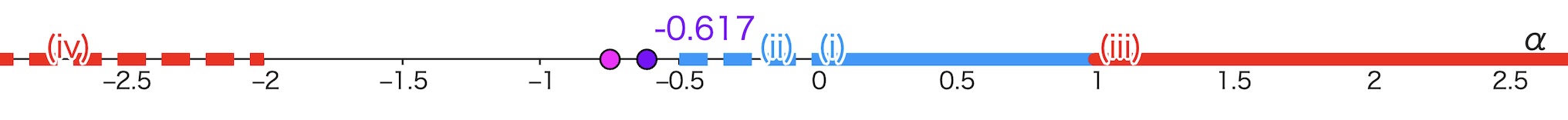}
    \subcaption{($-\alpha$)-homogeneous solutions with regular profiles in Theorem 1.2}
  \end{minipage}\\
 \begin{minipage}[b]{0.45\linewidth}
\hspace{-128pt}
\includegraphics[scale=0.23]{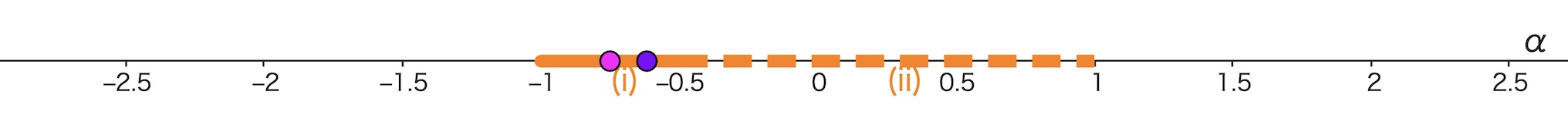}
    \subcaption{($-\alpha$)-homogeneous solutions with singular profiles in Theorem 1.3 (\textbf{the main result of this paper})}
  \end{minipage}\\
  \caption{The non-existence and existence ranges of $\alpha\in \mathbb{R}$ for rotational ($-\alpha$)-homogeneous solutions in Theorems 1.2 and 1.3. The purple point ($\alpha= -0.617$) represents the scaling exponent of numerical backward self-similar solutions with smooth profile functions \cite{WLGB}. The pink point ($\alpha=-0.75$) represents the fixed scaling exponent at which numerical backward self-similar solutions with peaked profile functions exist  \cite{WLGB}.}
\end{figure}

\begin{rem}[$-1<\alpha<-1/2$]
The non-existence and existence of rotational ($-\alpha$)-homogeneous solutions for $-1<\alpha< -1/2$ with regular profiles are generally unknown; see Theorems 1.5 and 1.7. We remark that the profile of backward self-similar solutions to the axisymmetric Euler equations without swirl \cite{Elgindi} are \textit{regular} and belong to $C^{1,\gamma}$ for small $\gamma\in (0,1)$. 
\end{rem}

\vspace{3pt}

\subsubsection{Self-similar Euler solutions}

In the sequel, we consider two particular stationary solutions to (1.1): the 2D Euler solutions $\rho\equiv 0$ and constant Bernoulli function solutions $\Pi=p+|u|^{2}/2+x_2\rho\equiv \textrm{const.}$ (Pseudo-Beltrami solutions). 

Elling \cite{Elling} constructed forward self-similar weak solutions to the 2D Euler equations in the whole plane for $-1<\alpha<1/2$ and ($-\alpha-1$)-homogeneous initial vorticity $\omega_0(x)=|x|^{-\alpha-1}\omega_0(x/|x|)$ with $2\pi/m$-periodic $\omega_0$
in the Wiener algebra $\omega_0(\phi)\in A(\mathbb{S}^{1})\subset C(\mathbb{S}^{1})$ for arbitrary $m\geq m_0$ for some $m_0\in \mathbb{N}$. Here, $\omega=\nabla^{\perp}\cdot u$ and $\nabla^{\perp}=(-\partial_2,\partial_1)$. Shao et al. \cite{SWZ} extended the result to $-1<\alpha<1$ and $2\pi/m$-periodic $\omega_0\in L^{1}(\mathbb{S}^{1})$ for arbitrary $m\geq 2$ (and also for measures $\omega_0\in \mathcal{M}(\mathbb{S}^{1})$) (The self-similar equations to forward self-similar solutions are the equations (1.2) replacing the signs of the first two terms in $(1.2)_1$ and $(1.2)_3$ by minus). Jeong \cite{Jeong17} and Elgindi and Jeong \cite{EJ20b} constructed unique scale-invariant weak solutions ($\alpha=-1$) for bounded $0$-homogeneous $m$-fold symmetric initial vorticity for $m\geq 3$; see also Elgindi et al. \cite{Elgindi2022} and Jeong and Said \cite{Jeong23} for long-time behavior.

Garc\'ia and G\'omez-Serrano \cite{Serrano22} constructed forward self-similar weak solutions to the generalized surface quasi-geostrophic (gSQG) equations for $\gamma\in (0,1)$ and $1<\alpha<1+\gamma$ for $2\pi/m$-periodic $\omega_0\in L^{p}(\mathbb{S}^{1})$ for $p>1/(1-\gamma)$ (The cases $\gamma=0$ and $\gamma=1$ correspond to the Euler equations and the SQG equations, respectively; see Figure 2 (A)). 

Theorems 1.5 and 1.6 show the non-existence and existence of rotational ($-\alpha$)-homogeneous 2D Euler solutions with regular profiles and the existence of rotational ($-\alpha$)-homogeneous 2D Euler solutions with singular profiles, respectively. Theorem 1.6 is essentially due to the recent work of Elgindi and Huang \cite[Lemma 5.3]{EH}. This statement includes the case $0<\alpha<1$ and detailed properties of profile functions. Theorem 1.5 is due to \cite[Theorem 1.7]{Abe11}. It turns out that the non-existence range of ($-\alpha$)-homogeneous solutions with regular profiles agrees with the existence range of ($-\alpha$)-homogeneous solutions with singular profiles in this class; see Figure 2 (B) and (C).

\begin{thm}[2D Euler solutions with regular profiles]
The following holds for rotational ($-\alpha$)-homogeneous solutions to (1.1) and (1.4) for $\rho\equiv 0$:

\noindent
(i) For $0\leq \alpha \leq 1$, no solutions $(u,p) \in C^{1}(\mathbb{R}^{2}\backslash \{0\})$ exist.\\
(ii) For $-1\leq \alpha < 0$, no solutions $(u,p) \in C^{2}(\mathbb{R}^{2}\backslash \{0\})$ exist.\\
(iii) For $\alpha>1$, solutions $(u,p)\in C^{2}(\mathbb{R}^{2}\backslash \{0\})$ exist.\\
(iv) For $\alpha<-1$, solutions $(u,p)\in C^{1}(\mathbb{R}^{2})$ exist.
\end{thm}

\begin{thm}[2D Euler solutions with singular profiles]
The following holds for rotational ($-\alpha$)-homogeneous solutions to (1.1), (1.3), and (1.4) for $\rho\equiv 0$:

\noindent
(i) For $-1<\alpha< 0$, solutions $(u,p)\in C^{\infty}(\mathbb{R}^{2}\backslash \{x_1=0\}\cup\{x_2=0\})\cap C(\mathbb{R}^{2})$ exist.\\
(ii) For $0<\alpha<1$, solutions $(u,p)\in C^{\infty}(\mathbb{R}^{2}\backslash \{x_1=0\}\cup\{x_2=0\})$ exist.

The solutions in (i) are weak solutions to (1.1) in $\mathbb{R}^{2}$ and the $C^{1}$-norm of $u$ diverges on the axis $\{x_1=0\}\cup \{x_2=0\}$. The velocity $u$ is ($-\alpha$)-H\"older continuous for the radial variable and ($2+2/(\alpha-1)$)-H\"older continuous for the angular variable. The solutions in (ii) satisfy (1.1) in $\mathbb{R}^{2}\backslash \{x_1=0\}\cup \{x_2=0\}$. The velocity $u$ diverges on the axis $\{x_1=0\}\cup \{x_2=0\}$.
\end{thm}

\vspace{3pt}

\subsubsection{Self-similar Pseudo-Beltrami solutions}

The other class of stationary solutions to (1.1) are solutions with constant Bernoulli function $\Pi=p+|u|^{2}/2+x_2\rho$. We show that this class excludes ($-\alpha$)-homogeneous solutions with regular profiles for $-2<\alpha<-1$ in contrast to the 2D Euler case in Theorem 1.5 (iv), cf. Remark 1.4. On the other hand, the existence range of ($-\alpha$)-homogeneous solutions with singular profiles is restricted to $-1<\alpha<1$.

\begin{thm}[Pseudo-Beltrami solutions with regular profiles]
The following holds for rotational ($-\alpha$)-homogeneous solutions to (1.1) and (1.3) for $\Pi \equiv \textrm{const.}$:

\noindent
(i) For $-1/2\leq \alpha \leq 1$, no solutions $(u,p,\rho) \in C^{1}(\mathbb{R}^{2}\backslash \{0\})$ exist.\\
(ii) For $-2< \alpha < -1/2$, no solutions $(u,p,\rho) \in C^{2}(\mathbb{R}^{2}\backslash \{0\})$ exist. For $\alpha=-2$, no solutions $(u,p,\rho) \in C^{2}(\mathbb{R}^{2}\backslash \{0\})$ exist provided that $\nabla^{\perp}\cdot u/x_2$ vanishes on $x_2=0$.\\
(iii) For $\alpha>1$, solutions $(u,p,\rho) \in C^{2}(\mathbb{R}^{2}\backslash \{0\})$ exist.\\
(iv) For $\alpha<-2$, solutions $(u,p,\rho)\in C^{1}(\mathbb{R}^{2})$ exist.
\end{thm}

\begin{thm}[Pseudo-Beltrami solutions with singular profiles]
The following holds for rotational ($-\alpha$)-homogeneous solutions to (1.1), (1.3), and (1.4) for $\Pi \equiv \textrm{const.}$:

\noindent
(i) For $-1<\alpha< -1/2$, solutions $(u,p,\rho)\in C^{\infty}(\mathbb{R}^{2}\backslash \{x_1=0\}\cup\{x_2=0\})\cap C(\mathbb{R}^{2})$ exist. \\
(ii) For $-1/2<\alpha<1$, solutions $(u,p,\rho)\in C^{\infty}(\mathbb{R}^{2}\backslash \{x_1=0\}\cup\{x_2=0\})$ exist.

The solutions in (i) and (ii) have the same properties as those in Theorem 1.3.
\end{thm}

\vspace{3pt}

\subsubsection{Self-similar irrotational solutions}
We completely classify irrotational ($-\alpha$)-homogeneous solutions to (1.1). Their velocity fields are harmonic vector fields $u\neq 0$ for integers $\alpha\in \mathbb{Z}$ or $u=0$ with $x_2$-dependent functions $p=p(x_2)$ and $\rho=\rho(x_2)$ (static equilibria). We use  polar coordinates $(r,\phi)$ and the  orthogonal frame $e_r=(\cos\phi,\sin\phi)$ and $e_\phi=(-\sin\phi,\cos\phi)$.

\clearpage

 \begin{figure}[h]
 \begin{minipage}[b]{0.45\linewidth}
\hspace{-128pt}
\includegraphics[scale=0.228]{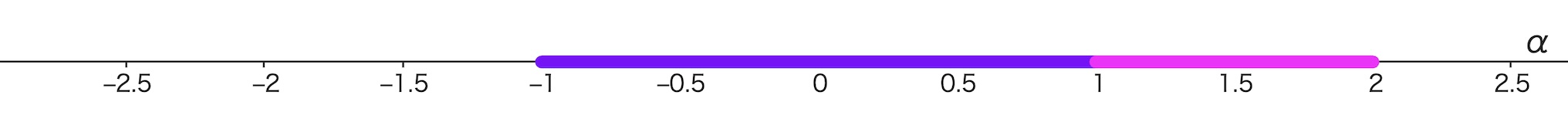}
    \subcaption{Forward self-similar weak solutions to 2D Euler and gSQG equations with scaling exponent $\alpha$}
  \end{minipage}\\
  \begin{minipage}[b]{0.45\linewidth}
\hspace{-128pt}
\includegraphics[scale=0.228]{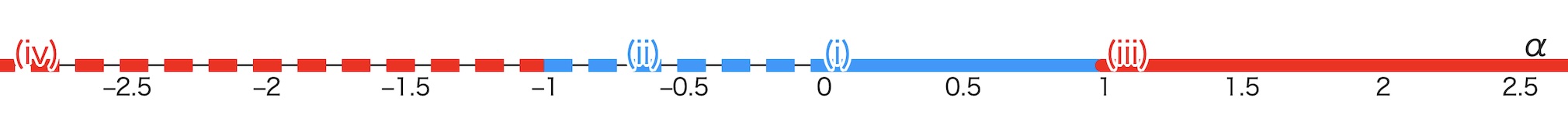}
\subcaption{($-\alpha$)-homogeneous Euler solutions with regular profiles in Theorem 1.5}
  \end{minipage}\\
  \begin{minipage}[b]{0.45\linewidth}
\hspace{-128pt}
\includegraphics[scale=0.228]{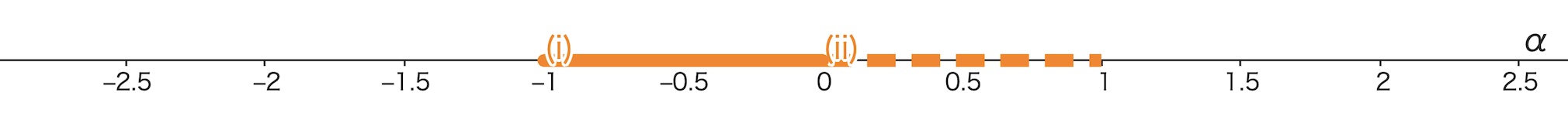}
\subcaption{($-\alpha$)-homogeneous Euler solutions with singular profiles in Theorem 1.6}
  \end{minipage}\\
  \caption{The existence ranges of the scaling exponent $\alpha\in \mathbb{R}$ for forward self-similar weak solutions to the 2D Euler equations (drawn in purple) \cite{Elling}, \cite{SWZ} and gSQG equations (drawn in pink) \cite{Serrano22}, and the non-existence and existence ranges of $\alpha\in \mathbb{R}$ for 2D Euler ($-\alpha$)-homogeneous solutions in Theorems 1.5 and 1.6.}
\end{figure}


 \begin{figure}[h]
 \begin{minipage}[b]{0.45\linewidth}
\hspace{-128pt}
\includegraphics[scale=0.2387]{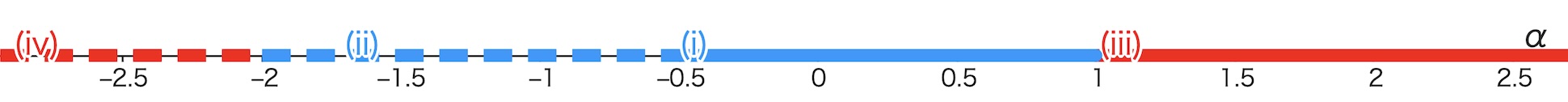}
    \subcaption{($-\alpha$)-homogeneous Pseudo-Beltrami solutions with regular profiles in Theorem 1.7}
  \end{minipage}\\
  \begin{minipage}[b]{0.45\linewidth}
\hspace{-128pt}
\includegraphics[scale=0.229]{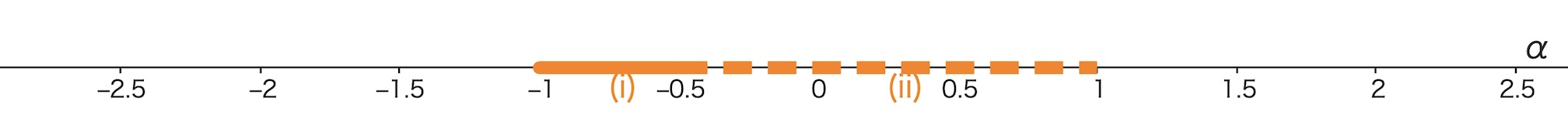}
\subcaption{($-\alpha$)-homogeneous Pseudo-Beltrami solutions with singular profiles in Theorem 1.8}
  \end{minipage}\\
  \caption{The non-existence and existence ranges of $\alpha\in \mathbb{R}$ for Pseudo-Beltrami ($-\alpha$)-homogeneous solutions in Theorems 1.7 and 1.8.}
\end{figure}

\begin{thm}[Irrotational solutions with regular profiles]
Irrotational ($-\alpha$)-homogeneous solutions $(u,p,\rho) \in C^{1}(\mathbb{R}^{2}\backslash \{0\})$ to (1.1) and (1.4) are either the following (i) or (ii):\\
\noindent
(i) The velocity field is the harmonic vector field
\begin{align}
u=\frac{C}{r^{\alpha}}\left(\sin((\alpha-1)\phi)e_{\phi}+\cos((\alpha-1)\phi)e_{r}  \right),\quad \alpha\in \mathbb{Z},  
\end{align}
for some constant $C\in \mathbb{R}\backslash \{0\}$. If $\rho\equiv 0$, there exists a constant $b\in \mathbb{R}$ such that 
\begin{align}
p=\frac{1}{r^{2\alpha}}b,\quad \alpha(C^{2}+2b)=0.  
\end{align}
If $\rho\nequiv 0$, the integers $\alpha$ are negative. 

\noindent
(ii) The solution is the static equilibrium
\begin{align}
u=0,\quad p=p(x_2),\quad \rho(x_2)=-p'(x_2),
\end{align}
If $\rho\equiv 0$, $\alpha=0$ and the pressure $p$ is constant. If $\rho\nequiv 0$, $\alpha<-1$. 
\end{thm}


\subsubsection{Desingularization of self-similar vortex sheets}

We finally show the convergence of ($-\alpha$)-homogeneous solutions of (1.1) with regular profiles to the ($-1$)-homogeneous vortex sheet solution with singular profiles as $\alpha\to 1+0$. This convergence issue appears in current sheet formation in astrophysics \cite{LB94}, \cite{Aly94}. All ($-1$)-homogeneous solutions to (1.1) with regular profiles are the irrotational solution (1.5) by Theorems 1.2 (i) and 1.9. The $(-1)$-homogeneous irrotational solution for $C=-1/2$ and $b=-1/8$ in (1.5) take the form 
\begin{align}
u_{\textrm{ir}}(x)=\frac{1}{2r}e_r,\qquad p_{\textrm{ir}}(x)=-\frac{1}{8r^{2}}.
\end{align}
See Figure 4 (A) for streamlines.

Rotational ($-1$)-homogeneous solutions with singular profiles can be constructed by replacing $e_r$ by $-e_{r}$ on intervals in $[0,\pi]$. The simplest one forms
\begin{align}
u_{\textrm{sh}}(x)=
\begin{cases}
\ -\displaystyle\frac{1}{2r}e_{r}\quad &0<\phi<\displaystyle\frac{\pi}{2},\\
\ \displaystyle\frac{1}{2r}e_{r}\quad &\displaystyle\frac{\pi}{2}<\phi<\pi.
\end{cases}
\qquad p_{\textrm{sh}}(x)=-\frac{1}{8r^{2}}.
\end{align}\\
The solution (1.9) is a weak solution to the 2D Euler equations in $\mathbb{R}^{2}\backslash \{0\}$ satisfying the $x_2$-symmetry (1.4) with the discontinuous velocity field $u_{\textrm{sh}}\in C^{1}(\mathbb{R}^{2}\backslash \{x_1=0\})$ producing the vortex sheet $\omega_{\textrm{sh}}=r^{-2}\delta_{\pi/2}(\phi)$, cf. \cite[Chapter 9]{MaB}. See Figure 4 (B) for streamlines. Self-similar vortex sheet solutions are relevant to non-uniqueness and long-time behavior of solutions to the 2D Euler equations; see Elling and Gnann \cite{EG}, Cie\'{s}lak et al. \cite{Cieslak1}, \cite{Cieslak2}, \cite{Cieslak3}, Jeong and Said \cite{Jeong23}, and Shao et al. \cite{SWZ}. We show that ($-\alpha$)-homogeneous solutions to the Boussinesq equations desingularize the ($-1$)-homogeneous vortex sheet solution (1.9); see Figure 4 (C).

\begin{thm}[Desingularization as $\alpha\to 1+0$]
There exist ($-\alpha$)-homogeneous solutions $(u_\alpha,p_\alpha,\rho_{\alpha})\in C^{2}(\mathbb{R}^{2}\backslash \{0\})$ to (1.1) and (1.4) for $1<\alpha<2$ such that $r^{\alpha}u_\alpha\to r u_{\textrm{sh}}$, $r^{2\alpha}p_{\alpha}\to r^{2}p_{\textrm{sh}}$, and $r^{2\alpha+1}\rho_{\alpha}\to 0$ locally uniformly for $\phi \in \mathbb{S}^{1}\backslash \{\pm\pi/2\}$ as $\alpha\to 1+0$. The function $r^{2\alpha}p_{\alpha}$ converges to $r^{2}p_{\textrm{sh}}$ also at $\phi=\pm \pi/2$. 
\end{thm}

This study's ($-\alpha$)-homogeneous solutions provide examples of singular stationary solutions belonging to supercritical spaces for the local well-posedness of the 2D inviscid Boussinesq equations. These examples may reveal interesting behavior of nonstationary solutions in supercritical regimes; see, e.g., works on the 2D Euler equations: loss of regularity in time \cite{Jeong21}, \cite{CMO} and non-uniqueness \cite{Elling}, \cite{Vishik18}, \cite{Vishik18b}, \cite{EH}, \cite{SWZ}.

\begin{figure}
\begin{minipage}[b]{0.45\linewidth}
\hspace{-20pt}
\includegraphics[scale=0.1]{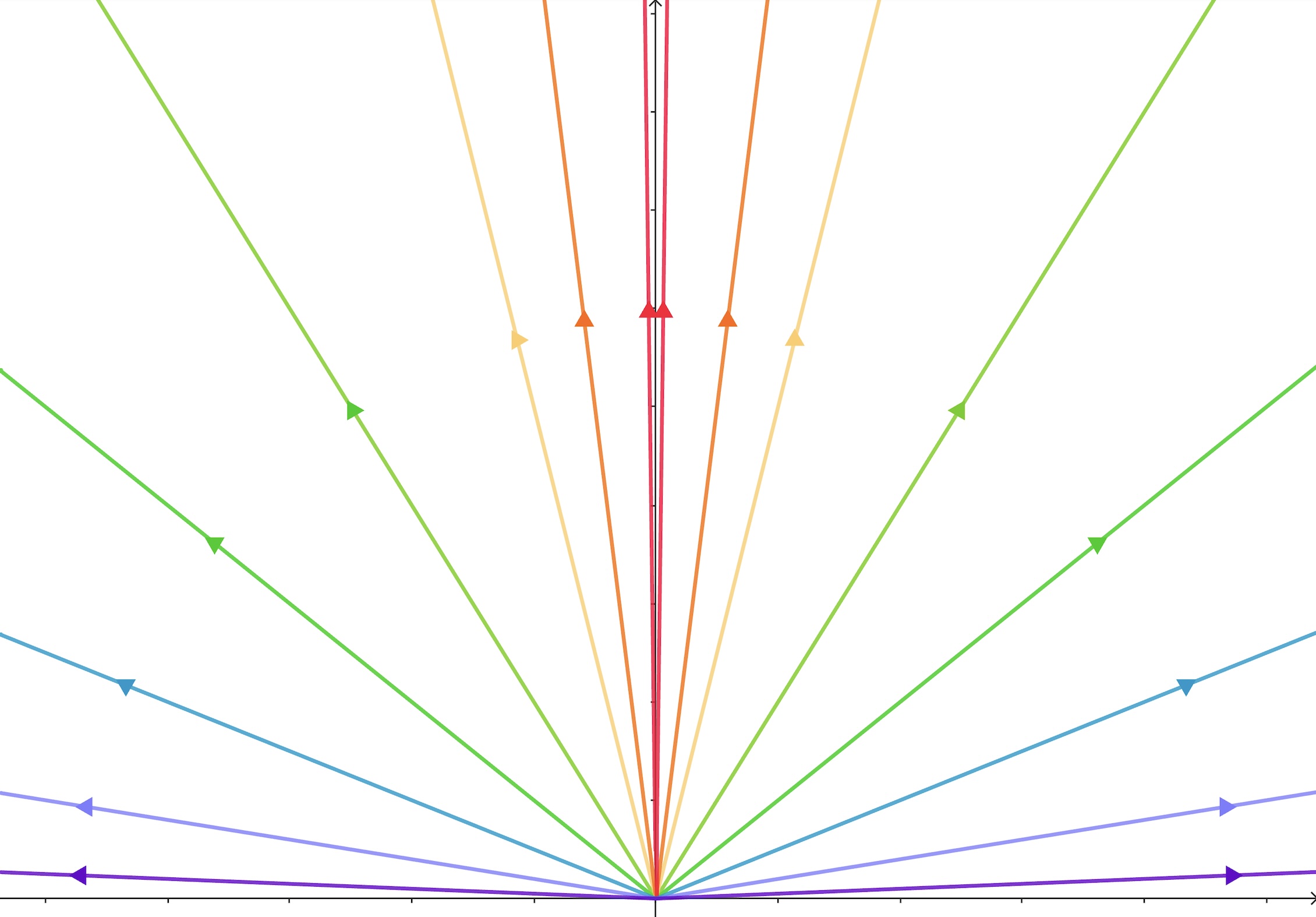}
\subcaption{The ($-1$)-homogeneous irrotational solution in (1.8)}
  \end{minipage}\\
\vspace{20pt}
\begin{minipage}[b]{0.45\linewidth}
\hspace{-20pt}
\includegraphics[scale=0.1]{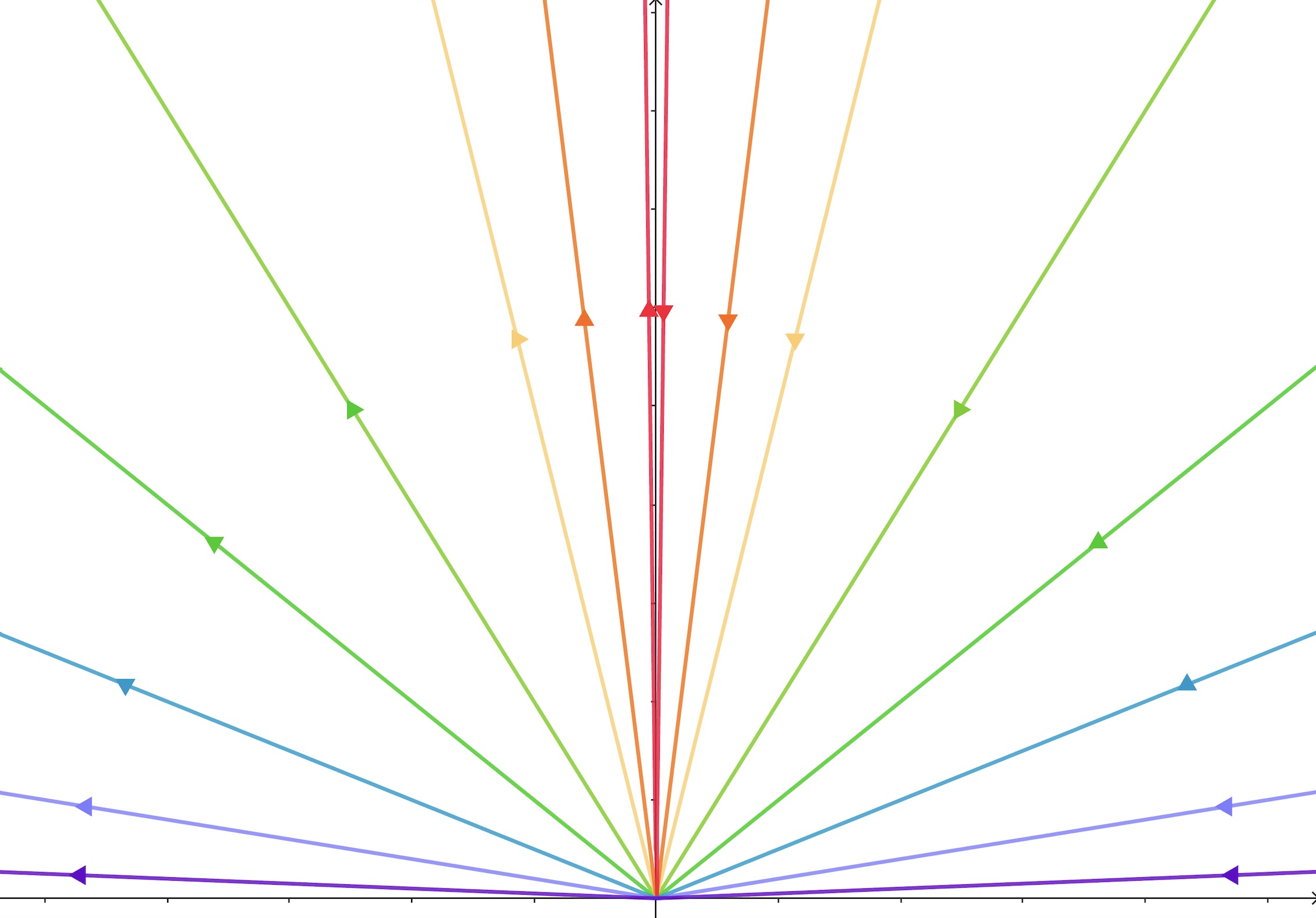}
\subcaption{The ($-1$)-homogeneous vortex sheet solution in (1.9)}
  \end{minipage}\\
\vspace{20pt}
\begin{minipage}[b]{0.45\linewidth}
\hspace{-20pt}
\includegraphics[scale=0.11]{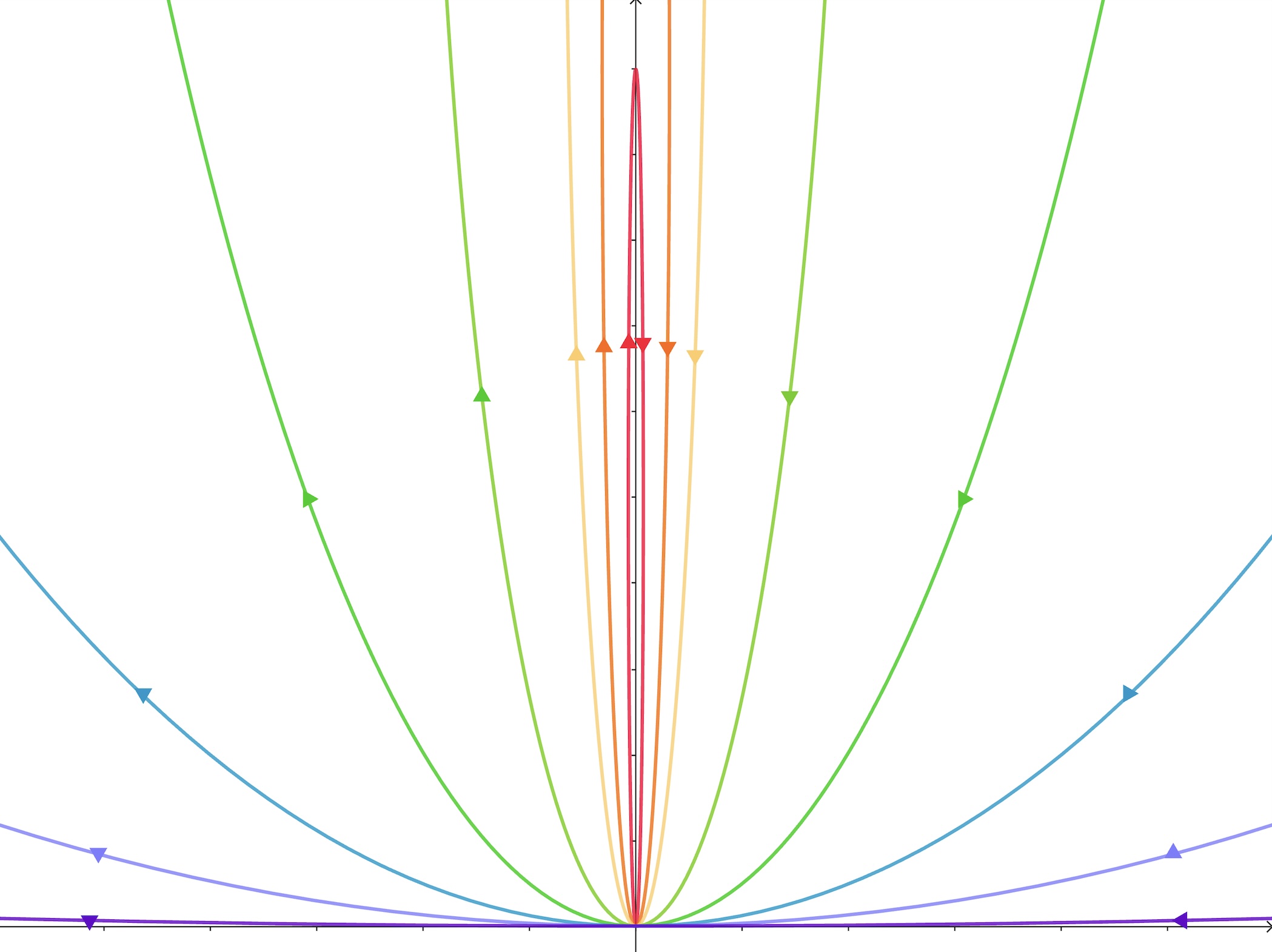}
    \subcaption{The ($-\alpha$)-homogeneous Boussinesq solution for $1<\alpha<2$ in Theorem 1.10}
\end{minipage}\\
\vspace{20pt}
  \caption{Streamlines of ($-\alpha$)-homogeneous solutions}
\end{figure}

\subsection{Relevant works on stratified flows: blow-up and stability}

\subsubsection{The singularity formation}

The singularity formation of the 2D inviscid Boussinesq equations has attracted attention during the last decade. Chae \cite[Theorem 2.4]{Chae07} investigated the non-existence of backward self-similar solutions to the 2D inviscid Boussinesq equations. Chae and Wolf \cite{CW19} and Chae and Constantin \cite{CC21} established non-blow-up criteria that exclude self-similar singularities for solutions to the inviscid  Boussinesq equations. 

Elgindi and Jeong \cite{EJ20} constructed scale-invariant blow-up solutions ($\alpha=-1$) to the inviscid Boussinesq equations in a corner domain of angle $\pi/2$. It is an open question \cite[Problem 8]{DE} whether scale-invariant solutions with vanishing vorticity and density gradient on the boundary exist. Elgindi and Pasqualotto \cite{Elgindi23}, \cite{Elgindi23b} constructed blow-up solutions in the whole plane for some $C^{1,\gamma}$ initial data smooth except for the origin. C\'{o}rdoba et al. \cite{CMZ23} and C\'{o}rdoba and Mart\'{i}nez-Zoroa \cite{CM23} discussed non-self-similar and non-axisymmetric blow-up solutions to the Euler equations; see also Chen \cite{Chen23}. 

\subsubsection{Stability of static equilibrium}

The static equilibrium $p=p(x_2)$ and $\rho(x_2)=p'(x_2)$ is a physically important steady-state of the inviscid Boussinesq equations describing a stable density of lighter fluid above and heavier fluid below. Elgindi and Widmayer \cite{EW} and Jurja and Widmayer \cite{JW} show the local in-time stability of the $1$-homogeneous static equilibrium 
\begin{align}
u_s=0,\quad p_s=\frac{1}{2}x_{2}^{2},\quad \rho_s=-x_2.
\end{align}
Elgindi \cite{Elgindi17}, Castro et al. \cite{CCL}, and Bianchini et al. \cite{Bianchini} demonstrated the asymptotic stability of the $1$-homogeneous static equilibrium in the IPM equation. See also Castro et al. \cite{CCL19b} and Tao et al. \cite{Wu20} for stability in other systems. Kiselev and Yao \cite{KY23} and Kiselev et al. \cite{KPY22} demonstrated density gradient growth (small-scale formations) both in the IPM equations and in the 2D inviscid Boussinesq equations. The work \cite{KY23} showed the instability of the $1$-homogeneous static equilibrium in the IPM equations with low regular norms, cf. \cite{CCL}. See also Bianchini et al. \cite{BCM} for the ill-posedness result to the IPM equations in $H^{2}$, cf. \cite{CGO}.  

Masmoudi et al. \cite{Masmoudi22} demonstrated the asymptotic stability of the static equilibria with Couette flow in the viscous Boussinesq equations without diffusivity. Bedrossian et al. \cite{Bedrossian23} showed the stability of static equilibria with Couette flow in the inviscid Boussinesq equations and vorticity and density gradient growth during the existence of smooth solutions (shear-buoyancy instability).

\subsection{Ideas of the proof}

We demonstrate the main theorems of this paper in the following order:

\begin{itemize}
\item[(I)] The non-existence of ($-\alpha$)-homogeneous solutions with regular profiles
\item[(II)] The existence of ($-\alpha$)-homogeneous solutions with regular profiles
\item[(III)] The existence of ($-\alpha$)-homogeneous solutions with singular profiles
\end{itemize}

The approaches to (I) and (II) are based on works on homogeneous solutions to the 3D Euler equations \cite{LuoShvydkoy}, \cite{Shv}, and \cite{Abe11}; see \cite[6.1.2]{BV22}. We consider the stationary analog between the axisymmetric Euler equations and the 2D inviscid Boussinesq equations. The approach to (I) is based on the homogeneous Boussinesq equations on the semi-circle. The approach to (II) is based on the Dubreil-Jacotin--Long equation for the stream function \cite{Dub}, \cite{Long53}, cf. \cite{Yih58}, \cite[Chapter 3]{Yih65}, \cite{CLV}, \cite{CDG21b}. We consider the 1D homogeneous Dubreil-Jacotin--Long equation for the angular variable $\phi\in [0,\pi]$ with \textit{positive} nonlinear powers and construct ($-\alpha$)-homogeneous solutions with regular profiles by applying three minimax theorems according to $\alpha\in \mathbb{R}$ values.

The completely new part of this study is (III) the existence of homogeneous solutions with singular profiles. We construct solutions to the 1D homogeneous Dubreil-Jacotin--Long equation with \textit{negative} nonlinear powers. This equation is non-autonomous and not integrable, unlike the case of the 2D Euler equations \cite{EH}. We construct solutions to the non-autonomous equation based on the maximum principle \cite{delPino}. 

We outline the approaches to (I)--(III).

\subsubsection{The non-existence of homogeneous solutions with regular profiles}

The $(-\alpha)$-homogeneous solutions $u=r^{-\alpha}(a(\phi)e_{\phi}+f(\phi)e_{r})$, $p=r^{-2\alpha}p(\phi)$, and $\rho=r^{-2\alpha-1}\rho(\phi)$ to  (1.1) and (1.4) satisfy the equations on the semi-circle: 
\begin{equation}
\begin{aligned}
a \omega+2\alpha \Pi &=(2\alpha+1)(\sin \phi)\rho,\\
f\omega+\Pi'&=(\sin\phi)\rho',\\
(1-\alpha)f+a'&=0, \\ 
a\rho'&=(2\alpha+1)f\rho,\quad \phi\in (0,\pi),
\end{aligned}
\end{equation}
with the vorticity $\omega=(1-\alpha)a-f'$ and the Bernoulli function $\Pi=p+|a|^{2}/2+|f|^{2}/2+(\sin\phi)\rho$. The odd symmetry of $u^{2}$ and $\rho$ in (1.4) imply the boundary conditions
\begin{align}
a(0)=a(\pi)=0,\quad \rho(0)=\rho(\pi)=0.   
\end{align}
We use the first integral condition $u\cdot \nabla \Pi=0$ for homogeneous solutions, 
\begin{align}
a\Pi'=2\alpha f\Pi.
\end{align}
The system (1.11) for $(a,f,p,\rho)$ has simpler forms in the following cases: 
\begin{itemize}
\item[(i)] $\omega\equiv 0$: a couple of systems for $(a,f)$ and $(\Pi,\rho)$.  
\item[(ii)] $\alpha=1$: a system for $(f,p,\rho)$
\item[(iii)] $\rho\equiv 0$: a system for $(a,f)$ with constant $p$
\item[(iv)] $\Pi\equiv \textrm{const.}$: a system for $(a,f,\rho)$
\end{itemize}

In (i), systems for $(a,f)$ and $(\Pi,\rho)$ are integrable. In (ii), no solutions exist for $f\equiv 0$, and all solutions for $f\nequiv 0$ are irrotational. In (iii) for $-1\leq \alpha<1$ and (iv) for $-2\leq \alpha<1$, we show that all rotational ($-\alpha$)-homogeneous solutions are irrotational and with zero-densities, i.e., $\omega\equiv 0$ and $\rho\equiv 0$. They demonstrate Theorem 1.9, the case $\alpha=1$ of Theorem 1.2 (i), (i) and (ii) of Theorems 1.5 and 1.7. 

In (iii), we use the first integral condition $u\cdot \nabla \omega=0$ for the  homogeneous vorticity,
\begin{align}
a\omega'=(\alpha+1)f\omega.
\end{align}
The vorticity is an odd function and satisfies the boundary condition 
\begin{align}
\omega(0)=\omega(\pi)=0.
\end{align}
We integrate the equations $(1.11)_3$ and (1.13) for $0\leq \alpha<1$ and the equations $(1.11)_3$ and (1.14) for $-1<\alpha<0$ by using the boundary condition $(1.12)_1$ and conclude that $\omega\equiv 0$. The case $\alpha=-1$ uses the condition (1.15). 

In (iv), we use the first integral condition $u\cdot \nabla \Omega=0$ for the homogeneous normalized vorticity $\Omega=\omega/x_2$,  
\begin{align}
a\Omega'=(\alpha+2)f\Omega.
\end{align}
We integrate the equation $(1.11)_3$ and $(1.11)_4$ for $-1/2< \alpha< 1$ and the equation $(1.11)_3$ and (1.16) for $-2<\alpha<-1/2$ by using the boundary condition $(1.12)_1$ and conclude that $\omega\equiv 0$. The cases $\alpha=-1/2$ uses the boundary condition $(1.12)_2$. The case $\alpha=-2$ uses the additional condition $\Omega(0)=\Omega(\pi)=0$. 

In a general case in (i) and (ii) of Theorem 1.2, we integrate the equations $(1.11)_3$ and $(1.11)_4$ for $-1/2< \alpha<1$ by using the boundary condition $(1.12)_1$ and first deduce $\rho\equiv 0$ and conclude that $\omega\equiv 0$ by using (i) and (ii) of Theorem 1.5. The case $\alpha=-1/2$ uses the boundary condition $(1.12)_2$. 

\subsubsection{The existence of homogeneous solutions with regular profiles}

The main contribution of this work is the existence of homogeneous solutions to (1.1) and (1.4). The density $\rho$ and the Bernoulli function $\Pi$ are the first integrals of the velocity field $u=\nabla^{\perp}\psi$ and locally are functions of the stream function $\psi$. We assume that they are global functions of $\psi$ in a half-plane, i.e., $\Pi=\Pi(\psi)$ and $\rho=\rho(\psi)$. Then, the momentum equation $u^{\perp}\omega+\nabla \Pi=x_2\nabla \rho$ for the $\pi/2$ counter-clockwise rotation of the velocity $u^{\perp}=(-u^{2},u^{1})$ implies that the stream function is a solution to the Dubreil-Jacotin--Long equation \cite{Dub}, \cite{Long53}, cf. \cite{Yih58}, \cite[Chapter 3]{Yih65}, \cite{CLV}, \cite{CDG21b}:
\begin{equation}
\begin{aligned}
-\Delta \psi&=-\Pi'(\psi)+x_2\rho'(\psi),\quad x\in \mathbb{R}^{2}_{+},\\
\psi(x_1,0)&=0,\quad x_1\in \mathbb{R}.
\end{aligned}
\end{equation}
Solutions to this elliptic equation provide steady stratified flows for given functions $\rho(\psi)$ and $\Pi(\psi)$. We construct ($-\alpha+1$)-homogeneous solutions 
\begin{align}
\psi=\frac{w(\phi)}{r^{\alpha-1}},
\end{align}
by choosing ($-2\alpha+1$)-homogeneous densities and ($-2\alpha$)-homogeneous Bernoulli functions as follows:
\begin{align}
\Pi=C_1|\psi|^{2+\frac{2}{\alpha-1}},\quad \rho=C_2\psi|\psi|^{1+\frac{3}{\alpha-1}}.
\end{align}
Then, the function $w$ satisfies the semilinear equation:  
\begin{equation}
\begin{aligned}
-w''&=\beta^{2}w+c_1 w|w|^{\frac{2}{\beta}}+c_2 \sin{\phi} |w|^{1+\frac{3}{\beta}},\quad \phi\in (0,\pi),\\
w(0)&=w(\pi)=0,
\end{aligned}
\end{equation}
for the constants 
\begin{align}
\beta=\alpha-1,\quad c_1=-2C_1\left(1+\frac{1}{\beta}\right), \quad c_2=2C_2\left(1+\frac{3}{2\beta}\right). 
\end{align}
We construct solutions to the homogeneous Dubreil-Jacotin--Long equation (1.20) for \textit{positive} nonlinear powers by assuming that the constants $(\beta,c_1,c_2)$ satisfy the conditions:  
\begin{equation}
\begin{aligned}
&\beta\in \mathbb{R}\backslash [-2,0],\quad c_1,c_2\geq 0,\quad c_1\neq 0\ \textrm{or}\ c_2\neq 0,\\
&c_2=0\quad \textrm{for}\ -3\leq \beta<-2.
\end{aligned}
\end{equation}
We construct solutions of (1.20) by seeking critical points of the associated functional
\begin{align}
I[w]=\frac{1}{2}\int_{0}^{\pi}\left(|w'|^{2}-\beta^{2}|w|^{2}\right)\dd \phi
-\int_{0}^{\pi}\left(\frac{c_1\beta}{2(\beta+1)}|w|^{2+\frac{2}{\beta}}+\frac{c_2\beta}{2\beta+3}\sin\phi w|w|^{1+\frac{3}{\beta}}  \right)\dd \phi.
\end{align} 
The principal eigenvalue $\mu_1=1-\beta^{2}$ of the operator $-\partial_{\phi}^{2}-\beta^{2}$ on $L^{2}(0,\pi)$ is positive for $0<\beta<1$ and non-positive for $1\leq \beta$ and $\beta<-2$. Moreover, the powers of the nonlinear terms in $(1.20)_1$ are superlinear for $\beta>0$ and sublinear for $\beta<-2$. We thus divide the existence problem to (1.20) into the following three cases and apply different minimax theorems according to the values of $\beta$:
\begin{itemize}
\item[(i)] $0<\beta<1$: Mountain pass theorem 
\item[(ii)] $1\leq \beta$: Linking theorem
\item[(iii)] $\beta<-2$: Saddle point theorem
\end{itemize}

The main steps of applying those minimax theorems are to verify the Palais--Smale conditions for the functional $I$ and to obtain inequalities of the form $\inf_{N}I>\max_{M_0}I$ for suitable subsets $N$ and $M_0$ in $H^{1}_{0}(0,\pi)$. We show the Palais--Smale conditions for the superlinear case $\beta>0$ and for the sublinear case $\beta<-2$ respectively, and choose suitable subsets $N$ and $M_0$ in $H^{1}_{0}(0,\pi)$ according to the values of $\beta$ to obtain the inequality $\inf_{N}I>\max_{M_0}I$. 

For $0<\beta<1$, a mountain pass theorem applied to a modified functional to (1.23) and a strong maximum principle imply the existence of positive and even symmetric solutions to (1.20) for $\phi=\pi/2$ in $(0,\pi)$.

The existence of solutions to (1.20) for $(\beta,c_1,c_2)$ satisfying (1.22) implies (iii) and (iv) of Theorems 1.2, 1.5, and 1.7.

\subsubsection{The existence of homogeneous solutions with singular profiles}

The main results in this work (Theorems 1.3, 1.6, and 1.8) are based on the existence of solutions to the homogeneous Dubreil-Jacotin--Long equation (1.20) with \textit{negative} nonlinear powers for $(\beta,c_1,c_2)$ satisfying 
\begin{equation}
\begin{aligned}
&-2<\beta<0,\quad c_1, c_2\geq 0,\quad c_1\neq 0\ \textrm{or}\ c_2\neq 0,\\
&c_1=0\quad \textrm{for}\ \beta=-1,\\
&c_2=0\quad \textrm{for}\ \beta=-3/2.
\end{aligned}
\end{equation}
We construct ($-\alpha$)-homogeneous solutions to (1.1) satisfying the $x_1$-symmetry (1.3) by constructing positive solutions to the following problem for $s=-1-2/\beta\in (0,\infty)$ and $s'=-1-3/\beta\in (1/2,\infty)$:  
\begin{equation}
\begin{aligned}
-w''&=\beta^{2}w+c_1 w^{-s}+c_2 \sin{\phi} w^{-s'},\\
w&>0,\quad \phi\in \left(0,\pi/2\right),\\
w(0)&=w(\pi/2)=0.
\end{aligned}
\end{equation}
Solutions to (1.25) provide ($-\alpha$)-homogeneous solutions to (1.1), (1.3), and (1.4) with singular profiles for $-1<\alpha<1$ in Theorems 1.3, 1.6, and 1.8.  

The problem (1.25) is a singular elliptic problem; see, e.g., \cite{Oliva}. We construct solutions to (1.25) based on the argument of del Pino \cite{delPino} using the maximum principle. The eigenvalues of the operator $-L_{\beta}=-\partial_{\phi}^{2}-\beta^{2}$ on $(0,\pi/2)$ are $\mu_{k}=4k^{2}-\beta^{2}$ for $k\in \mathbb{N}$. For $-2<\beta<0$, the principal eigenvalue $\mu_1=4-\beta^{2}$ is positive and the associated eigenfunction is $\sin2\phi$. We show that the maximum principle holds for the operator $-L_{\beta}$ and $-2<\beta<0$; namely, that for a non-negative coefficient $c(\phi)\geq 0$, the functions $w$ satisfying
\begin{align*}
-L_{\beta}w+c(\phi)w\geq 0,\quad \phi\in (0,\pi/2),\quad w(0)=w(\pi/2)=0,
\end{align*}
are positive in $(0,\pi/2)$. Using this maximum principle, we show the a priori estimate for solutions to (1.25): 
\begin{align}
a\sin2\phi\leq w\leq b(\sin2\phi)^{\sigma},\quad \phi\in (0,\pi/2),
\end{align}
for arbitrary $\beta^{2}/4<\sigma\leq \min\{1,-2\beta/3\}$ ($\sigma\neq 1$) for constants $a=a(\beta)$ and $b=b(\beta)$, independent of $w$.

We construct solutions to (1.25) by solving the problem
\begin{equation}
\begin{aligned}
-L_{\beta}w&=g_{\varepsilon}(\phi,w),\quad \phi\in (0,\pi/2), \\
w(0)&=w(\pi/2)=0,
\end{aligned}
\end{equation}
for the bounded function 
\begin{align}
g_{\varepsilon}(\phi,t)=c_1(t_++\varepsilon)^{-s}+c_2\sin\phi (t_++\varepsilon)^{-s'},\quad t_+=\max\{t,0\},\quad \varepsilon>0. 
\end{align}
Since the principal eigenvalue of $-L_{\beta}$ is positive for $-2<\beta<0$, solutions of (1.27) are obtained by minimizing the following functional on $H^{1}_{0}(0,\pi/2)$:
\begin{align}
I[w]=\frac{1}{2}\int_{0}^{\pi/2}\left(|w'|^{2}-\beta^{2}|w|^{2} \right)\dd \phi
-\int_{0}^{\pi/2}G(\phi,w)\dd \phi,\quad G(\phi,w)=\int_{0}^{w}g_{\varepsilon}(\phi,s)\dd s.
\end{align}
Solutions of (1.27) are positive by a maximum principle and satisfy the a priori estimate (1.26). Taking $\varepsilon\to0$  yields solutions to (1.25).

\subsection{Organization of the paper}
Section 2 demonstrates the rigidity of $(-\alpha)$-homogeneous solutions (Theorem 1.9 and (i) and (ii) of Theorems 1.2, 1.5, and 1.7). Section 3 shows the existence of $(-\alpha)$-homogeneous solutions with regular profiles ((iii) and (iv) of Theorems 1.2, 1.5, and 1.7). Section 4 shows the existence of $(-\alpha)$-homogeneous solutions with singular profiles (Theorems 1.3, 1.6, and 1.8). Section 5 demonstrates the desingularization of the ($-1$)-homogeneous vortex sheet solution (Theorem 1.10). Appendix A reviews the existence of backward self-similar solutions to quasi-linear fluid equations.

\subsection{Acknowledgements}
While preparing this manuscript, the authors learned the work of Elgindi and Huang \cite{EH}, which was essential for the main result of this study (Theorem 1.3). The authors thank Tarek Elgindi for informing us of this work. KA is supported by the JSPS through the Grant in Aid for Scientific Research (C) 24K06800 and MEXT Promotion of Distinctive Joint Research Center Program JPMXP0619217849. DG is supported by a startup grant from Brooklyn College and NSF grant DMS-2406852. IJ is supported by the National Research Foundation of Korea grant No. 2022R1C1C1011051.

\section{The non-existence of homogeneous solutions with regular profiles}

We derive the homogeneous Boussinesq system on the semi-circle (1.11) and first investigate the system (1.11) for particular classes: (i) $\omega\equiv 0$, (ii) $\alpha=1$, (iii) $\rho\equiv 0$, and (iv) $\Pi\equiv \textrm{const.}$ We then demonstrate the non-existence of general rotational ($-\alpha$)-homogeneous solutions using the result in case (iii). 

\subsection{The homogeneous Boussinesq system}
Stationary solutions to the Boussinesq equations (1.1) and (1.4) satisfy the following equations in a half-plane:
\begin{equation}
\begin{aligned}
\omega u^{\perp}+\nabla \Pi&=x_2 \nabla \rho, \\
\nabla \cdot u&=0,\\
u\cdot \nabla \rho&=0, \quad x\in \mathbb{R}^{2}_{+},\\
u^{2}(x_1,0)&=0,\quad x_1\in \mathbb{R},
\end{aligned}
\end{equation}
for the vorticity $\omega=\nabla^{\perp}\cdot u$ and the Bernoulli function $\Pi=p+|u|^{2}/2+x_2\rho$. The density is an odd function and vanishes on the boundary
\begin{align}
\rho(x_1,0)=0,\quad x_1\in \mathbb{R}.
\end{align}
Multiplying $u$ by the momentum equation $(2.1)_1$ implies that 
\begin{align}
u\cdot\nabla \Pi=0.   
\end{align}
If $\rho\equiv 0$, taking the rotation to (2.1$)_1$ implies that 
\begin{align}
u\cdot \nabla \omega=0.  
\end{align}
If $\Pi\equiv 0$, taking the rotation to $\Omega u^{\perp}=\nabla \rho$ for $\Omega=\omega/x_2$ implies that 
\begin{align}
u\cdot \nabla \Omega=0.   
\end{align}

We consider $(-\alpha)$-homogeneous vector fields 
\begin{align}
u(x)=\frac{1}{r^{\alpha}}(v+f(\phi)e_{r} ),\quad v=a(\phi)e_{\phi}.  
\end{align}
The vorticity of $(-\alpha)$-homogeneous vector fields are $(-\alpha-1)$-homogeneous and expressed as 
\begin{align}
\omega(x)=\frac{1}{r^{\alpha+1}}\left( (1-\alpha)a(\phi)-f'(\phi) \right).
\end{align}
We consider ($-2\alpha$)-homogeneous functions $p$ and ($-2\alpha-1$)-homogeneous functions $\rho$. We identify them with functions on the circle $\mathbb{S}^{1}$ and denote by $p(x)=r^{-2\alpha}p(\phi)$ and $\rho(x)=r^{-2\alpha-1}\rho(\phi)$. We say that $(u,p,\rho)$ is ($-\alpha$)-homogeneous if $u$, $p$, and $\rho$ are ($-\alpha$)-homogeneous, ($-2\alpha$)-homogeneous, and ($-2\alpha-1$)-homogeneous, respectively. For ($-\alpha$)-homogeneous $(u,p,\rho)\in C^{1}(\mathbb{R}^{2}\backslash \{0\})$, the functions $a,f, p$, and $\rho$ are continuously differentiable on the circle, i.e., $(a,f,p,\rho)\in C^{1}(\mathbb{S}^{1})$. Under the $x_2$-symmetry (1.4), they are symmetric for $\theta\in (-\pi,\pi)$ and identified with functions on $(0,\pi)$. The trace conditions $(2.1)_4$ and (2.2) then yield
\begin{align}
a(0)=a(\pi)=0,\quad \rho(0)=\rho(\pi)=0.
\end{align}
The Bernoulli function $\Pi$ and the normalized vorticity $\Omega$ of ($-\alpha$)-homogeneous $(u,p,\rho)$ are ($-2\alpha$)-homogeneous and ($-\alpha-2$)-homogeneous, respectively. By the polar coordinates expression of the gradient $\nabla=e_r\partial_r+r^{-1}e_{\phi}\partial_{\phi}$, we derive from the equations $(2.1)_1$, $(2.1)_2$, ad $(2.1)_4$ that $(-\alpha)$-homogeneous solutions $(u,p,\rho)$ to (1.1) and (1.4) satisfy 
\begin{equation}
\begin{aligned}
a\omega+2\alpha \Pi&=(2\alpha+1)(\sin\phi) \rho,\\
f\omega+\Pi'&=(\sin\phi) \rho',\\
(1-\alpha)f+a'&=0,\quad \phi\in (0,\pi),\\
a(0)=a(\pi)&=0. 
\end{aligned}
\end{equation}
The equations $(2.1)_3$ and (2.3) for $(-\alpha)$-homogeneous solutions are expressed as 
\begin{align}
a\rho'=(2\alpha+1)f\rho,\quad a\Pi'=2\alpha f \Pi.   
\end{align}
The equations (2.9) are a system for $(a,f,p,\rho)$. In the subsequent sections, we investigate particular classes of solutions.

\subsection{Irrotational solutions}
We assume that $\omega\equiv 0$ in (2.9). Then, $(a,f)$ and $(\Pi,\rho)$ satisfy: 
\begin{equation}
\begin{aligned}
(1-\alpha)a-f'&=0,\\
(1-\alpha)f+a'&=0,\quad \phi\in (0,\pi),\\
a(0)=a(\pi)&=0,
\end{aligned}
\end{equation}
\begin{equation}
\begin{aligned}
2\alpha \Pi&=(2\alpha+1)(\sin\phi)\rho,\\
\Pi'&=(\sin\phi)\rho',\quad \phi\in (0,\pi),\\
\rho(0)&=\rho(\pi)=0.
\end{aligned}
\end{equation}

\vspace{3pt}

\begin{proof}[Proof of Theorem 1.9]
By eliminating the function $a$ from $(2.11)$, 
\begin{align*}
-f''&=(\alpha-1)^{2}f,\quad \phi\in (0,\pi),\\
f'(0)&=f'(\pi)=0.
\end{align*}
Solutions to this equation are $f=C\cos((\alpha-1)\phi )$ for $C\neq 0$ and $\alpha\in \mathbb{Z}$ or $C=0$ and $\alpha\in \mathbb{R}$. Thus $u\nequiv 0$ is expressed as (1.5) with some constants $C\neq 0$ and $\alpha\in \mathbb{Z}$ or $u\equiv 0$ and $\alpha\in \mathbb{R}$. 

By differentiating the equation $(2.12)_1$ and combining it with $(2.12)_2$, 
\begin{align*}
\rho'+(2\alpha+1)(\cot \phi)\rho=0.
\end{align*}
By multiplying $(\sin\phi)^{2\alpha+1}$ by the above equation and integrating, 
\begin{align*}
\rho(\sin\phi)^{2\alpha+1}=C_0,
\end{align*} 
for some constant $C_0$. By the continuity of the density on $\mathbb{S}^{1}$ and the trace condition $(2.12)_3$, $\rho\nequiv 0$ and $\alpha<-1/2$ or $\rho\equiv 0$ and $\alpha\in \mathbb{R}$.

Suppose that $u\nequiv 0$. Then, $u$ is given by (1.5) for $\alpha\in \mathbb{Z}$. If $\rho\equiv 0$, $p=r^{-2\alpha}b$ for $b\in \mathbb{R}$ satisfying $\alpha (2b+C^{2})=0$ by $(2.12)_2$ and $(2.12)_1$. If $\rho\nequiv 0$, $-\alpha\in \mathbb{N}$.

Suppose that $u\equiv 0$. By (2.12$)_1$ and (2.12$)_2$,
\begin{align*}
2\alpha p&=(\sin\phi)\rho,\\
p'+\cos\phi\rho&=0, \quad \phi\in (0,\pi).
\end{align*}
Since $\rho\in C^{1}(\mathbb{S}^{1})$, we have $p\in C^{2}(\mathbb{S}^{1})$. If $\rho\equiv 0$, $p(\phi)=\textrm{const.}$ and $\alpha p=0$. Thus $p(x)=p(\phi)/r^{\alpha}$ is constant in $\mathbb{R}^{2}$. If $\rho\nequiv 0$, 
\begin{align*}
\rho(\phi)=\frac{C_0}{(\sin\phi)^{2\alpha+1}},\quad p''(\phi)=C_0\frac{2\alpha \cos^{2}\phi+1 }{(\sin\phi)^{2\alpha+2}}.
\end{align*}
Thus, $\alpha<-1$ by $p\in C^{2}(\mathbb{S}^{1})$.
\end{proof}

\subsection{($-1$)-homogeneous solutions}
We assume that $\alpha=1$ in (2.9). Then, $a\equiv 0$ and $(f,p,\rho)$ satisfies 
\begin{equation}
\begin{aligned}
2 \Pi&=3(\sin\phi) \rho,\\
-ff'+\Pi'&=(\sin\phi) \rho',\quad \phi\in (0,\pi).\\
\end{aligned}
\end{equation}
The first integral conditions (2.10) form 
\begin{align}
f\rho=0,\quad f \Pi=0.   
\end{align}


\begin{lem}[$\alpha=1$]
There exist no ($-1$)-homogeneous solutions $(u,p,\rho)\in C^{1}(\mathbb{R}^{2}\backslash \{0\})$ to (1.1) and (1.4) other than the irrotational solution for $\alpha=1$ (i.e., $u=r^{-1}Ce_{r}$, $p=r^{-2}b$, and $\rho=0$ satisfying $C^{2}+2b=0$).
\end{lem}

\begin{proof}
Suppose that $f\equiv 0$. Then, by (2.13), 
\begin{align*}
\rho (\sin\phi)^{3}=C_0,
\end{align*}
for some constant $C_0$. By the boundedness of $\rho$ on $\mathbb{S}^{1}$, $\rho\equiv 0$. By $(2.13)_1$, $\Pi\equiv 0$ and $p\equiv 0$. Thus, no solutions exist.

Suppose that $f\nequiv 0$. Then there exists an interval $J\subset (0,\pi)$ such that $f\neq 0$ on $J$. The functions $\Pi$ and $\rho$ vanish on $J$ by (2.14) and hence $f=C$ for some constant $C\neq 0$ by $(2.13)_2$. The interval $J$ is extendable to $(0,\pi)$ and $\Pi\equiv 0$ and $\rho\equiv 0$. Thus, $\omega=(1-\alpha)a-f'\equiv 0$.
\end{proof}

\begin{rem}[Viscous-diffusive Boussinesq]
In contrast to the inviscid Boussinesq equations in Lemma 2.1, $(-1)$-homogeneous solutions can exist for the viscous and diffusive Boussinesq equations:  
\begin{equation*}
\begin{aligned}
\omega u^{\perp}+\nabla \Pi&=x_2 \nabla \rho+\Delta u, \\
\nabla \cdot u&=0,\\
u\cdot \nabla \rho&=\Delta \rho,
\end{aligned}
\end{equation*}
in $\mathbb{R}^{2}\backslash \{0\}$ or in a sectorial domain $D=\{x=(x_1,x_2)\in \mathbb{R}^{2}\backslash \{0\}\ |\ |\textrm{arctan}(x_2/x_1)|<\phi_0\ \}$ for $\phi_0\in (0,\pi)$ subject to the Dirichlet boundary conditions $u=0$ and $\rho=0$ on $\partial D$. The viscous and diffusion terms of $(-1)$-homogeneous vector fields $u$ and ($-3$)-homogeneous functions $\rho$ are as follows: 
\begin{align*}
\Delta u&=\left(\partial_r^{2}+\frac{1}{r}\partial_r+\frac{1}{r^{2}}\partial_{\phi}^{2} \right)\left(\frac{1}{r}(v+fe_{r}) \right)
=\frac{1}{r^{3}} \left( (a''+2f')e_{\phi}+(f''-2a')e_{r}   \right),\\
\Delta \rho&=\left(\partial_r^{2}+\frac{1}{r}\partial_r+\frac{1}{r^{2}}\partial_{\phi}^{2} \right)\left(\frac{1}{r^{3}}\rho \right)
=\frac{1}{r^{5}}(9\rho-\rho'').
\end{align*}
The function $a$ is constant by $(2.9)_3$. The homogeneous viscous-diffusive Boussinesq system follows adding the viscous and diffusive terms into the inviscid equations $(2.9)_1$, $(2.9)_2$ and $(2.10)_1$: 
\begin{equation*}
\begin{aligned}
af'-2p-a^{2}-f^{2}&=-(\sin\phi) \rho+f'',\\
p'+(\cos\phi) \rho&=2f',\\
-3f\rho&=9\rho-\rho''.
\end{aligned}
\end{equation*}
The constant $a$ is zero for the sectorial domain $D$ by the Dirichlet boundary condition.

For the 2D Navier--Stokes equations $\rho\equiv 0$, the above system is reduced to a second-order ODE for $f$ with some constants $a$ and $C$ satisfying $2p=4f+C$ \cite{Jeffery}, \cite{Hamel}. It is known \cite{Rosenhead}, \cite{Fraenkel62}, \cite{Fraenkel63} that for given $\phi_0\in (0,\pi)$ and flux $F\in \mathbb{R}$, there exists a solution $f$ and a constant $C$ satisfying 
\begin{align*}
-4f-C-f^{2}&=f'',\quad \phi\in (-\phi_0,\phi_0),\\
f(\pm\phi_0)&=0,\\
F&=\int_{-\phi_0}^{\phi_0}f \dd \phi.
\end{align*}
The work \cite[Theorem 2]{Sverak2011} showed that all solutions for $\phi\in [-\pi,\pi]$ with $F=0$ are periodic functions $f$ expressed by elliptic functions and $C=\int_{-\pi}^{\pi}f^{2}\dd \phi$, cf. \cite{GW15} for different ($-1$)-homogeneous solutions in $\mathbb{R}^{2}\backslash \{0\}$ representing spirals. We note that Brandolese and Karch \cite{BK23} recently constructed forward self-similar solutions to the 3D Oberbeck-Boussinesq system with Newtonian gravitational field, cf. \cite{JS}. See also Tsai \cite{Tsai24}. 
\end{rem}

\subsection{2D Euler solutions}
We assume that $\rho\equiv 0$ in (2.9). Then, $p$ is constant on $\mathbb{S}^{1}$ by $(2.9)_2$ and $(2.9)_3$ and $(a,f)$ satisfies 
\begin{equation}
\begin{aligned}
a^{2}-af'+2\alpha p+\alpha f^{2}&=0,\\
(1-\alpha)f+a'&=0,\quad \phi\in(0,\pi),\\
a(0)=a(\pi)&=0.
\end{aligned}
\end{equation}
The conditions (2.4) and (2.3) are expressed as 
\begin{align}
a\omega'=(\alpha+1)f \omega, \quad a\Pi'=2\alpha f \Pi.
\end{align}
The vorticity is an odd function under the $x_2$-symmetry (1.4) and satisfies the boundary condition 
\begin{align}
\omega(0)=\omega(\pi)=0.
\end{align}


\begin{lem}
Let $-1\leq \alpha\leq 1$. Let $(u,p)$ be a ($-\alpha$)-homogeneous solution to (1.1) and (1.4) for $\rho\equiv 0$. Assume that\\ 
\noindent
(i) $(u,p)\in C^{1}(\mathbb{R}^{2}\backslash \{0\})$ for $0\leq \alpha\leq 1$, and \\
(ii) $(u,p)\in C^{2}(\mathbb{R}^{2}\backslash \{0\})$ for $-1\leq \alpha< 0$.\\
\noindent
Then, $\omega\equiv 0$.
\end{lem}

\begin{proof}
We consider the following five cases (i)-(v) for $-1\leq \alpha\leq 1$.\\
\noindent
(i) $\alpha=1$. No solutions exist other than the irrotational solution $\omega\equiv 0$ by Lemma 2.1.\\
(ii) $0<\alpha<1$. By $(2.15)_2$ and $(2.16)_2$, 
\begin{align*}
(1-\alpha)a\Pi'+2\alpha a'\Pi=0.
\end{align*}
Suppose that $a\Pi\nequiv 0$. Then there exists an open interval $J\subset (0,\pi)$ such that $a\Pi\neq 0$ on $J$. By dividing the above equation by $a\Pi$ and integrating it, 
\begin{align*}
|\Pi|^{1-\alpha}|a|^{2\alpha}=C_0,
\end{align*}
for some constant $C_0\neq 0$. The condition $0<\alpha<1$ implies that the function $a\Pi$ does not vanish on the boundary of $J$, and we may assume that $J=(0,\pi)$. By the boundary condition $(2.15)_3$, $C_0=0$. This is a contradiction. Thus $a\Pi\equiv 0$. 

Suppose that $\Pi\nequiv 0$. Then there exists $J\subset (0,\pi)$ such that $\Pi\neq 0$ on $J$. Then $a=0$ and $f=0$ by $(2.15)_2$ and $p=0$ and $\Pi=0$ on $J$ by $(2.15)_1$. This a contradiction. Thus $\Pi\equiv 0$. 

By $(2.15)_1$,
\begin{align*}
0=a^{2}-af'+2\alpha p+\alpha f^{2}=a\omega+2\alpha \Pi=a\omega.
\end{align*}
If $\omega\neq 0$ on some $J\subset (0,\pi)$, $a=0$ and $f=0$ by $(2.15)_2$, and $\omega=(1-\alpha)a-f'=0$. This is a contradiction. Thus $\omega\equiv 0$.

\noindent
(iii) $\alpha=0$. By $(2.16)_2$, 
\begin{align*}
a\Pi'=0.
\end{align*}
Suppose that $\Pi'\nequiv 0$. Then, there exists $J\subset (0,\pi)$ such that $\Pi'\neq 0$ on $J$. By $a=0$ and $(2.15)_2$, $f=0$ and $\Pi'=p'+aa'+ff'=0$. This is a contradiction. Thus $\Pi'\equiv 0$. 

By $(2.15)_2$ and $\omega=a-f'$,
\begin{align*}
f\omega=f\omega+\Pi'=f(a-f' )+p'+aa'+ff'=0.
\end{align*}
Suppose that $\omega\nequiv 0$. Then there exists an open interval $J\subset (0,\pi)$ such that $\omega\neq 0$ on $J$. By $f=0$ and $(2.15)_2$, $\omega=a$ is constant on $J$. We may assume $J=(0,\pi)$.  By the boundary condition $(2.15)_3$, $\omega\equiv 0$. This is a contradiction. Thus $\omega\equiv 0$.

\noindent
(iv) $-1<\alpha<0$. By $(2.16)_1$ and $(2.15)_2$, 
\begin{align*}
(1-\alpha)a\omega'+(1+\alpha)a'\omega=0.
\end{align*}
Suppose that $a\omega\nequiv 0$. Then there exists an open interval $J\subset (0,\pi)$ such that $a\omega\neq 0$ on $J$. By dividing the above equation and integrating it, 
\begin{align*}
|\omega|^{1-\alpha}|a|^{1+\alpha}=C_0,
\end{align*}
for some constant $C_0\neq 0$. The condition $-1<\alpha<0$ implies that $a\omega$ does not vanish on the boundary of $J$, and we may assume that $J=(0,\pi)$. By the boundary condition $(2.15)_3$, $C_0=0$. This is a contradiction. Thus $a\omega\equiv 0$. 

Suppose that $\omega\nequiv 0$. Then, there exists $J\subset (0,\pi)$ such that $\omega\neq 0$ on $J$. By $a=0$ and (2.15$)_2$, $f=0$ and $\omega=(1-\alpha)a-f'=0$ on $J$. This is a contradiction. Thus $\omega\equiv 0$.\\
\noindent
(v) $\alpha=-1$. By $(2.16)_1$, 
\begin{align*}
a\omega'=0.
\end{align*}
Suppose that $\omega'\nequiv 0$. Then there exists $J\subset (0,\pi)$ such that $\omega'\neq 0$ on $J$. Then $a=0$ and $f=0$ on $J$ by $(2.15)_2$ and $\omega=(1-\alpha)a-f'=0$. This is a contradiction. Thus $\omega'\equiv 0$. By the boundary condition (2.17), $\omega\equiv 0$. 
\end{proof}

\subsection{Pseudo-Beltrami solutions}
We assume that $\Pi\equiv \textrm{const.}$ in (2.9). Then $\alpha\Pi=0$ by $(-2\alpha)$-homogeneity of $\Pi$, and $(a,f,\rho)$ satisfies 
\begin{equation}
\begin{aligned}
a\omega&=(2\alpha+1)(\sin\phi) \rho,\\
f\omega&=(\sin\phi) \rho',\\
(1-\alpha)f+a'&=0,\quad \phi\in (0,\pi),\\
a(0)=a(\pi)&=0. 
\end{aligned}
\end{equation}
The conditions $(2.1)_3$ and (2.5) are expressed as 
\begin{align}
a\rho'=(2\alpha+1)f\rho,\quad   a\Omega'=(\alpha+2)f\Omega.  
\end{align}
The function $\rho$ satisfies 
\begin{align}
\rho(0)=\rho(\pi)=0. 
\end{align}
The function $\Omega$ is an even function under the $x_2$-symmetry (1.4) and generally does not vanish on the boundary. If the homogeneous $\Omega=\omega/x_2$ vanishes on $x_2=0$, the function defined on a semi-circle $\Omega=\Omega(\phi)$ satisfies the boundary condition
\begin{align}
\Omega(0)=\Omega(\pi)=0.
\end{align}


\begin{lem}
Let $-2\leq \alpha\leq 1$. Let $(u,p,\rho)$ be a ($-\alpha$)-homogeneous solution to (1.1) and (1.4) for $\Pi\equiv \textrm{const}$. Assume that \\
\noindent
(i) $(u,p,\rho)\in C^{1}(\mathbb{R}^{2}\backslash \{0\})$ for $-1/2\leq \alpha\leq 1$, and \\
(ii) $(u,p,\rho)\in C^{2}(\mathbb{R}^{2}\backslash \{0\})$ for $-2\leq \alpha< -1/2$.\\
\noindent
For $\alpha=-2$, assume in addition that $\nabla^{\perp}\cdot u/x_2$ vanishes on $\{x_2=0\}$. Then, $\omega\equiv 0$ and $\rho\equiv 0$.
\end{lem}

\begin{proof}
We consider the following five cases (i)-(v) for $-2\leq \alpha\leq 1$.\\
\noindent
(i) $\alpha=1$. The result $\omega\equiv 0$ and $\rho\equiv 0$ follows from Lemma 2.1.\\
\noindent
(ii) $-1/2<\alpha<1$. By $(2.18)_3$ and $(2.19)_1$, 
\begin{align*}
(1-\alpha)a\rho'+(2\alpha+1)a'\rho=0.
\end{align*}
Suppose that $a\rho\nequiv  0$. Then, there exists an open interval $J\subset (0,\pi)$ such that $a\rho\neq  0$ on $J$. By dividing the above equation by $a\rho$ and integrating it,
\begin{align*}
|\rho|^{1-\alpha}|a|^{2\alpha+1}=C_0,
\end{align*}
for some constant $C_0\neq 0$. The conditions $-1/2<\alpha<1$ and $(2.18)_4$ imply that $C_0=0$. This is a contradiction. Thus $a\rho\equiv 0$.

Suppose that $\rho\nequiv 0$. Then, there exists $J\subset (0,\pi)$ such that $\rho\neq 0$ on $J$. Then, $a=0$ and $\rho=0$ by $(2.18)_1$. This is a contradiction. Thus $\rho\equiv 0$. 

By $(2.18)_1$ and $(2.18)_2$, $a\omega\equiv 0$ and $f\omega\equiv 0$. If $\omega\neq 0$ for some $J\subset (0,\pi)$, $a=0$ and 

\begin{align*}
(f^{2})'=2ff'=-2f \left( (1-\alpha)a-f' \right)=-2f\omega=0.
\end{align*}
Thus $\omega=-f'=0$. This is a contradiction. Thus $\omega\equiv 0$ and $\rho\equiv 0$.\\
\noindent
(iii) $\alpha=-1/2$. By $(2.19)_1$, 
\begin{align*}
a\rho'=0.
\end{align*}
Suppose that $\rho'\nequiv 0$. Then there exists $J\subset (0,\pi)$ such that $\rho'\neq 0$ on $J$. The equations $a\rho'\equiv 0$, $(2.18)_3$ and $(2.18)_2$ imply that $a=0$, $f=0$, and $\rho'=0$. This is a contradiction. Thus $\rho'\equiv 0$.

By the boundary condition (2.20), $\rho\equiv 0$. Then $a\omega=0$ and $f\omega=0$ by $(2.18)_1$ and $(2.18)_2$. Suppose that $\omega\nequiv 0$. Then, there exists $J\subset (0,\pi)$ such that $\omega\neq 0$ on $J$. By $a=f=0$ on $J$, $\omega=(1-\alpha)a-f'=0$ on $J$. This is a contradiction. Thus $\omega\equiv 0$ and $\rho\equiv 0$.\\
\noindent
(iv) $-2<\alpha<-1/2$. By $(2.18)_3$ and (2.19$)_2$, 
\begin{align*}
(1-\alpha)a\Omega'+(\alpha+2)a'\Omega=0.
\end{align*}
Suppose that $a\Omega\nequiv 0$. Then, there exists an open interval $J\subset (0,\pi)$ such that $a\Omega\neq 0$ on $J$. By dividing the above equation by $a\Pi$ and integrating it,
\begin{align*}
|\Omega|^{1-\alpha} |a|^{\alpha+2}=C_0,
\end{align*}
for some constant $C_0\neq 0$. The conditions $-2<\alpha<-1/2$ and $(2.18)_4$ imply that $C_0=0$. This is a contradiction. Thus $a\Omega\equiv 0$.

By $(2.18)_1$ and $(2.18)_2$, $\rho\equiv 0$, $a\omega\equiv 0$ and $f\omega\equiv 0$. Then $\omega\equiv 0$ and $\rho\equiv 0$ in the same way as in (iv).\\
\noindent
(v) $\alpha=-2$. By $(2.19)_2$, 
\begin{align*}
a\Omega'=0.
\end{align*}
Suppose that $\Omega'\nequiv 0$. Then, there exists $J\subset (0,\pi)$ such that $\Omega'\neq 0$ on $J$. By $(2.18)_3$, $a=0$ and $f=0$ on $J$ and $\Omega=\omega/\sin\phi=((1-\alpha)a-f')/\sin\phi=0$ on $J$. This is a contradiction. Thus $\Omega'\equiv 0$. 

By the assumption (2.21) and the equation $(2.18)_1$, $\omega=\sin\phi\Omega\equiv 0$ and $\rho\equiv 0$.
\end{proof}

\subsection{Rotational solutions}
We consider general solutions to the system (2.9) with the first integral conditions (2.10).

\begin{lem}
Let $-1/2\leq \alpha\leq 1$. Let $(u,p,\rho)\in C^{1}(\mathbb{R}^{2}\backslash \{0\})$ be a ($-\alpha$)-homogeneous solution to (1.1) and (1.4). Assume in addition that $(u,p)\in C^{2}(\mathbb{R}^{2}\backslash \{0\})$ for $-1/2\leq \alpha<0$. Then, $\omega\equiv 0$ and $\rho\equiv 0$.
\end{lem}

\begin{proof}
(i) $\alpha=1$. The result $\omega\equiv 0$ and $\rho\equiv 0$ follows from Lemma 2.1.\\
\noindent
(ii) $-1/2<\alpha<1$. By $(2.9)_3$ and $(2.10)_1$,  
\begin{align*}
(1-\alpha)a\rho'+(2\alpha+1)a'\rho=0.
\end{align*}
Suppose that $a \rho\nequiv 0$. Then, there exists an open interval $J\subset (0,\pi)$ such that $a \rho\neq 0$ on $J$. By dividing the above equation and integrating it, 
\begin{align*}
|\rho|^{1-\alpha}|a|^{2\alpha+1}=C_0, 
\end{align*}
for some constant $C_0\neq 0$. The conditions $-1/2<\alpha<1$ and $(2.9)_4$ imply that $C_0=0$. This is a contradiction. Thus $a \rho\equiv 0$. 

Suppose that $\rho\nequiv 0$. Then, there exists an open interval $J\subset (0,\pi)$ such that $\rho\neq 0$ on $J$. By $(2.9)_3$, $a=0$ and $f=0$ and $\omega=(1-\alpha)-f'=0$ on $J$. Thus, $(\Pi,\rho)$ satisfies the equations (2.12) on $J$. By integrating an equation for $\rho$, 
\begin{align*}
\rho (\sin\phi)^{2\alpha+1}=C_0,
\end{align*}
for some constant $C_0\neq 0$. We may assume that $J=(0,\pi)$. The condition $-1/2<\alpha<1$ implies that $C_0=0$ and $\rho=0$. This is a contradiction. Thus, $\rho\equiv 0$ and $\omega\equiv 0$ by Lemma 2.3 (i) and (ii).\\
\noindent
(iii) $\alpha=-1/2$. By $(2.10)_1$, 
\begin{align*}
a\rho' =0.
\end{align*}
Suppose that $\rho'\nequiv 0$. Then there exists $J\subset (0,\pi)$ such that $\rho'\neq 0$. By $(2.9)_3$, $a=0$ and $f=0$ and $\omega=(1-\alpha)-f'= 0$. By $(2.9)_1$ and $(2.9)_2$, $\Pi=0$ and $\rho'=0$. This is a contradiction. Thus $\rho'\equiv 0$.

By the boundary condition $(2.8)_2$ and Lemma 2.3 (ii), $\rho\equiv 0$ and $\omega\equiv 0$.
\end{proof}

\begin{proof}[Proof of (i) and (ii) of Theorems 1.2, 1.5, and 1.7]
The results follow from Lemmas 2.3, 2.4, and 2.5.
\end{proof}


\section{The existence of homogeneous solutions with regular profiles}

We construct solutions to the homogeneous Dubreil-Jacotin--Long equation (1.20) for constants $(\beta,c_1,c_2)$ satisfying the conditions (1.22) based on three minimax theorems outlined in subsection 3.2. The ($-\alpha$)-homogeneous solutions with regular profiles in (iii) and (iv) of Theorems 1.2, 1.5, and 1.7 are obtained by setting the constants $(\alpha, C_1, C_2)$ by (1.21) and the functions $(\psi,\Pi,\rho)$ by (1.18) and (1.19). 

\subsection{The homogeneous Dubreil-Jacotin--Long equation with positive nonlinear powers}

We first derive the Dubreil-Jacotin--Long equation for stream functions in a half-plane. For symmetric solenoidal vector fields (1.4), there exist stream functions such that
\begin{align*}
\psi(x_1,x_2)=-\psi(x_1,-x_2).
\end{align*}
By the momentum equation $(2.1)_1$, the Bernoulli function $\Pi$ and the density $\rho$ are the first integrals of the velocity field $u=\nabla^{\perp}\psi$ and local functions of $\psi$. We assume that they are global functions of $\psi$ in the half-plane, i.e. $\Pi=\Pi(\psi)$, and $\rho=\rho(\psi)$. Substituting them into $(2.1)_1$ implies that the stream function $\psi$ is a solution to the Dubreil-Jacotin--Long equation in a half-plane \cite{Dub}, \cite{Long53}, cf. \cite{Yih58}, \cite[Chapter 3]{Yih65}, \cite{CLV}, \cite{CDG21b}:
\begin{equation}
\begin{aligned}
-\Delta \psi&=-\Pi'(\psi)+x_2 \rho'(\psi),\quad x\in \mathbb{R}^{2}_{+},\\
\psi(x_1,0)&=0,\quad x_1\in \mathbb{R}.
\end{aligned}
\end{equation}
Solutions of (3.1) for given functions $\Pi(\psi)$ and $\rho(\psi)$ provide steady stratified flows (2.1).

We construct ($-\alpha$)-homogeneous solutions by choosing homogeneous functions $\Pi$ and $\rho$. We assume that the stream function $\psi$ is $(-\alpha+1)$-homogeneous and expressed as 
\begin{align}
\psi=\frac{w(\phi)}{r^{\alpha-1}},
\end{align}
with some functions $w(\phi)$ and the polar coordinates $(r,\phi)$ in $\mathbb{R}^{2}_{+}$. We choose the functions $\Pi$ and $\rho$ by 
\begin{equation}
\begin{aligned}
\Pi(\psi)&=C_1|\psi|^{2+\frac{2}{\alpha-1}}=C_1\frac{1}{r^{2\alpha}}|w|^{2+\frac{2}{\alpha-1}},   \\
\rho(\psi)&=C_2\psi |\psi|^{1+\frac{3}{\alpha-1}}=C_2\frac{1}{r^{2\alpha+1}}w|w|^{1+\frac{3}{\alpha-1}},
\end{aligned}
\end{equation}
with constants $C_1$ and $C_2$ so that $\Pi$ is an even function for the $x_2$-variable and ($-2\alpha$)-homogeneous and $\rho$ is an odd function for the $x_2$-variable and ($-2\alpha+1$)-homogeneous. By using the polar coordinates,
\begin{align*}
\Delta \psi &=\left(\partial_r^{2}+\frac{1}{r}\partial_r+\frac{1}{r^{2}}\partial_{\phi}^{2} \right)\left(\frac{w(\phi)}{r^{\alpha-1}}\right)
=\frac{1}{r^{\alpha+1}}\left( (\alpha-1)^{2}w+w''  \right),\\
\Pi'(\psi)&=C_1\left(\frac{2\alpha }{\alpha-1}\right) \frac{1}{r^{\alpha+1}}w|w|^{\frac{2}{\alpha-1}},\\
\rho'(\psi)&=C_2\left(\frac{2\alpha+1 }{\alpha-1}\right)\frac{1}{r^{\alpha+2}}|w|^{1+\frac{3}{\alpha-1}}.
\end{align*}
Substituting them into the equation (3.1) implies that $w$ is a solution to the one-dimensional semi-linear Dirichlet problem:  
\begin{equation}
\begin{aligned}
-w''&=\beta^{2}w+c_1 w|w|^{\frac{2}{\beta}}+c_2 \sin{\phi} |w|^{1+\frac{3}{\beta}},\quad \phi\in (0,\pi),\\
w(0)&=w(\pi)=0,
\end{aligned}
\end{equation}
for the constants 
\begin{align}
\label{betacchoices}
\beta=\alpha-1,\quad c_1=-2C_1 \left(1+\frac{1}{\beta}\right),\quad c_2=2C_2 \left(1+\frac{3}{2\beta}\right).
\end{align}
We assume that the constants ($\beta$, $c_1$, $c_2$) satisfy 
\begin{equation}
\begin{aligned}
&\beta\in \mathbb{R}\backslash [-2,0],\quad c_1,c_2\geq 0,\quad c_1\neq 0\ \textrm{or}\ c_2\neq 0,\\
& c_2=0\quad \textrm{for}\ -3\leq \beta <-2.
\end{aligned}
\end{equation}

\begin{thm}
Let ($\beta$, $c_1$, $c_2$) satisfy (3.6). There exists a solution $w\in C^{2}[0,\pi]\cap C_0[0,\pi]$ to (3.4). For $\beta>0$, $w\in C^{3}[0,\pi]$. For $0<\beta<1$, there exists a positive solution to (3.4)
which is even symmetric with respect to $\phi=\pi/2$ and increasing in $(0,\pi/2)$. 
\end{thm}

\begin{rem}[Non-existence of positive solutions for $1\leq |\beta|$]
No positive solutions exist to (3.4) for $1\leq |\beta|$. The principal eigenvalue and the associated eigenfunction of the operator $-L_{\beta}=-\partial_{\phi}^{2}-\beta^{2}$ are $1-\beta^{2}$ and $e_{1}=\sqrt{2/\pi}\sin \phi$. By multiplying $e_1$ by (3.4),
\begin{align*}
(1-\beta^{2})\int_{0}^{\pi}we_1\dd \phi=\int_{0}^{\pi}\left(c_1 w|w|^{\frac{2}{\beta}}+c_2 \sin{\phi} |w|^{1+\frac{3}{\beta}}\right)e_1\dd \phi.
\end{align*}
Thus, $|\beta|<1$ for positive $w$. 
\end{rem}

\begin{rem}[Green function]
The Green function of the operator $-\partial_{\phi}^{2}$ in $(0,\pi)$ is the function
\begin{align*}
G(\phi,\phi')=\frac{1}{2}(\phi+\phi')-\frac{1}{\pi}\phi\phi'-\frac{1}{2}|\phi-\phi'|.
\end{align*}
Namely, $G(\phi,\phi')$ is a distributional solution to 
\begin{align*}
-\partial_{\phi}^{2}G(\phi,\phi')&=\delta_{\phi'}(\phi),\quad \phi\in (0,\pi),\\
G(0,\phi')&=G(\pi,\phi')=0,
\end{align*}
where the delta function $\delta_{\phi'}(\phi)$ is supported at $\phi'\in (0,\pi)$. The function $\psi_{\textrm{sh}}=G(\phi,\pi/2)=\pi/4-|\phi-\pi/2|/2$ is the stream function of the ($-1$)-homogeneous vortex sheet solution (1.9),
\begin{align*}
u_{\textrm{sh}}=\nabla^{\perp}\psi_{\textrm{sh}}
=\left(e_{\phi}\partial_r-\frac{1}{r}e_{r}\partial_{\phi}\right)G\left(\phi,\frac{\pi}{2}\right)=
\begin{cases}
\ -\displaystyle\frac{1}{2r}e_{r}\quad &0<\phi<\displaystyle\frac{\pi}{2},\\
\ \displaystyle\frac{1}{2r}e_{r}\quad &\displaystyle\frac{\pi}{2}<\phi<\pi.
\end{cases}
\end{align*}
\end{rem}

\subsection{The variational approach}

We outline the variational approach to construct solutions to (3.4). We set the operator $L_{\beta}$, the domain $D$, and the function $g(\phi,w)$ by
\begin{equation}
\begin{aligned}
L_{\beta}&=\partial_{\phi}^{2}+\beta^{2},\\
D&=(0,\pi),\\
g(\phi,w)&=c_1 w|w|^{\frac{2}{\beta}}+c_2 \sin{\phi} |w|^{1+\frac{3}{\beta}},
\end{aligned}
\end{equation}
and recast the problem (3.4) as 
\begin{equation}
\begin{aligned}
-L_{\beta}w&=g(\phi,w)\quad \textrm{in}\ D,\\
w&=0\quad \textrm{on}\ \partial D.
\end{aligned}
\end{equation}

\begin{prop}
The eigenvalues of the operator $-L_{\beta}=-\partial_{\phi}^{2}-\beta^{2}$ on $D=(0,\pi)$ are $\mu_k=k^{2}-\beta^{2}$ for $k\in \mathbb{N}$ and the associated eigenfunctions $e_k=\sqrt{2/\pi}\sin{(k\phi)}$ are orthonormal basis on $L^{2}(D)$. 
\end{prop}

We prepare the bilinear form estimates associated with the operator $-L_{\beta}$ and apply minimax theorems to the associated functional to the problem (3.8) as sketched in the following subsubsections.

\subsubsection{Bilinear form estimates }
The bilinear form associated with the operator $-L_{\beta}$ is of the form
\begin{align}
B(w,\eta)=\int_{D}(w'\eta'-\beta^{2}w\eta)\dd \phi.
\end{align}
We will show that this bilinear form is coercive on $H^{1}_{0}(D)$ for $0<\beta<1$. For $\beta\in \mathbb{R}\backslash [-2,1)$, we take a positive integer $K\geq 1$ such that 
\begin{align}
\mu_1<\mu_2<\cdots<\mu_K=K^{2}-\beta^{2}\leq 0<\mu_{K+1}<\cdots.
\end{align}
We will show that the following direct sum decomposition holds
\begin{equation}
\begin{aligned}
&H^{1}_{0}(D)=Y\oplus Z,\\
&Y=\textrm{span}\ (e_1,e_2,\cdots,e_K),\\
&Z=\left\{z\in H^{1}_{0}(D)\ \middle|\ (z,y)_{L^{2}}=0,\ y\in Y\ \right\},
\end{aligned}
\end{equation}
and that the bilinear form $B(\cdot,\cdot)$ is coercive on $Z$. We will prepare the necessary bilinear form estimates in Lemmas 3.10 and 3.11.

\subsubsection{Regularity of critical points}
The functional associated with the problem (3.8) is as follows:
\begin{equation}
\begin{aligned}
I[w]&=\frac{1}{2}\int_{D} \left(|w'|^{2}-\beta^{2}|w|^{2}\right) \dd \phi-\int_{D}G(\phi,w)\dd \phi,\\
G(\phi,w)&=\int_{0}^{w}g(\phi,s)\dd s=\frac{c_1\beta}{2(\beta+1)}|w|^{2+\frac{2}{\beta}}+\frac{c_2 \beta}{2\beta+3}(\sin\phi) w|w|^{1+\frac{3}{\beta}}.
\end{aligned}
\end{equation}
We will show that the functional $I$ is Fr\'echet differentiable and construct classical solutions to (3.8) by seeking a critical point $w\in H^{1}_{0}(D)$ of $I$ in the sense that 
\begin{align}
\langle I'[w],\eta\rangle=0,\quad \eta\in H^{1}_{0}(D).
\end{align}

\subsubsection{The three minimax theorems}
We find critical points of the functional $I$ by applying the following three minimax theorems according to the value of $\beta\in \mathbb{R}\backslash [-2,0]$:\\
 
\noindent
(i) $0<\beta<1$: Mountain pass theorem,\\
(ii) $1\leq \beta$: Linking theorem,\\
(iii) $\beta<-2$: Saddle point theorem. \\

The three divisions are due to the sign of the principal eigenvalue $\mu_1$ of the linear operator $-L_{\beta}$ and the nonlinear powers in $g(\phi,w)$: the superlinear case $\beta> 0$ and the sublinear case $\beta<-2$.

\subsubsection{The Palais--Smale condition}
The main steps in applying those minimax theorems are the Palais--Smale condition and functional inequalities on subsets. A Palais--Smale sequence $\{w\}\subset H^{1}_{0}(D)$ at level $c\in \mathbb{R}$ is a sequence satisfying the conditions
\begin{align*}
I[w_n]&\to c,\qquad  \\
I'[w_n]&\to 0\quad \textrm{in}\ H^{-1}(D)=H^{1}_{0}(D)^{*}.
\end{align*}
The functional $I$ satisfies the (PS$)_c$ condition if any Palais-Smale sequence at level $c\in \mathbb{R}$ has a convergent subsequence in $H^{1}_{0}(D)$. We will show that any Palais--Smale sequence at level $c$ is a bounded sequence both for the superlinear case $\beta\!>\!0$ and the sublinear case $\beta\!<\!-2$ and has a convergent subsequence in $H^{1}_{0}(D)$ (Lemma 3.16).

\subsubsection{Functional inequalities on subsets}
We will show that the functional $I$ satisfies the inequality of the form 
\begin{align}
\inf_{N}I >\max_{M_0} I,
\end{align}
by choosing suitable subsets $N$ and $M_0$ in $H^{1}_{0}(D)$ according to the value of $\beta$ in Lemma 3.17 and applying the three minimax theorems. \\

\noindent
(i) $0<\beta<1$. The functional $I$ is positive for small $w\in H^{1}_{0}(D)$ by $\mu_1>0$ and negative for large $w\in H^{1}_{0}(D)$ by the superlinear nonlinearity $\beta>0$. We will show the inequality (3.14) for 
\begin{equation}
\begin{aligned}
&N=\left\{w\in H^{1}_{0}(D)\ \middle|\  ||w||_{H^{1}_{0}}=r_0\   \right\},\\
&M_0=\{0,w_0  \},
\end{aligned}
\end{equation}
for some $r_0>0$ and $w_0\in H^{1}_{0}(D)$ such that $||w_0||_{H^{1}_{0}}>r_0$ and apply a mountain pass theorem (Lemma 3.6).

\noindent
(ii) $1\leq \beta$. The functional $I$ is positive for small $w\in Z$ by the direct sum decomposition (3.11) and negative for large $w\in H^{1}_{0}(D)$ and for $w\in Y$. We will show the inequality (3.14) for 
\begin{equation}
\begin{aligned}
&N=\left\{w\in Z\ \middle|\  ||w||_{H^{1}_{0}}=r_0\   \right\},\\
&M_0=\left\{w=y+\lambda z_0\in Y\oplus \mathbb{R}z_0\ \middle|\  \lambda=0\ \textrm{and}\ ||w||_{H^{1}_0}\leq \rho_0,\ \textrm{or}\ \lambda>0\ \textrm{and}\ ||w||_{H^{1}_{0}}=\rho_0\  \right\},
\end{aligned}
\end{equation}
for some $z_0\in Z$ and $\rho_0>r_0=||z_0||_{H^{1}_{0}}$ and apply a linking theorem (Lemma 3.8).

\noindent
(iii) $\beta<-2$. The functional $I$ is bounded from below on $Z$ by the sublinear nonlinearity $\beta<-2$ and tends to be minus infinity for large $w\in Y$. We will show the inequality (3.14) for 
\begin{equation}
\begin{aligned}
&N=Z,\\
&M_0=\left\{y\in Y\ \middle|\  ||y||_{H^{1}_{0}}=\rho_0\ \right\},
\end{aligned}
\end{equation}
for some $\rho_0>0$, apply a saddle point theorem (Lemma 3.9).

\subsubsection{Positive solutions}

We find positive solutions to (3.4) by solutions to the modified problem 
\begin{equation}
\begin{aligned}
-w''&=\beta^{2}w_{+}+c_1 w^{1+\frac{2}{\beta}}_{+}+c_2 \sin{\phi} w^{1+\frac{3}{\beta}}_{+},\quad \phi\in (0,\pi),\\
w(0)&=w(\pi)=0,
\end{aligned}
\end{equation}
where $s_{+}=\max\{s,0\}$. The functional associated with this problem is as follows:
\begin{equation}
\begin{aligned}
\tilde{I}[w]&=\frac{1}{2}\int_{D} \left(|w'|^{2}-\beta^{2}w_{+}^{2}\right)\dd \phi-\int_{D}\tilde{G}(\phi,w)\dd \phi,\\
\tilde{G}(\phi,w)&=\int_{0}^{w}g(\phi,s)\dd s=\frac{c_1\beta}{2(\beta+1)}w_{+}^{2+\frac{2}{\beta}}+\frac{c_2\beta}{2\beta+3}(\sin\phi)w_{+}^{2+\frac{3}{\beta}}.
\end{aligned}
\end{equation}
We find critical points of $\tilde{I}$ for $0<\beta<1$ by the mountain pass theorem by modifying the argument to $I$ and using the pointwise estimate $w_{+}\leq |w|$. The classical solutions $w\in C^{2}[0,\pi]\cap C_{0}[0,\pi]$ to (3.18) satisfy $-w''\geq 0$ and hence are positive by the strong maximum principle \cite[6.4. Theorem 3]{E}. Positive solutions to (3.4) are even symmetric for $\phi=\pi/2$ and increasing in $(0,\pi/2)$ by the symmetry result for the one-dimensional non-autonomous Dirichlet problem \cite[Theorem 1']{GNN}.\\

In the sequel, we first state the three minimax theorems and construct classical solutions to (3.8) following the steps above.  

\subsection{The three minimax theorems}

According to \cite{Willem}, we define the Palais--Smale condition and state three minimax theorems for a Banach space $X$ equipped with the norm $||\cdot||_{X}$. Let $C^{1}(X;\mathbb{R})$ be the space of all functionals $I[\cdot]: X\to \mathbb{R}$ such that the Fr\'echet derivative $I': X\to X^{*}$ exists and is continuous on $X$. We say that $w\in X$ is a critical point of $I$ if 
\begin{align*}
\langle I'[w],\eta\rangle=0\quad \textrm{for all}\ \eta\in X.
\end{align*}
The constant $c\in \mathbb{R}$ is a critical value if a critical point $w\in X$ exists, such as $c=I[w]$. 
\begin{defn}[Palais--Smale condition]
We say that a sequence $\{w_n\}\subset X$ is a Palais--Smale sequence at level $c\in \mathbb{R}$ if 
\begin{align*}
I[w_n]&\to c,\\ 
I'[w_n]&\to 0\quad \textrm{in}\ X^{*}.
\end{align*}
We say that $I$ satisfies the $(\textrm{PS})_c$ condition if any Palais--Smale sequence at level $c\in \mathbb{R}$ has a convergent subsequence in $X$.
\end{defn}

The following three minimax theorems are due to Ambrosetti and Rabinowitz \cite[Theorems 2.10, 2.11, 2.12]{Willem}.

\begin{lem}[Mountain pass theorem]
Assume that there exist $w_0\in X$ and $r_0>0$ such that $||w_0||_{X}>r_0$ and 
\begin{align*}
\inf_{||w||_{X}=r_0}I[w]>I[0]\geq I[w_0]. 
\end{align*}
Assume that $I$ satisfies the $(\textrm{PS})_c$ condition with 
\begin{align*}
c&=\inf_{\gamma\in \Lambda}\max_{0\leq t\leq 1}I[\gamma(t)], \\
\Lambda&=\left\{\gamma\in C([0,1]; X )\ \middle|\ \gamma(0)=0,\ \gamma(1)=w_0\  \right\}.
\end{align*}
Then, $c$ is a critical value of $I$.
\end{lem}

\begin{rem}
For the functional $I$ satisfying $I[0]\geq I[w_0]$, the functional estimate in Lemma 3.6 can be expressed as $\inf_{N}I>\max_{M_0}I$ for 
\begin{align*}
N&=\left\{w\in X\ \middle|\ ||w||_{X}=r_0  \right\}, \\
M_0&=\left\{0,w_0\right\}.
\end{align*}
\end{rem}

We apply linking and saddle point theorems when $X$ admits a direct sum decomposition $X=Y\oplus Z$ with a finite-dimensional subspace $Y$ and a subspace $Z$.

\begin{lem}[Linking theorem]
For $\rho_0>r_0>0$ and $z_0\in Z$ such that $||z_0||_{X}=r_0$, set 
\begin{align*}
M&=\left\{w=y+\lambda z_0\in Y\oplus \mathbb{R}z_0\ |\ ||w||_{X}\leq \rho_0,\ \lambda\geq 0,\ y\in Y\right\},\\
M_0&=\left\{w=y+\lambda z_0\in Y\oplus \mathbb{R}z_0\ |\ \lambda=0\ \textrm{and}\ ||y||_{X}\leq \rho_0,\ \textrm{or}\ \lambda> 0\ \textrm{and}\ ||w||_{X}= \rho_0 \right\},\\
N&=\left\{z\in Z\ |\ ||z||_{X}=r_0\  \right\}.
\end{align*}
Assume that there exist $\rho_0>r_0>0$ and $z_0\in Z$ satisfying $||z_0||_{X}=r_0$ such that 
\begin{align*}
\inf_{N}I>\max_{M_0}I.
\end{align*}
Assume that $I$ satisfies the $(\textrm{PS})_c$ condition with 
\begin{align*}
c&=\inf_{\gamma\in \Lambda}\max_{w\in M}I[\gamma[w]],\\
\Lambda&=\left\{\gamma\in C(M; X)\ \middle|\ \gamma|_{M_0}=id\  \right\}.
\end{align*}
Then, $c$ is a critical value of $I$.
\end{lem}

\begin{lem}[Saddle-point theorem]
For $\rho_0>0$, set 
\begin{align*}
M&=\left\{w\in Y\ \middle|\ ||w||_{X}\leq \rho_0\ \right\},\\
M_0&=\left\{w\in Y\ \middle|\ ||w||_{X}= \rho_0\ \right\},\\
N&=Z.
\end{align*}
Assume that there exists $\rho_0>0$ such that 
\begin{align*}
\inf_{N}I>\max_{M_0}I.
\end{align*}
Assume that $I$ satisfies the $(\textrm{PS})_c$ condition with 
\begin{align*}
c&=\inf_{\gamma\in \Lambda}\max_{w\in M}I[\gamma[w]],\\
\Lambda&=\left\{\gamma\in C(M; X)\ \middle|\ \gamma|_{M_0}=id\  \right\}.
\end{align*}
Then, $c$ is a critical value of $I$.
\end{lem}

\subsection{Bilinear form estimates}

We first prepare $L^{2}$-bilinear form estimates to (3.9) and the direct sum decomposition (3.11) by the Fourier expansion of $w\in L^{2}(D)$ using the orthogonal basis $\{e_k\}$ on $L^{2}(D)$. We then show that the bilinear form is coercive on $H^{1}_{0}(D)$ for $0<\beta<1$ and on $Z$ for $\beta\in \mathbb{R}\backslash [-2,1)$ by a contradiction argument. 

\begin{lem}
(i) For $0<\beta<1$, the bilinear form (3.9) satisfies  
\begin{align}
B(w,w)\geq \mu_1 ||w||_{L^{2}}^{2},\quad w\in H^{1}_{0}(D).
\end{align}
(ii) For $\beta\in \mathbb{R}\backslash  [-2,1)$, the direct sum decomposition (3.11) holds. The bilinear form (3.9) satisfies 
\begin{equation}
\begin{aligned}
&B(w,w)=B(y,y)+B(z,z),\quad w=y+z\in Y\oplus Z, \\
\label{s3Bcoer}
&\mu_K||y||_{L^{2}}^{2}\geq  B(y,y)\geq \mu_{1}||y||_{L^{2}}^{2},\quad y\in Y, \\
&B(z,z)\geq \mu_{K+1}||z||_{L^{2}}^{2},\quad z\in Z.
\end{aligned}
\end{equation}
\end{lem}

\begin{proof}
(i) By $-e_k''-\beta^{2}e_k=\mu_k e_k$ and integration by parts,
\begin{align*}
B(e_k,e_l)
=\int_{D}(e_k'e_l'-\beta^{2}e_ke_l) \dd\phi
=\mu_k\int_{D}e_ke_l \dd\phi=\mu_k\delta_{k,l},
\end{align*}
where $\delta_{k,l}$ denotes Kronecker delta. By Fourier expansion of $w=\sum_{k=1}^{\infty}(w,e_k)_{L^{2}}e_k 
\in H^{1}_{0}(D)$,
\begin{align*}
B(w,w)=\sum_{k,l=1}^{\infty}(w,e_k)_{L^{2}}(w,e_l)_{L^{2}}B(e_k,e_l)=\sum_{k=1}^{\infty}\mu_k|(w,e_k)_{L^{2}}|^{2}\geq \mu_1||w||_{L^{2}}^{2}.  
\end{align*}
Thus (3.20) holds. 

(ii) For $w\in H^{1}_{0}(D)$, we set 
\begin{align*}
w=\sum_{k=1}^{\infty}(w,e_k)_{L^{2}}e_k=\sum_{k=1}^{K}(w,e_k)_{L^{2}}e_k+\sum_{k=K+1}^{\infty}(w,e_k)_{L^{2}}e_k=:y+z.
\end{align*}
By $w\in H^{1}_{0}(D)$, $y\in H^{1}_{0}(D)$ and $z=w-y\in H^{1}_{0}(D)$. Thus, $z\in Z$ and the direct sum decomposition (3.11) holds. By
\begin{align*}
B(y,z)&=\sum_{k=1}^{K}\sum_{l=K+1}^{\infty}(w,e_k)_{L^{2}}(w,e_l)_{L^{2}}B(e_k,e_l)=0,\\
B(y,y)&=\sum_{k,l=1}^{K}(w,e_k)_{L^{2}}(w,e_l)_{L^{2}}B(e_k,e_l)=\sum_{k=1}^{K}\mu_n|(w,e_k)_{L^{2}}|^{2},\\
B(z,z)&=\sum_{k,l=K+1}^{\infty}(w,e_k)_{L^{2}}(w,e_l)_{L^{2}}  B(e_k,e_l)=\sum_{k=K+1}^{\infty}\mu_k|(w,e_k)_{L^{2}}|^{2},
\end{align*}
the bilinear form estimates (3.21) follow.
\end{proof}

\begin{lem}
(i) For $0<\beta<1$, there exists $\delta>0$ such that 
\begin{align}
\label{s3Bcoer2}
B(w,w)\geq \delta ||w||_{H^{1}_{0}}^{2},\quad w\in H^{1}_{0}(D).
\end{align}
(ii) For $\beta\in \mathbb{R}\backslash [-2,1)$, there exists $\delta>0$ such that 
\begin{align}
B(z,z)\geq \delta ||z||_{H^{1}_{0}}^{2},\quad z\in Z.
\end{align}
\end{lem}

\begin{proof}
We show (3.22). By (3.20), 
\begin{align*}
\delta=\inf\left\{B(w,w)\ \middle|\ ||w||_{H^{1}_{0}}=1,\ w\in H^{1}_{0}(D) \right\}\geq 0.
\end{align*}
We take a sequence $\{w_n\}\subset H^{1}_{0}(D)$ such that $B(w_n,w_n)\to \delta$ and $||w_n||_{H^{1}_{0}}^{2}=1$. By choosing a subsequence, $w_n\rightharpoonup w$ on $H^{1}_{0}(D)$ and $w_n\to w$ on $L^{2}(D)$ for some $w\in H^{1}_{0}(D)$ and  
\begin{align*}
\delta=\lim_{n\to\infty}B(w_n,w_n)=1-\beta^{2}||w||_{L^{2}}^{2}.
\end{align*}
By the lower semi-continuity of the norm $||w_n||_{H^{1}_{0}}$ in the weak convergence and the $L^{2}$-bilinear form estimate (3.20),
\begin{align*}
\delta=\lim_{n\to\infty}B(w_n,w_n)\geq B(w,w)\geq \mu_{1}||w||_{L^{2}}^{2}.
\end{align*}
If $w\equiv  0$, $\delta=1$. If $w\nequiv 0$, $\delta>0$. Thus, $\delta$ is positive, and (3.22) holds.

The estimate (3.23) follows the same argument using the $L^{2}$-bilinear form estimate $(3.21)_3$ on $Z$. 
\end{proof}

\subsection{Regularity of critical points}

We show that the functional (3.12) associated with the problem (3.8) is Fre\'chet differentiable, and its critical points are classical solutions to (3.8). In the subsequent subsection, we also prepare a functional identity to demonstrate the Palais--Smale condition.

\begin{prop}
The functional $I\in C^{1}(H^{1}_{0}(D);\mathbb{R})$ in (3.12) satisfies 
\begin{equation}
\langle I'[w],\eta\rangle =B(w,\eta)-(g(\cdot,w),\eta)_{L^{2}},\quad w,\eta\in H^{1}_{0}(D),
\end{equation}
\begin{equation}
\begin{aligned}
I[w]-\sigma \langle I'[w],w\rangle=\left(\frac{1}{2}-\sigma\right)B(w,w)&+c_1\left(\sigma-\frac{\beta}{2(\beta+1)}  \right)\int_{D}|w|^{2+\frac{2}{\beta}}\dd \phi\\
&+c_2\left(\sigma-\frac{\beta}{2(\beta+3)}  \right)\int_{D}(\sin\phi)|w|^{2+\frac{3}{\beta}}\dd \phi,\quad \sigma\in \mathbb{R}.
\end{aligned}
\end{equation}
\end{prop}

\begin{proof}
We set
\begin{align*}
I[w]=\frac{1}{2}B(w,w)-\int_{D}G(\phi,w)\dd \phi=:I_0[w]-J[w].
\end{align*}
For $\varepsilon>0$ and $w, \eta\in H^{1}_{0}(D)$, 
\begin{align*}
\frac{1}{\varepsilon}\left(I_0[w+\varepsilon \eta]-I_0[w] \right)
=B(w,\eta)+\frac{\varepsilon}{2}B(\eta,\eta).
\end{align*}
Thus the Gateaux derivative $D_G I_0 : H^{1}_{0}(D)\to H^{-1}(D)$ exists and $<D_G I_0[w],\eta>=B(w,\eta)$ for $\eta\in H^{1}_{0}(D)$. By the boundedness of the bilinear form on $H^{1}_{0}(D)\times H^{1}_{0}(D)$, the Gateaux derivative $D_G I_0[w]$ is continuous. Thus the Fre\'chet derivative $I_0' = D_G I_0$ exists and $I_0 \in C^1(H_0^1(\Omega); \mathbb{R})$. 

For $\varepsilon>0$ and $w, \eta\in H^{1}_{0}(D)$, 
\begin{align*}
\frac{1}{\varepsilon}\left(J[w+\varepsilon \eta]-J[w] \right)
=\frac{1}{\varepsilon}\int_{D}\left(G(\phi,w+\varepsilon\eta)-G(\phi,w)\right)\dd \phi
=\int_{D}\left(\int_{0}^{1}g(\phi,w+\varepsilon \eta\sigma)\eta \dd \sigma\right) \dd \phi
\end{align*}
By the pointwise estimate $|g(\phi,w+\varepsilon \eta \sigma)|\leq c_1(|w|+|\eta|)^{1+2/\beta}+c_2(|w|+|\eta|)^{1+3/\beta}$ for $0<\varepsilon\leq 1$ and $0\leq \sigma \leq 1$, letting $\varepsilon\to 0$ implies that 
\begin{align*}
<D_G J[w], \eta>=(g(\cdot,w),\eta )_{L^{2}},\quad \eta\in H^{1}_{0}(D).
\end{align*}
Since the Gateaux derivative $D_G J : H^{1}_{0}(D)\to H^{-1}(D)$ exists and is continuous, $J \in C^1(H_0^1(\Omega); \mathbb{R})$ and (3.24) holds.   

The identity (3.25) follows (3.12) and (3.24).
\end{proof}

\begin{lem}
Critical points $w\in H^{1}_{0}(D)$ of the functional $I$ are classical solutions $w\in C^{2}(\overline{D})\cap C_0(\overline{D})$ to (3.8). For $\beta>0$, $w\in C^{3}(\overline{D})$. 
\end{lem}

\begin{proof}
Since $H^{1}_{0}(D)\subset C^{1/2}(\overline{D})$, $g(\phi,w)$ is bounded continuous in $\overline{D}$ for critical points $w\in H^{1}_{0}(D)$ of $I$. By (3.22), the weak derivative $-w''=\beta^{2}w+g(\phi,w)$ is bounded and continuous in $\overline{D}$. Thus, $w$ is a classical solution to (3.8). For $\beta>0$, $g(\cdot ,w)$ is a $C^{1}$-function for $w$ and $g(\phi,w(\phi))\in C^{1}[0,\pi]$. By the equation (3.8), $w\in C^{3}[0,\pi]$.
\end{proof}

\subsection{The Palais--Smale condition}

We show that the functional $I$ satisfies the Palais--Smale condition. The primary step of the proof is to obtain the boundedness of the Palais--Smale sequence for the superlinear case $\beta>0$ and for the sublinear case $\beta<-2$.

For $0<\beta<1$, the Palais--Smale sequence is bounded by the identity (3.25) and the coercive estimate (3.22) on $H^{1}_{0}(D)$. For $\beta\geq 1$,  we use the direct sum decomposition (3.11) with the finite-dimensional space $Y$ and the coercive estimate (3.23) on $Z$. By a contradiction argument, we show the boundedness of the Palais--Smale sequence for all $\beta>0$. 

\begin{prop}
For $\beta>0$, any sequences $\{w_n\}\subset H^{1}_{0}(D)$ satisfying 
\begin{equation}
\begin{aligned}
&\sup_{n\geq 1}I[w_n]<\infty,\\
&I'[w_n]\to 0\quad H^{-1}(D),
\end{aligned}
\end{equation}
are bounded in $H^{1}_{0}(D)$.
\end{prop}

\begin{proof}
We argue by contradiction. Suppose that $M_n=||w_n||_{H^{1}_{0}}$ diverges. We set $\tilde{w}_n=w_n/M_n$ so that $||\tilde{w}_n||_{H^{1}_{0}}=1$. We take $\sigma \in (\beta/2(\beta+1), 1/2)$ so that all coefficients on the right-hand side of (3.25) are positive. If $c_1\neq 0$, we divide (3.25) for $w_n$ by $M_n^{2+2/\beta}$. If $c_2\neq 0$, we divide (3.25) for $w_n$ by $M_n^{2+3/\beta}$. By (3.25), letting $n\to\infty$ implies that $\tilde{w}_n\to 0$ on $L^{2}(D)$. 

By the direct sum decomposition (3.11), $\tilde{w}_n=\tilde{y}_n+\tilde{z}_n\in Y\oplus Z$ satisfies  
\begin{align*}
1=||\tilde{w}_{n}||_{H^{1}_{0}}^{2}&=||\tilde{y}_{n}||_{H^{1}_{0}}^{2}
+2(\tilde{y}_{n}', \tilde{z}_{n}')_{L^{2}}
+||\tilde{z}_{n}||_{H^{1}_{0}}^{2},\\
||\tilde{w}_{n}||_{L^{2}}^{2}&=||\tilde{y}_{n}||_{L^{2}}^{2}+||\tilde{z}_{n}||_{L^{2}}^{2}.
\end{align*}
Since $\tilde{w}_n$ vanishes on $L^{2}(D)$, so is $\tilde{y}_n$ and $\tilde{z}_n$. Since $Y$ is finite-dimensional, $\tilde{y}_n$ vanishes on $H^{1}_{0}(D)$ and $\tilde{z}_n$ is bounded on $H^{1}_{0}(D)$. Letting $n\to\infty$ implies that $1=\lim_{n\to\infty}||\tilde{z}_n||_{H^{1}_{0}}$. By dividing (3.25) for $w_n$ by $M_n^{2}$ and applying the bilinear form estimates $(3.23)_2$, (3.24), and (3.25), 
\begin{align*}
\frac{1}{M_n^{2}}\left(   I[w_n] -\sigma 
\langle I'[w_n],w_n \rangle \right)
\geq \left(\frac{1}{2}-\sigma\right)B(\tilde{w}_n, \tilde{w}_n)
&= \left(\frac{1}{2}-\sigma\right)\left(B(\tilde{y}_n, \tilde{y}_n)+B(\tilde{z}_n, \tilde{z}_n)  \right) \\
&\geq \left(\frac{1}{2}-\sigma\right)\left( \mu_1||\tilde{y}_n||_{L^{2}}^{2}+\delta||\tilde{z}_n||_{H^{1}_{0}}^{2} \right). 
\end{align*}
Letting $n\to\infty$ implies that $0\geq (1/2-\sigma)\delta>0$. This is a contradiction. Thus $\{w_n\}$ is bounded in $H^{1}_{0}(D)$.
\end{proof}

We use the identity (3.24) to demonstrate the Palais--Smale sequence boundedness for $\beta<-2$. 

\begin{prop}
For $\beta<-2$, any sequences $\{w_n\}\subset H^{1}_{0}(D)$ satisfying (3.26) are bounded in $H^{1}_{0}(D)$.
\end{prop}

\begin{proof}
Suppose that $M_n=||w_n||_{H^{1}_{0}}$ diverges. Then, $\tilde{w}_{n}=w_n/M_n$ satisfies $||\tilde{w}_{n}||_{H^{1}_{0}}=1$. By the direct sum decomposition (3.11), we set $\tilde{w}_{n}=\tilde{y}_{n}+\tilde{z}_{n}\in Y\oplus Z$ and $\tilde{\eta}_{n}=-\tilde{y}_n+\tilde{z}_n$. Since $\{\tilde{y}_{n}\}$ is bounded in $L^{2}(D)$ and $Y$ is finite-dimensional, $\{\tilde{y}_{n}\}$ and $\{\tilde{z}_{n}\}$ are bounded in $H^{1}_{0}(D)$. By substituting $w_n$ and $\tilde{\eta}_{n}$ into (3.24),
\begin{align*}
\frac{1}{M_n}\langle I'[w_n],\tilde{\eta}_{n}\rangle=B(\tilde{w}_n,\tilde{\eta}_{n} )-M_n^{\frac{2}{\beta}}c_1\int_{D}\tilde{w}_{n}|\tilde{w}_{n}|^{\frac{2}{\beta}}\tilde{\eta}_{n}\dd \phi
-M_n^{\frac{3}{\beta}}c_2\int_{D}(\sin\phi)\tilde{w}_{n}|\tilde{w}_{n}|^{\frac{3}{\beta}}\tilde{\eta}_{n}\dd \phi.
\end{align*}
By $\beta<-2$ and $c_2=0$ for $-3\leq \beta<-2$, letting $n\to\infty$ implies that $\lim_{n\to\infty}B(\tilde{w}_n,\tilde{\eta}_{n} )=0$. By the bilinear form estimates $(3.21)_2$ and (3.23), 
\begin{align*}
0=\lim_{n\to\infty}B(\tilde{w}_n,\tilde{\eta}_{n} )=\lim_{n\to\infty}\left(-B(\tilde{y}_n,\tilde{y}_{n} )+B(\tilde{z}_n,\tilde{z}_{n} ) \right)
\geq \limsup_{n\to\infty}\left( -\mu_K||\tilde{y}_{n}||_{L^{2}}^{2}+\delta||\tilde{z}_{n}||_{H^{1}_{0}}^{2}\right).
\end{align*}
Thus $\{\tilde{y}_{n}\}$ and $\{\tilde{z}_{n}\}$ vanish in $H^{1}_{0}(D)$ and we obtain the contradiction $1=\lim_{n\to\infty}||\tilde{w}_n||_{H^{1}_{0}}=0$. 
\end{proof}

\begin{lem}
For $\beta\in \mathbb{R}\backslash [-2,0]$ and $c\in \mathbb{R}$, the functional $I$ satisfies $(\textrm{PS})_c$ condition. 
\end{lem}

\begin{proof}
We take a Palais--Smale sequence $\{w_n\}\subset H^{1}_{0}(D)$ at level $c\in \mathbb{R}$. The sequence $\{w_n\}$ is bounded by Propositions 3.14 and 3.15, and we take a subsequence (still denoted by $\{w_n\}$) such that $w_n\rightharpoonup w$ on $H^{1}_{0}(D)$ and $w_n\to w$ in $C(\overline{D})$ for some $w\in H^{1}_{0}(D)$. By the weak convergence, 
\begin{align*}
\lim_{n\to\infty}\langle I'[w],w_n-w \rangle=0.
\end{align*}
On the other hand, by the condition $(3.26)_2$,
\begin{align*}
\lim_{n\to\infty}\langle I'[w_n],w_n-w \rangle=0,
\end{align*}
so by the identity (3.24), 
\begin{align*}
\langle I'[w_n]-I'[w],w_n-w\rangle
=B(w_n-w,w_n-w)-(g(\cdot,w_n)-g(\cdot,w), w_n-w )_{L^{2}}.
\end{align*}
The left-hand side vanishes as $n\to\infty$. Since $g(\phi, w_n)$ is uniformly bounded in $\overline{D}$ and $w_n$ uniformly converges to $w$ in $\overline{D}$, letting $n\to\infty$ implies that 
\begin{align*}
0=\lim_{n\to\infty}B(w_n-w,w_n-w)=\lim_{n\to\infty}\int_{D}\left( |w_n'-w'|^{2}-\beta^{2}|w_n-w|^{2} \right) \dd \phi
=\lim_{n\to\infty}\int_{D} |w_n'-w'|^{2} \dd \phi.
\end{align*}
Thus $w_n$ strongly converges to $w$ in $H^{1}_{0}(D)$.
\end{proof}

\subsection{Functional inequalities on subsets}

We choose subsets $N$ and $M_0$ in $H^{1}_{0}(D)$ according to the value of $\beta\in \mathbb{R}\backslash [-2,0]$ and obtain the inequality of the form (3.14). 

\begin{lem}
The functional $I$ satisfies the inequality (3.14) for $\beta$, $N$ and $M_0$ given by (3.15)-(3.17).
\end{lem}

\begin{proof}
(i) $0<\beta<1$. By the embedding $H^{1}_{0}(D)\subset L^{\infty}(D)$ and the coercive estimate (3.20) on $H^{1}_{0} (D)$, there exists a constant $C>0$ such that for $w\in H^{1}_{0}(D)$,
\begin{align*}
I[w]
\geq \frac{\delta}{2}||w||_{H^{1}_{0}}^{2}
-C\left(\frac{c_1\beta}{2(\beta+1)}||w||_{H^{1}_{0}}^{2+\frac{2}{\beta}}  
+\frac{c_2\beta}{2\beta+3}||w||_{H^{1}_{0}}^{2+\frac{3}{\beta}}
\right)
=||w||_{H^{1}_{0}}^{2}\left(\frac{\delta}{2}
-C\left(\frac{c_1\beta}{2(\beta+1)}||w||_{H^{1}_{0}}^{\frac{2}{\beta}}  
+\frac{c_2\beta}{2\beta+3}||w||_{H^{1}_{0}}^{\frac{3}{\beta}}
 \right)\right).
\end{align*}
Thus $\inf_{N}I>0$ for $N=\{w\in H^{1}_{0}(D)\ |\ ||w||_{H^{1}_{0}}=r_0\ \}$ and some small $r_0>0$.

For $\rho>0$ and $w\in H^{1}_{0}(D)$ satisfying $||w||_{H^{1}_{0}}=r_0$, letting $\rho\to\infty$ implies that 
\begin{align*}
I[\rho w]
&=\frac{\rho^{2}}{2}B(w,w)-\rho^{2+\frac{2}{\beta}}\frac{c_1\beta}{2(\beta+1)}\int_{D}|w|^{2+\frac{2}{\beta}}\dd \phi
-\rho^{2+\frac{3}{\beta}}\frac{c_1\beta}{2\beta+3}\int_{D}\sin\phi |w|^{2+\frac{3}{\beta}}\dd \phi\\
&=\rho^{2}\left(\frac{1}{2}B(w,w)-\rho^{\frac{2}{\beta}}\frac{c_1\beta}{2(\beta+1)}\int_{D}|w|^{2+\frac{2}{\beta}}\dd \phi
-\rho^{\frac{3}{\beta}}\frac{c_1\beta}{2\beta+3}\int_{D}\sin\phi |w|^{2+\frac{3}{\beta}}\dd \phi\right)\to -\infty.
\end{align*}
Thus $\max_{M_0}I\leq 0$ for $M_0=\{0,w_0\}$ and some $w_0\in H^{1}_{0}$ such that $||w_0||_{H^{1}_{0}}>r_0$. We demonstrated the inequality (3.14) for $\beta$, $N$, and $M_0$ given by (3.15).\\

\noindent
(ii) $\beta\geq 1$. By the coercive estimate (3.23) on $Z$, for $N=\{w\in Z\ |\ ||w||_{H^{1}_{0}}=r_0\ \}$ and some small $r_0>0$, 
\begin{align*}
\inf_{N}I=\inf\left\{I[w]\ \middle|\ w\in Z,\ ||w||_{H^{1}_{0}}=r_0\  \right\}>0.
\end{align*}
We set $z_0=r_0e_{K+1}/||e_{K+1}||_{H^{1}_{0}}$. For $\rho>0$, $w=y+\lambda z_0$, $y\in Y$, and $\lambda>0$, letting $\rho\to\infty$ implies that 
\begin{align*}
I[\rho w]=\rho^{2}\left( \frac{1}{2}B(w,w)-\rho^{\frac{2}{\beta}}\frac{c_1\beta}{2(\beta+1)}\int_{D}|w|^{2+\frac{2}{\beta}}\dd \phi
-\rho^{\frac{3}{\beta}}\frac{c_1\beta}{2\beta+3}\int_{D}\sin\phi|w|^{2+\frac{3}{\beta}}\dd \phi   \right)\to -\infty.
\end{align*}
Since $Y\oplus \mathbb{R}z_0$ is finite-dimensional, there exists $\rho_0>r_0$ such that 
\begin{align*}
\max\left\{ I[w]\ \middle|\ ||w||_{H^{1}_{0}}=\rho_0,\ w\in Y\oplus \mathbb{R}z_0\  \right\}\leq 0.
\end{align*}
The functional $I$ is non-positive on $Y$ and hence $\max_{M_0}I\leq 0$ for $M_0=\{w=y+\lambda z_0\in Y\oplus \mathbb{R}z_0\ |\ \lambda=0\ \textrm{and}\ ||w||_{H^{1}_{0}}\leq \rho_0,\ \textrm{or}\ \lambda>0\ \textrm{and}\ ||w||_{H^{1}_{0}}=\rho_0\ \} $. We demonstrated the inequality (3.14) for $\beta$, $N$, and $M_0$ given by (3.16).\\

\noindent
(iii) $\beta<-2$. For arbitrary $\varepsilon>0$, we apply Young's inequality to estimate 
\begin{align*}
||w||_{H^{1}_{0}}^{2+\frac{2}{\beta}}&\leq \varepsilon \left(1+\frac{1}{\beta}\right) ||w||_{H^{1}_{0}}^{2}+\frac{1}{(-\beta) \varepsilon^{-\beta-1}},\\
||w||_{H^{1}_{0}}^{2+\frac{3}{\beta}}&\leq \varepsilon \left(1+\frac{3}{2\beta}\right) ||w||_{H^{1}_{0}}^{2}+\frac{3}{(-2\beta) \varepsilon^{\frac{3}{-2\beta}-1}}.
\end{align*}
By the embedding $H^{1}_{0}(D)\subset L^{\infty}(D)$, there exists a constant $C$ independent of $\varepsilon$ such that 
\begin{align*}
I[z]
\geq \frac{1}{2}B(z,z)
-C\left(
\frac{c_1\beta}{2(\beta+1)}||z||_{H^{1}_{0}}^{2+\frac{2}{\beta}}
+\frac{c_2\beta}{2\beta+3}||z||_{H^{1}_{0}}^{2+\frac{3}{\beta}}
\right).
\end{align*}
By applying the coercive estimate (3.23) on $Z$ and the above estimates,
\begin{align*}
I[z]\geq \frac{1}{2}\left(\delta -C(c_1+c_2)\varepsilon \right) ||z||_{H^{1}_{0}}^{2}
-C\left(  \frac{-1}{(\beta+1)\varepsilon^{-\beta-1}}+\frac{-3}{(2\beta+3)\varepsilon^{-\frac{3}{2\beta}-1 }}     \right).
\end{align*}
Thus for $N=Z$ and small $\varepsilon>0$ such that $\delta-C(c_1+c_2)\varepsilon>0$,
\begin{align*}
\inf_{N}I\geq C \left(\frac{1}{(\beta+1)\varepsilon^{-\beta-1}}+\frac{3}{(2\beta+3)\varepsilon^{-\frac{3}{2\beta}-1  } } \right)>-\infty.
\end{align*}
By the embedding $H^{1}_{0}(D)\subset L^{\infty}(D)$, there exists a constant $C>0$ such that 
\begin{align*}
\int_{D}|y|^{2+\frac{2}{\beta}}\dd \phi \leq C||y||_{H^{1}_{0}}^{2+\frac{2}{\beta}}.
\end{align*}
By dividing both side by $||y||_{H^{1}_{0}}^{2+2/\beta}$, 
\begin{align*}
C_1=\inf\left\{  \int_{D}|\tilde{y}|^{2+\frac{2}{\beta}}\dd \phi\  \middle|\ ||\tilde{y}||_{H^{1}_{0}}=1,\  \tilde{y}\in Y\   \right\}.
\end{align*}
Since $Y$ is finite-dimensional, $C_1$ is positive. Similarly, the constant
\begin{align*}
C_2=\inf\left\{  \int_{D}(\sin\phi)|\tilde{y}|^{2+\frac{3}{\beta}}\dd \phi\  \middle|\ ||\tilde{y}||_{H^{1}_{0}}=1,\  \tilde{y}\in Y\   \right\}
\end{align*}
is positive. Since $B(y,y)$ is non-positive for $y\in Y$, 
\begin{align*}
\max\left\{I[y]\ \middle|\ ||y||_{H^{1}_{0}}=\rho, \ y\in Y\  \right\}
\leq -\frac{c_1\beta}{2(\beta+1)}\rho^{2+\frac{2}{\beta}} C_1-\frac{c_2\beta}{2\beta+3}\rho^{2+\frac{3}{\beta}}C_2.
\end{align*}
The right-hand side diverges to minus infinity as $\rho\to\infty$. Thus for $M_0=\{w\in H^{1}_{0}(D)\ |\ ||w||_{H^{1}_{0}}=\rho_0\ \}$ and some large $\rho_0>0$, 
\begin{align*}
\inf_{N}I>\max_{M_0}I.
\end{align*}
We demonstrated the inequality (3.14) for $\beta$, $N$, and $M_0$ given by (3.17).
\end{proof}

\begin{proof}[Proof of Theorem 3.1]\!\!
The functional $I$ satisfies the (PS$)_c$ condition for $\beta\in \mathbb{R}\backslash [-2,0]$ and $c\!\in\! \mathbb{R}\!$ by Lemma 3.16.

\noindent
(i) $0<\beta<1$. The inequality (3.14) holds for $\beta$, $N$, and $M_0$ given by (3.15) by Lemma 3.17. Thus, the assumption of the mountain pass theorem (Lemma 3.6) is satisfied by Remark 3.7, and the existence of a critical point $w\in H^{1}_{0}(D)$ follows.\! The critical point is a classical solution $w\in C^{3}(\overline{D})$ to (3.8) by Lemma 3.13.

\noindent
(ii) $\beta\geq 1$. The inequality (3.14) holds for $\beta$, $N$, and $M_0$ given by (3.16) by Lemma 3.17. Thus, the existence of a critical point $w\in H^{1}_{0}(D)$ follows the linking theorem (Lemma 3.8). The critical point is a classical solution $w\in C^{3}(\overline{D})$ to (3.8) by Lemma 3.13.

\noindent
(iii) $\beta<-2$. The inequality (3.14) holds for $\beta$, $N$, and $M_0$ given by (3.17) by Lemma 3.17. Thus, the existence of a critical point $w\in H^{1}_{0}(D)$ follows the saddle point theorem (Lemma 3.9). The critical point is a classical solution $w\in C^{2}(\overline{D})$ to (3.8) by Lemma 3.13.

For $0<\beta<1$, we consider the functional $\tilde{I}$ in (3.19) instead of $I$. We modify the argument to $I$ and obtain a classical solution $w\in C^{3}(\overline{D})$ to (3.18). Solutions to (3.18) are positive by a strong maximum principle \cite[6.4. Theorem 3]{E} and thus solutions to (3.4). Positive solutions to (3.4) are even symmetric with respect to $\phi=\pi/2$ and increasing in $(0,\pi/2)$ by a symmetry result for non-autonomous one-dimensional Dirichlet problems \cite[Theorem 1']{GNN}.
\end{proof}

\subsection{Homogeneous solution construction}

We demonstrate the existence of rotational ($-\alpha$)-homogeneous solutions to (1.1) and (1.4) for $\alpha\in \mathbb{R}\backslash [-1,1]$ with regular profiles by using Theorem 3.1 and complete the proofs for (iii) and (iv) of Theorems 1.2, 1.5, and 1.7.

\begin{prop}
Let $(\beta,c_1,c_2)$ satisfy (3.6). Let $w\in C^{2}[0,\pi]\cap C_{0}[0,\pi]$ be a solution to (3.4). Set $(\alpha, C_1, C_2)$ by (3.5). Set $(\psi,\Pi,\rho)$ by (3.2) and (3.3), and extend them to $\mathbb{R}^{2}\backslash \{0\}$ by the $x_2$-symmetry (1.4). Then, $(u,p,\rho)\in C^{2}(\mathbb{R}^{2}\backslash \{0\})$ for $\alpha>1$ and $(u,p,\rho)\in C^{1}(\mathbb{R}^{2})$ for $\alpha<-1$.
\end{prop}

\begin{proof}
The functions $(u,p,\rho)$ are of the form: 
\begin{equation}
\begin{aligned}
u=\frac{-1}{r^{\alpha}}\left((\alpha-1)we_{\phi}+w'e_{r} \right),\quad
\rho=\frac{C_2}{r^{2\alpha+1}}w|w|^{1+\frac{3}{\alpha-1}},\quad
\Pi=\frac{C_1}{r^{2\alpha}}|w|^{2+\frac{2}{\alpha-1}},\quad
p=\Pi-\frac{1}{2}|u|^{2}-x_2\rho.
\end{aligned}
\end{equation}
The second derivative of $w$ vanishes at $\phi=0$ and $\pi$ by the Dirichlet boundary condition and the equation (3.4). Thus $w\in C^{2}(\mathbb{S}^{1})$ by the odd extension.

For $\alpha<-1$, $u\in C^{1}(\mathbb{S}^{1})$ and $\Pi\in C^{1}(\mathbb{S}^{1})$. For $-2\leq \alpha<-1$, $C_2=0$ and $\rho\equiv 0$. For $\alpha<-2$, $\rho\in C^{1}(\mathbb{S}^{1})$. Thus for $\alpha<-1$, $\rho\in C^{1}(\mathbb{S}^{1})$ and $p\in C^{1}(\mathbb{S}^{1})$.

For $\alpha<-1$, $u(x)=r^{-\alpha} u(\phi)$, $\rho(x)=r^{-2\alpha-1}\rho(\phi)$, and $\Pi(x)=r^{-2\alpha}\Pi(\phi)$ are $C^{1}$-functions for $r\geq 0$. Thus $(u,p,\rho)\in C^{1}(\mathbb{R}^{2})$. 

For $\alpha>1$, $w\in C^{3}(\mathbb{S}^{1})$ and $(u,p,\rho)\in C^{2}(\mathbb{R}^{2}\backslash \{0\})$. 
\end{proof}

\begin{lem}
The functions $(u,p,\rho)$ in Proposition 3.18 are ($-\alpha$)-homogeneous solutions to (1.1) and (1.4) in $\mathbb{R}^{2}\backslash \{0\}$ for $\alpha>1$ and in $\mathbb{R}^{2}$ for $\alpha<-1$. 
\end{lem}

\begin{proof}
Since $(u,p,\rho)\in C^{1}(\mathbb{R}^{2}\backslash \{0\})$, we can differentiate $\psi$, $\Pi$, and $\rho$ in (3.27). By (3.3), 
\begin{align*}
\omega&=\Delta \psi
=\frac{1}{r^{\alpha+1}}\left((\alpha-1)^{2}w+w''\right)
=\frac{1}{r^{\beta+2}}\left(\beta^{2}w+w''\right),\\
\nabla \Pi&=(\nabla \psi) \Pi'(\psi)
=(\nabla \psi) C_1\left(\frac{2\alpha}{\alpha-1}\right)\frac{1}{r^{\alpha+1}}w|w|^{\frac{2}{\alpha-1}}
=-\nabla \psi  \frac{c_1}{r^{\beta+2}} w|w|^{\frac{2}{\beta}},\\
\nabla \rho&=
(\nabla \psi)\rho'(\psi)
=(\nabla \psi)C_2\left(\frac{2\alpha+1}{\alpha-1}\right)\frac{1}{r^{\alpha+2}}|w|^{1+\frac{3}{\alpha-1}}
=\nabla \psi  \frac{c_2}{r^{\beta+3}} |w|^{1+\frac{3}{\beta}},
\end{align*}
By the equation (3.4) in $(0,\pi)$, 
\begin{align*}
\omega u^{\perp}+\nabla \Pi-x_2\nabla \rho
=-\nabla \psi \frac{1}{r^{\beta+2}}\left(\beta^{2}w+w''+c_1w|w|^{\frac{2}{\beta}}+c_2 \sin\phi |w|^{1+\frac{3}{\beta}} \right)
=0. 
\end{align*}
Since $u=\nabla^{\perp} \psi$ is solenoidal, $(u,p,\rho)$ satisfies the equation (1.1) in $\mathbb{R}^{2}\backslash \{0\}$. For $\alpha<-1$, $(u,p,\rho)\in C^{1}(\mathbb{R}^{2})$ satisfies (1.1) in $\mathbb{R}^{2}$ by continuity at the origin. 
\end{proof}
\begin{proof}[Proof of (iii) and (iv) of Theorems 1.2, 1.5, and 1.7]
The existence results (iii) and (iv) of Theorem 1.2 follow the existence of rotational ($-\alpha$)-homogeneous solutions to (1.1) and (1.4) for $\alpha>1$ and $\alpha<-1$ with $C_1\neq 0$ or $C_2\neq 0$ in Lemma 3.19 ($C_2=0$ for $-2\leq \alpha<-1$). The existence results (iii) and (iv) of Theorems 1.5 and 1.7 are the cases $C_1\neq 0$ and $C_2=0$ and $C_1= 0$ and $C_2\neq 0$, respectively.
\end{proof}

\section{The existence of homogeneous solutions with singular profiles}

We construct ($-\alpha$)-homogeneous solutions with singular profiles $(u,p,\rho)\in C^{\infty}(\mathbb{R}^{2}\backslash \{x_1=0\}\cup \{x_2=0\})$ to (1.1) and (1.4) for $-1<\alpha<1$ satisfying the additional $x_1$-symmetry (1.3).

\subsection{The homogeneous Dubreil-Jacotin--Long equation with negative nonlinear powers}

We look for a stream function $\psi$ satisfying the $x_1$-symmetry
\begin{equation*}
\psi(x_1,x_2)=-\psi(-x_1,x_2),
\end{equation*}
and the Dubreil-Jacotin-Long equation (3.1) in the first quadrant $\{x_1>0, x_2>0\}$ with the additional boundary condition $\psi(0,x_2)=0$ for $x_2\geq 0$. Solutions to the equation (3.1) in the first quadrant provide stationary solutions to (1.1) by using the symmetric extensions (1.3) and (1.4). By choosing homogeneous functions $(\psi,\Pi,\rho)$ in the same way as (3.2) and (3.3) with constants $(\alpha,C_1,C_2)$, 
we arrive at the following problem,
\begin{equation}
\begin{aligned}
-w''&=\beta^{2}w+c_1w^{-s}+c_2\sin\phi w^{-s'},\\
w&>0,\quad \phi\in (0,\pi/2),\\
w(0)&=w(\pi/2)=0,
\end{aligned}
\end{equation}
where  $(\beta,c_1,c_2)$ are as in (3.5), and where
\begin{align}
s=-1-\frac{2}{\beta},\quad s'=-1-\frac{3}{\beta}.
\end{align}
By our choices in (3.5),
 $c_1$ and $c_2$ are zero at $\beta=-1$ and $-3/2$. We consider the constants $(\beta,c_1,c_2)$ satisfying the conditions
\begin{equation}
\begin{aligned}
&-2<\beta<0,\quad c_1, c_2\geq 0,\quad c_1\neq 0\ \textrm{or}\ c_2\neq 0,\\
&c_1=0\quad \textrm{for}\ \beta=-1,\qquad
c_2=0\quad \textrm{for}\ \beta=-3/2.
\end{aligned}
\end{equation}

The main result of this section is the construction of a unique solution to
the above problem.

\begin{thm}
Let $(\beta,c_1,c_2)$ satisfy (4.3). There exists a unique positive solution $w\in C^{\infty}(0,\pi/2)\cap C_0[0,\pi/2]$ to (4.1). For $-2<\beta< -3/2$, $w\in C^{1,2+2/\beta}[0,\pi/2]$. For $-1<\beta<0$, $w'$ diverges at $\phi=0$ and $\pi/2$. 
\end{thm}

\begin{rem}[The second derivative at the boundary]
\label{2ndderrem}
For $c_1\neq 0$, $w''$ diverges at $\phi=0$ and $\pi/2$ due to the negative nonlinear powers. For $c_1=0$ and $c_2\neq 0$, $w''$ diverges at $\phi=\pi/2$.
\end{rem}

\subsection{The autonomous equation}
Elgindi and Huang \cite[Lemma 5.3]{EH} showed the existence of solutions to the autonomous equation (4.1) for constants $-2<\beta<-1$, $c_1>0$, and $c_2=0$ by integrating the equation. We first note that the existence of solutions to the autonomous equation for $-1<\beta<0$ also follows from the same argument. 

\begin{thm}
Let $(\beta,c_1,c_2)$ satisfy (4.3). Assume that $c_2=0$ and $\beta\neq -1$. Then, there exists a unique positive solution $w\in C^{\infty}(0,\pi/2)\cap C_0[0,\pi/2]$ to (4.2). For $-2<\beta< -1$, $w\in  C^{1, 2+2/\beta}[0,\pi/2]$. For $-1<\beta<0$, $w\in C^{-\beta}[0,\pi/2]$. 
\end{thm}

\begin{rem}
When $\beta=-1$ and $c_1=c_2=0$, the equation (4.1) takes the form
\begin{equation*}
\begin{aligned}
-w''&=w,\\
w(0)&=w(\pi/2)=0,
\end{aligned}
\end{equation*}
which has no non-trivial solutions.
\end{rem}

\begin{rem}[The first derivative at the boundary]
Multiplying both sides of (4.1) by $2w'$
and integrating, we find 
\begin{align*}
|w'|^{2}+\beta^{2}|w|^{2}+\frac{2c_1}{1-s}w^{1-s}=C,\quad \phi\in [0,\pi/2],
\end{align*}
for some constant $C$. For $-2<\beta<-1$, $s<1$ and $w'$ is continuous at $\phi=0$ and $\pi/2$. Moreover, this identity implies the H\"older continuity $w'\in C^{1-s}[0,\pi/2]$ for $s=-1-2/\beta$. For $-1<\beta<0$, $s>1$ and $w'$ diverges at $\phi=0$ and $\pi/2$.  
\end{rem}

\begin{proof}[Proof of Theorem 4.3]
For simplicity, we take $c_1=1$. We construct positive increasing solutions to (4.1) with $c_1 = 1, c_2 = 0$ in $(0,\pi/4)$ and extend them to $(\pi/4,\pi/2)$ by reflection $w(\phi)=w(\pi/2-\phi)$. Namely, it suffices to construct solutions to the equations
 \begin{equation}
 \begin{aligned}
 -w''&=\beta^{2}w+w^{-s},\\
 w&>0,\quad w'>0,\quad  \phi\in (0,\pi/4),\\
 w(0)&=w'(\pi/4)=0.
 \end{aligned}
\end{equation}
For a constant $B>0$, we define the function
\begin{align}
\phi_s(w;B)=\int_{0}^{w}\frac{\dd \eta}{\displaystyle\sqrt{\beta^{2}(B^{2}-\eta^{2})+\frac{2}{1-s}\left(B^{1-s}-\eta^{1-s} \right)}},\quad 0\leq w\leq B.
\end{align}
The function $\phi_s(w; B)\in C^{1}(0,B]\cap C[0,B]$ is positive and  increasing and satisfies $\phi_s(0;B)=0$. Since $d \phi_s/d w>0$, the inverse function $w=w(\phi)$ exists and satisfies 
\begin{align*}
w'(\phi)=\sqrt{\beta^{2}(B^{2}-w^{2})+\frac{2}{1-s}\left(B^{1-s}-w^{1-s} \right)},
\qquad w'(0) = 0.
\end{align*}
By differentiating $|w'(\phi)|^{2}$, $w$ satisfies the equations (4.4$)_1$. 

We now show that there is $B > 0$, so that
$w'(\pi/4) \!\!=\!\! 0$. Note that $\phi_s(B;B)$ is
continuous for $B \geq 0$ and
\begin{equation*}
\lim_{B\to0}\phi_s(B;B)=0,\quad \lim_{B\to\infty}\phi_s(B;B)=\frac{(1+s)\pi}{4}>\frac{\pi}{4}.
\end{equation*}
Thus there exists $B_s>0$ such that $\phi_s(B_s; B_s)=\pi/4$. Since
\begin{equation*}
\phi_s(w(\pi/4),B_s)=\frac{\pi}{4}=\phi_s(B_s,B_s),
\end{equation*}
and $\phi_s(w;B_s)$ is increasing, we conclude that $w(\pi/4)=B_s$ and 
\begin{align*}
w'(\pi/4)=\sqrt{\beta^{2}\left(B^{2}_s-w(\pi/4)^{2}\right)+\frac{2}{1-s}\left(B^{1-s}_s-w(\pi/4)^{1-s}   \right) }=0.
\end{align*}
The right-hand side of (4.4$)_1$ is a smooth function for $w$ and hence $w\in C^{\infty}(0,\pi/2)$. 

By the integral representation (4.5) for $-1<\beta<0$, 
\begin{align*}
\lim_{w\to 0}\frac{\phi_s(w)}{w^{\frac{2}{s+1}}}=\left( \frac{1}{s+1} \sqrt{\frac{s-1}{2}} \right)^{\frac{2}{s+1}}.
\end{align*}
Thus $w\in C^{-\beta}[0,\pi/2]$. For $-2<\beta<-1$, $w\in C^{1,2+2/\beta}[0,\pi/2]$. For $-1<\beta<0$, $w'$ diverges at $\phi=0$ and $\pi/2$ as in Remark 4.5.
\end{proof}


\begin{rem}[Symmetry for the autonomous equation]
More generally, we can construct sign-changing solutions to the autonomous equation (4.1) with $c_2=0$ in $(0,\pi)$ by dividing $(0,\pi)$ into $2m$ intervals $[j\pi/2m,(j+1)\pi/2m ]$ for $m\geq 2$ and $j=0,\cdots,2m-1$. The case $m=2$ is the smallest integer for which we can construct solutions for all $-2<\beta<0$. For $m=1$, the condition $\lim_{B\to\infty}\phi(B;B)= (1+s)\pi/4>\pi/2$ imposes the restriction $-1<\beta<0$.  
\end{rem}

\subsection{Outline of construction
in the non-autonomous case }

We outline the construction of solutions to the non-autonomous equation (4.1) for $c_2\neq 0$. We set the operator $L_{\beta}$, the domain $D$, and the function $g$ by
\begin{equation}
\begin{aligned}
\label{singular}
L_{\beta}&=\partial_{\phi}^{2}+\beta^{2},\\
D&=(0,\pi/2),\\
g(\phi,w)&=c_1w^{-s}+c_2\sin\phi w^{-s'},
\end{aligned}
\end{equation}
and recast the problem (4.1) as 
\begin{equation}
\begin{aligned}
\label{recast}
-L_{\beta}w&=g(\phi,w), \quad
w>0\quad \textrm{in}\ D,\\
w&=0\quad \textrm{on}\ \partial D.
\end{aligned}
\end{equation}
 The properties of the operator $-L_{\beta}$ on $D=(0,\pi/2)$ are parallel to those on $(0,\pi)$ demonstrated in Section 3:
\begin{prop}
The eigenvalues of the operator $-L_{\beta}=-\partial_{\phi}^{2}-\beta^{2}$ on $D=(0,\pi/2)$ are $\mu_k=4k^{2}-\beta^{2}$ and the eigenfunctions $e_{k}=(2/\sqrt{\pi})\sin(2k \phi)$ are orthonormal basis on $L^{2}(D)$. 
\end{prop}

\subsubsection{Bilinear form estimates}

The bilinear form associated with the operator $-L_{\beta}$ is 
\begin{equation*}
B(w,\eta)=\int_{D}(w'\eta'-\beta^{2}w\eta)\dd \phi,\quad w,\eta \in H^{1}_{0}(D).
\end{equation*}

The same arguments we used to prove (3.22) and (3.23) give the following result.
\begin{lem}
For $-2<\beta<0$, the bilinear form satisfies 
\begin{equation}
\label{s4Bcoer}
B(w,w)\geq \mu_1 ||w||_{L^{2}}^{2},\quad w\in H^{1}_{0}(D).
\end{equation}
Moreover, there exists a constant $\delta>0$ such that 
\begin{equation}
\label{s4Bcoer2}
B(w,w)\geq \delta ||w||_{H^{1}_{0}}^{2},\quad w\in H^{1}_{0}(D).
\end{equation}
\end{lem}

\subsubsection{Pointwise estimates}

We construct solutions to the singular elliptic problem (4.7) based on the argument of del Pino \cite{delPino} using the maximum principle; see Remark 4.14. We prepare the following maximum principle for the operator $L_{\beta}$.
 
\begin{lem}[Maximum principle]
Let $-2<\beta<0$. Let $c(\phi)\geq 0$ be a non-negative and locally bounded function in $D$. Suppose that $w\in C^{2}(D)\cap C(\overline{D})$ satisfies 
\begin{equation*}
-L_{\beta}w+c(\phi)w\geq 0,\quad \phi\in D.
\end{equation*}
Then, $w\geq \inf_{\partial D}w$. If $\inf_{\partial D}w=0$, $w$ is positive in $D$.  
\end{lem}

\begin{proof}
For $\varphi=\sin2\phi$, the function $v=w/\varphi$ satisfies
\begin{align*}
0\leq -Lw+c(\phi)w=-w''-\beta^{2}w+c(\phi)w
=\varphi\left(-v''-2(\log\varphi)'v'+(4-\beta^{2}+c(\phi))v \right).
\end{align*}
Since $4-\beta^{2}+c(\phi)\geq 0$
and $\sin 2\phi/\sin 2\varepsilon \geq 1,$ the classical maximum principle \cite[p.6, Theorem 3]{PW} implies  
\begin{equation*}
\frac{w}{\varphi}=v\geq \inf_{\partial D_{\varepsilon}}v=\frac{1}{\sin 2\varepsilon}\min\{w(\varepsilon), w(\pi/2-\varepsilon)\},\quad \phi \in D_{\varepsilon}=(\varepsilon,\pi/2-\varepsilon).
\end{equation*}
Since $\min\{w(\varepsilon), w(\pi/2-\varepsilon)\}$ is continuous for $\varepsilon\!>\!0$, 
letting $\varepsilon\to 0$ gives $w\geq \inf_{\partial D}w$. 

We assume that $\inf_{\partial D}w=0$. If $w\geq 0$ takes a minimum in $D$, there exists $\phi_0\in D$ such that $w(\phi_0)=0$. By the strong maximum principle \cite[p.6, Theorem 3]{PW}, the function $v=w/\varphi\geq 0$ does not take a minimum in $D$. This contradicts $v(\phi_0)=0$. Thus, $w$ is positive in $D$.   
\end{proof}

We will use the maximum principle to derive an a priori estimate of the form  
\begin{align}
a \sin2\phi\leq w(\phi)\leq b (\sin2\phi)^{\sigma},\quad \phi\in D,
\end{align}
for solutions $w\in C^{\infty}(D)\cap C(\overline{D})$ to (4.7) with constants $\beta^{2}/4<\sigma\leq \min\{1, -2\beta/3\}$ satisfying $\sigma\neq 1$ and $a=a(\beta)$ and $b=b(\beta)$ (Lemma 4.12). The maximum principle also implies the uniqueness of solutions.

\subsubsection{Existence of solutions}

We define the bounded function 
\begin{align}
g_{\varepsilon}(\phi,t)=c_1(t_{+}+\varepsilon)^{-s}+c_2\sin\phi (t_{+}+\varepsilon)^{-s'},\quad \varepsilon>0,
\end{align}
and construct solutions to the approximate problem 
\begin{equation}
\begin{aligned}
-L_{\beta}w&=g_{\varepsilon}(\phi,w),\quad \textrm{in} D,\\
w&=0\quad \textrm{on}\ \partial D.
\end{aligned}
\end{equation}
The functional associated with this problem is 
\begin{align}
I[w]=\frac{1}{2}B(w,w)-J[w],\quad J[w]=\int_{D}G(\phi,w)\dd \phi,\quad 
G(\phi,w)=\int_{0}^{w}g_{\varepsilon}(\phi ,t)\dd t.
\end{align}
The functional $I$ is Fr\'echet differentiable, i.e., $I\in C^{1}(H^{1}_{0}(D); \mathbb{R} ) $, and has
Fr\'echet derivative
\begin{equation}
\langle I'[w],\eta\rangle=B(w,\eta)-(g(\phi,w),\eta)_{L^{2}},\quad \eta\in H^{1}_{0}(D).
\end{equation}
We find critical points of (4.13) by minimizing the functional $I[\cdot]$ on $H^{1}_{0}(D)$ using the bilinear form estimates (4.8) and (4.9) and construct smooth solutions $w\in C^{\infty}(\overline{D})\cap C_0(\overline{D})$ to (4.12). Solutions of (4.12) are positive by the maximum principle.  
We then apply the estimate (4.10) for approximate solutions and obtain solutions to (4.7) by letting $\varepsilon\to 0$.
\subsection{Pointwise estimates}

We first prove an a priori estimate for solutions to (4.7).

\begin{prop}
Let $(\beta,c_1,c_2)$ satisfy (4.3). Let $\varphi=\sin2\phi$. There exists a constant $a=a(\beta)>0$ such that 
\begin{equation}
-L_{\beta}(a\varphi)\leq g(\phi,a\varphi),\quad \phi\in D.
\end{equation}
\end{prop}
\begin{proof}
For any $\sigma>0$, the function $\varphi^{\sigma}$ satisfies 
\begin{equation}
-L_{\beta}\varphi^{\sigma}=
\sigma(1-\sigma)\varphi^{\sigma-2}|\varphi'|^{2}
+(4\sigma-\beta^{2}) \varphi^{\sigma}.
\end{equation}
We take $\sigma=1$ and observe that  
\begin{align*}
L_{\beta}(a\varphi)+g(\phi,a\varphi)
=c_1\frac{1}{a^{s}\varphi^{s}}+c_2\sin\phi \frac{1}{a^{s'}\varphi^{s'}}-a(4-\beta^{2})\varphi. 
\end{align*}
By (4.3), $c_1\neq 0$ or $c_2\neq 0$. For $c_1\neq 0$, 
\begin{align*}
L_{\beta}(a\varphi)+g(\phi,a\varphi)
\geq c_1\frac{1}{a^{s}\varphi^{s}}-a(4-\beta^{2})\varphi
\geq \frac{1}{a^{s}\varphi^{s}}\left(c_1-a^{s+1}\left(4-\beta^{2}\right) \right).
\end{align*}
For $c_2\neq 0$, we use $\varphi\leq 2\sin\phi$ and estimate 
\begin{align*}
L_{\beta}(a\varphi)+g(\phi,a\varphi)
\geq c_2\sin\phi \frac{1}{a^{s'}\varphi^{s'}}-a(4-\beta^{2})\varphi
&=\frac{\sin\phi}{a^{s'}\varphi^{s'}}\left(c_2-2a^{s'+1}(4-\beta^{2})\varphi^{s'}\cos\phi \right) \\
&\geq \frac{\sin\phi}{a^{s'}\varphi^{s'}}\left(c_2-2a^{s'+1}(4-\beta^{2}) \right). 
\end{align*}\\
Thus (4.15) holds with small $a=a(\beta)>0$.
\end{proof}


\begin{prop}
Let $(\beta,c_1,c_2)$ satisfy (4.3). Let $\beta^{2}/4<\sigma\leq \min\{1,-2\beta/3\}$ be a constant satisfying $\sigma\neq 1$. There exists a constant $b=b(\beta)>0$ such that 
\begin{equation}
-L_{\beta}(b\varphi^{\sigma})\geq g(\phi, b \varphi^{\sigma}),\quad \phi\in D.
\end{equation}
\end{prop}

\begin{proof}
By (4.16), 
\begin{align*}
-L_{\beta}(b\varphi^{\sigma})-g(\phi,b\varphi^{\sigma})
&=\frac{b}{\varphi^{2-\sigma}}\left(\sigma (1-\sigma)|\varphi'|^{2}+(4\sigma-\beta^{2})\varphi^{2}   \right)-c_1 \frac{1}{b^{s}\varphi^{\sigma s}}-c_2\sin\phi \frac{1}{b^{s'}\varphi^{\sigma s'}} \\
&\geq \frac{b}{\varphi^{2-\sigma}} \min\{4\sigma (1-\sigma),(4\sigma-\beta^{2})\}   -c_1 \frac{1}{b^{s}\varphi^{\sigma s}}-c_2\frac{1}{b^{s'}\varphi^{\sigma s'}}.
\end{align*}
Since $\beta^{2}/4<\sigma<1$, the first term is positive, and since $\sigma\leq -2\beta /3$, 
\begin{align*}
2-\sigma\geq \sigma s' \geq \sigma s.
\end{align*}
Since $0\leq \varphi\leq 1$, $\varphi^{2-\sigma}\leq \varphi^{\sigma s'}\leq \varphi^{\sigma s}$. Thus (4.17) holds for large $b=b(\beta)>0$. 
\end{proof}


\begin{lem}
Let $(\beta, c_1, c_2)$ satisfy (4.3). Let 
$\sigma \neq 1$
satisfy $\beta^{2}/4<\sigma\leq \min\{1, -2\beta/3\}$. There exist constants $a=a(\beta)$ and $b=b(\beta)$ such that 
\begin{align}
a \sin2\phi\leq w(\phi)\leq b (\sin2\phi)^{\sigma},\quad \phi\in D,
\end{align}
hold for positive solutions $w\in C^{\infty}(D)\cap C_{0}(\overline{D})$ to (4.7).
\end{lem}

\begin{proof}
We set $v=w-a\varphi$ with the constant $a$ in Proposition 4.10. Then by (4.15),
\begin{equation*}
-L_{\beta}v-g(\phi,w)+g(\phi,a \varphi)\geq 0. 
\end{equation*}
For points $\phi\in D$ where $w(\phi)=a\varphi(\phi)$, we have $-(L_{\beta}v)(\phi)\geq 0$. For points $\phi\in D$ where $w(\phi)\neq a\varphi(\phi)$,  
\begin{align*}
g(\phi,w)-g(\phi,a\varphi)
&=c_1w^{-s}+c_2\sin\phi w^{-s'}-c_1 (a\varphi)^{-s}-c_2\sin\phi (a\varphi)^{-s'} \\
&=c_1\frac{w^{-s}-(a\varphi)^{-s} }{w-a\varphi}v+c_2\sin\phi \frac{w^{-s'}-(a\varphi)^{-s'} }{w-a\varphi}v 
=: -\tilde{c}(\phi)v.
\end{align*}
The function $\tilde{c}(\phi)$ is non-negative and locally bounded in $\{w\neq a\varphi\} \cap D$ since $t^{-s}$ and $t^{-s'}$ are decreasing for $t>0$. We set a non-negative and locally bounded function $c(\phi)$ in $D$ by $c(\phi)=0$ on $\{w= a\varphi\}$ and $c(\phi)=\tilde{c}(\phi)$ on $\{w\neq a\varphi\}$ so that $-L_{\beta}v+c(\phi)v\geq 0$ and conclude that $v\geq 0$ by the maximum principle.

Similarly, we set $v=b\varphi^{\sigma}-w$ with the constants $b$ and $\sigma$ as in Proposition 4.11. Then, by (4.16)
\begin{align*}
-L_{\beta}v-g(\phi,b\varphi^{\sigma})+g(\phi,w)\geq 0. 
\end{align*}
For points $\phi\in D$ where $w(\phi)=b\varphi^{\sigma}(\phi)$, we have $-(L_{\beta}v)(\phi)\geq 0$. For points $\phi\in D$ where $w(\phi)\neq b\varphi^{\sigma}(\phi)$, 
\begin{align*}
g(\phi,b\varphi^{\sigma})-g(\phi,w)
&=c_1(b\varphi^{\sigma})^{-s}+c_2\sin\phi (b\varphi^{\sigma})^{-s'}-c_1 w^{-s}-c_2\sin\phi w^{-s'} \\
&=c_1\frac{(b\varphi^{\sigma})^{-s}-w^{-s} }{b\varphi^{\sigma}-w}v+c_2\sin\phi \frac{(b\varphi^{\sigma})^{-s'}-w^{-s'} }{b\varphi^{\sigma}-w}v 
=: -\tilde{c}(\phi)v.
\end{align*}
The function $\tilde{c}(\phi)$ is non-negative and locally bounded in $\{w\neq b\varphi^{\sigma}\} \cap D$. We set a non-negative and locally bounded function $c(\phi)$ in $D$ by $c(\phi)=0$ on $\{w= b\varphi^{\sigma}\}$ and $c(\phi)=\tilde{c}(\phi)$ on $\{w\neq b\varphi^{\sigma}\}$ so that $-L_{\beta}v+c(\phi)v\geq 0$ and conclude that $v\geq 0$ by the maximum principle.
\end{proof}

\begin{lem}[Uniqueness]
Positive solutions $w\in C^{\infty}(D)\cap C_0(\overline{D})$ of (4.7) are unique. 
\end{lem}

\begin{proof}
We set $v=w_1-w_2$ for two solutions $w_1$ and $w_2$ of (4.7). For points where $w_1(\phi)=w_2(\phi)$, we have $-(L_{\beta}v)(\phi)= 0$. For points $\phi\in D$ where $w_1(\phi)\neq w_2(\phi)$,
\begin{align*}
-L_{\beta}v=g(\phi,w_1)-g(\phi,w_2)
&=c_1w_1^{-s}+c_2\sin\phi w_1^{-s'}-c_1w_2^{-s}-c_2\sin\phi w_2^{-s'} \\
&=c_1\frac{w_1^{-s}-w_2^{-s} }{w_1-w_2}v+c_2\sin\phi \frac{w_1^{-s'}-w_2^{-s'} }{w_1-w_2}v=:-\tilde{c}(\phi)v.
\end{align*}
The function $\tilde{c}(\phi)$ is non-negative and locally bounded in $\{w_1\neq w_2\} \cap D$. We set a non-negative and locally bounded function $c(\phi)$ in $D$ by $c(\phi)=0$ on $\{w_1= w_2\}$ and $c(\phi)=\tilde{c}(\phi)$ on $\{w_1\neq w_2\}$ so that $-L_{\beta}v+c(\phi)v\geq 0$ and conclude that $v\geq 0$ by the maximum principle. Applying the maximum principle to $-v$, it follows that $v=0$.
\end{proof}


\if{

\begin{lem}
Let $(\beta,c_1,c_2)$ satisfy (3.4). Assume that $-1<\beta<0$. Then, there exists a constant $a=a(\beta)>0$ such that 

\begin{align}
a(\sin2\phi)^{-\beta}\leq w,\quad \phi\in D,
\end{align}\\
for positive solutions $w\in C^{\infty}(D)\cap C(\overline{D})$ of (4.7). 
\end{lem}

\begin{proof}
It suffices to show 
\begin{align}
-L_{\beta}(a\varphi^{\sigma})\leq g(\phi,a\varphi^{\sigma}),\quad \phi\in D.
\end{align}\\
for $\sigma=-\beta$ and small constant $a>0$. Then, the lower bound (4.18) follows the maximum principle as in Lemma 4.12. By the same computation as Proposition 4.11,

\begin{align*}
L_{\beta}(a\varphi^{\sigma})+g(\phi,a\varphi^{\sigma})
&=c_1 \frac{1}{a^{s}\varphi^{\sigma s}}+c_2\sin\phi \frac{1}{a^{s'}\varphi^{\sigma s'}}-\frac{a}{\varphi^{2-\sigma}}\left(\sigma (1-\sigma)|\varphi'|^{2}+(4\sigma-\beta^{2})\varphi^{2}   \right) \\
&\geq c_1 \frac{1}{a^{s}\varphi^{\sigma s}}+c_2\sin\phi\frac{1}{a^{s'}\varphi^{\sigma s'}}
-\frac{a}{\varphi^{2-\sigma}} \max\{4\sigma (1-\sigma),(4\sigma-\beta^{2})\}.
\end{align*}\\
We set $D=D_1\cup D_2=(0,\pi/4)\cup [\pi/4,\pi/2)$. Since $\sigma =-\beta$, 

\begin{align*}
\sigma s'+\sigma-2=1.
\end{align*}\\
Thus on $D_1$, 

\begin{align*}
L(a\varphi^{\sigma})+g(\phi,a\varphi^{\sigma})
&\geq c_1\frac{1}{a^{s}\varphi^{\sigma s}}
+\frac{1}{\varphi^{2-\sigma}}\left(\frac{c_2\sin\phi}{a^{s'}\varphi}-a\max\{4\sigma (1-\sigma),(4\sigma-\beta^{2})\} \right) \\
&\geq c_1\frac{1}{a^{s}\varphi^{\sigma s}}
+\frac{1}{\varphi^{2-\sigma}}\left(\frac{c_2}{2a^{s'}}-a\max\{4\sigma (1-\sigma),(4\sigma-\beta^{2})\} \right).
\end{align*}\\
The right-hand side is non-negative for $a\leq a_1$ and some $a_1>0$. On $D_2$, we use $\sin\phi\geq 1/\sqrt{2}$. By $\sigma-2=\sigma \leq \sigma s'$ and $0\leq \varphi\leq 1$, $\varphi^{\sigma-2}\geq \varphi^{\sigma s'}$ and 

\begin{align*}
L_{\beta}(a\varphi^{\sigma})+g(\phi,a\varphi^{\sigma})
\geq \frac{1}{\varphi^{\sigma-2}}\left( \frac{c_1}{a^{s}}+\frac{c_2}{\sqrt{2}a^{s'}}-a \max\{4\sigma (1-\sigma),(4\sigma-\beta^{2})\}\right).
\end{align*} \\
The right-hand side is non-negative for $a\geq a_2$ and some $a_2>0$. Thus (4.19) holds for $a\leq \min\{a_1,a_2\}$. 
\end{proof}

}\fi


\begin{rem}[Singular elliptic problems]
The above problem is the one-dimensional analog
of the singular elliptic problem
\begin{equation}
\begin{aligned}
-\Delta \psi &=k(x)\psi^{-\nu}\quad \textrm{in}\ \Omega,\\
\psi&=0\quad \textrm{on}\ \partial \Omega,
\end{aligned}
\end{equation}
for a non-negative function $k$ in a bounded domain $\Omega\subset \mathbb{R}^{n}$, $n\geq 2$ with constant $\nu>0$; see \cite{Fulks}, \cite{Stu}, and \cite{CRT} for the origin of the problem (4.19). For positive $k(x)>0$ in $\overline{\Omega} $, $\nu=1$ is the threshold for $C^{1}$ regularity on the boundary. For $\nu>1$, $\psi\notin C^{1}(\overline{\Omega}) $ and $\psi\in C^{2/(\nu+1)}(\overline{\Omega}) $ \cite{Lazer}. For functions $k(x)$ vanishing on the boundary, the regularity of solutions has been investigated in \cite{delPino}, \cite{GuiLin} under the condition, e.g., $k(x) \approx \textrm{dist}(x,\partial \Omega) $, cf. \cite[Theorem 8]{FQS}. The $C^{1}$-threshold for such $k$ is $\nu=2$ and for $\nu>2$, $\psi\notin C^{1}(\overline{\Omega}) $ and $\psi\in C^{3/(\nu+1)}(\overline{\Omega}) $. See \cite{Oliva} for a survey and also \cite{FQS}, \cite{Biri} for recent works on singular problems of fully nonlinear equations. 

The singular 1D problem (4.7) is similar to the classical problem (4.19), though the fact that
the function $k(\phi)=\sin\phi$ vanishes only at one side of the boundary $\phi=0$ leads to different regularity at $\phi=0$ and $\pi/2$. 
\end{rem}

\begin{rem}[The first derivative at the boundary]
The identity 
 \begin{align*}
|w'|^{2}+\beta^{2}w^{2}+\frac{2c_1}{1-s}w^{1-s}+\frac{2c_2}{1-s'}\sin\phi w^{1-s'}+\frac{2c_2}{1-s'}\int_{\phi}^{\pi/4}\cos\tilde{\phi} w^{1-s'} \dd \tilde{\phi}=C,\quad \phi\in D,
 \end{align*}
holds for solutions to (4.7) with some constant $C$ by multiplying $2w'$ by (4.7) and integrating it on $(\pi/4,\phi)$. For $-1<\beta<0$, $1<s<s'$ and $w'$ diverges on the boundary $\phi=0$ and $\pi/2$. For $-2<\beta<-3/2$, $s<s'<1$ and $w\in C^{1}(\overline{D})$. Moreover, $w'\in C^{2+2/\beta}(\overline{D})$.   
\end{rem}

\subsection{Existence for the non-autonomous equation}
We now solve the approximate problem (4.12).

\begin{prop}
Let $\delta>0$ be the constant in Lemma 4.8. The functional $I\in C^{1}(H^{1}_{0}(D); \mathbb{R})$ in (4.13) satisfies the following inequalities: 
\begin{align}
\frac{\delta}{4}||w||_{H^{1}_{0}}^{2}\leq I[w]+\frac{\pi}{8\delta}\left(c_1\varepsilon^{-s}+c_2\varepsilon^{-s'} \right)^{2},\quad w\in H^{1}_{0}(D),\\
\label{s4diffbd}
|J[w]-J[\eta]|\leq \left(c_1\varepsilon^{-s}+c_2\varepsilon^{-s'} \right) ||w-\eta||_{L^{1}},\quad w,\eta\in H^{1}_{0}(D).
\end{align}
\end{prop}

\begin{proof}
The inequality (4.8) is nothing but the Poincar\'e inequality
\begin{equation*}
2||w||_{L^{2}}\leq ||w||_{H^{1}_{0}},\quad w\in H^{1}_{0}(D).
\end{equation*}
Since $g_{\varepsilon}(\phi,t)\leq c_1\varepsilon^{-s}+c_2\varepsilon^{-s'}$, by the Poincar\'e inequality,
\begin{align*}
J[w]\leq \left(c_1\varepsilon^{-s}+c_2\varepsilon^{-s'} \right)||w||_{L^{1}}
\leq \sqrt{\frac{\pi}{2}}\left(c_1\varepsilon^{-s}+c_2\varepsilon^{-s'} \right) ||w||_{L^{2}} 
&\leq \frac{\sqrt{\pi}}{2\sqrt{2}}\left(c_1\varepsilon^{-s}+c_2\varepsilon^{-s'} \right) ||w||_{H^{1}_{0}} \\
&\leq \frac{\pi}{8\delta }\left(c_1\varepsilon^{-s}+c_2\varepsilon^{-s'} \right)^{2}+ \frac{\delta}{4}||w||_{H^{1}_{0}}^{2}.
\end{align*}
By the lower bound for the bilinear form (4.9),
\begin{align*}
\frac{\delta}{2}||w||_{H^{1}_{0}}^{2}\leq \frac{1}{2}B(w,w)=I[w]+J[w].
\end{align*}
By subtracting $(\delta/4)||w||_{H^{1}_{0}}^{2}$ from both sides, the estimate (4.20) follows. Bounding
\begin{align*}
|J[w]-J[\eta]|=\left| \int_{D}\int_{\eta}^{w}g(\phi,t)\dd t \dd \phi  \right|
\leq \left(c_1\varepsilon^{-s}+c_2\varepsilon^{-s'} \right) ||w-\eta||_{L^{1}},
\end{align*}
we obtain the estimate (4.21).
\end{proof}

\begin{lem}
Let $(\beta,c_1,c_2)$ satisfy (4.3). For
any $\varepsilon>0$, there is a positive solution $w\in C^{\infty}(\overline{D})\cap C_0(\overline{D})$ to (4.12).
\end{lem}

\begin{proof}
The functional $I$ is bounded from below on $H^{1}_{0}$ by (4.20). We take a minimizing sequence $\{w_n\}\subset H^{1}_{0}$ such that $I[w_n]\to \inf_{\tilde{w}\in H^{1}_{0}}I[\tilde{w}]$. By (4.20), $\{w_n\}$ is bounded in $H^{1}_{0}$. By choosing a subsequence, $w_n\rightharpoonup w$ in $H^{1}_{0}$ and $w_n\to w$ in $C(\overline{D})$. By (4.21), $J[w_n]\to J[w]$ and
\begin{equation*}
\inf_{\tilde{w}\in H^{1}_{0}}I[\tilde{w}]
=\lim_{n\to\infty}I[w_n]
=\lim_{n\to\infty}\left(\frac{1}{2}B(w_n,w_n)-J[w_n] \right)
\geq \frac{1}{2}B(w,w)-J[w]=I[w] \geq \inf_{\tilde{w}\in H^{1}_{0}}I[\tilde{w}]. 
\end{equation*}\\
Thus $w_n\to w$ in $H^{1}_{0}$ and $I[w]=\inf_{\tilde{w}\in H^{1}_{0}}I[\tilde{w}]$. By $I'[w]=0$, $w\in H^{1}_{0}$ is a weak solution to (4.12). Since $g_{\varepsilon}(\phi,w(\phi))\in C(\overline{D})$, $w\in C^{2}(\overline{D})$ is a classical solution. By $g_{\varepsilon}(\phi,w)\geq 0$ and the maximum principle, $w\in C^{2}(\overline{D})$ is positive. The regularity $w\in C^{\infty}(\overline{D})$ follows by differentiating the equation (4.12$)_1$. 
\end{proof}


We now show that the estimate (4.10) holds for solutions to the approximate problem (4.12).


\begin{lem}
Let $(\beta, c_1, c_2)$ satisfy (4.3). Let $\beta^{2}/4<\sigma\leq \min\{1, -2\beta/3\}$ be a constant satisfying $\sigma\neq 1$. There exist constants $a=a(\beta)$ and $b=b(\beta)$ such that 
for all $\phi \in D,$ the bound
\begin{equation}
a \sin2\phi\leq w_{\varepsilon}(\phi)\leq b (\sin2\phi)^{\sigma}
\end{equation}
holds for positive solutions $w_{\varepsilon}\in C^{\infty}(\overline{D})\cap C_0(\overline{D})$ to (4.12) and $0<\varepsilon\leq 1$.
\end{lem}

\begin{proof}
We show (4.22) by a small modification to the proof of Lemma 4.12. Since $g_{\varepsilon}(\phi,t)<g (\phi,t)$ for $t>0$, using (4.17) we have 
\begin{align}
-L_{\beta}(b\varphi^{\sigma})\geq g_{\varepsilon}(\phi, b\varphi^{\sigma}),\quad \phi\in D
\end{align}
with the same constant $b=b(\beta)>0$ in Proposition 4.11. For $a>0$,
\begin{align*}
L_{\beta}(a\varphi)+g_{\varepsilon}(\phi,a\varphi)
=c_1 \frac{1}{(a\varphi+\varepsilon)^{s}}+c_2 \sin\phi \frac{1}{(a\varphi+\varepsilon)^{s'}}-a(4-\beta^{2})\varphi. 
\end{align*}
For $c_1\neq 0$, 
\begin{align*}
L_{\beta}(a\varphi)+g_{\varepsilon}(\phi,a\varphi)
\geq \frac{1}{(a\varphi+\varepsilon)^{s}}\left(c_1-a\left(4-\beta^{2}\right)(a+1)^{s} \right).
\end{align*}
For $c_2\neq 0$,
\begin{align*}
L_{\beta}(a\varphi)+g_{\varepsilon}(\phi,a\varphi)
\geq \frac{\sin\phi}{(a\varphi+\varepsilon)^{s'}}\left(c_2-2a\left(4-\beta^{2}\right)(a+1)^{s'} \right).
\end{align*}
Thus for small $a=a(\beta)>0$ and $0<\varepsilon\leq 1$,
\begin{align}
-L_{\beta}(a\varphi)\leq g_{\varepsilon}(\phi,a\varphi),\quad \phi\in D.
\end{align}
Since the functions $(t+\varepsilon)^{-s}$ and $(t+\varepsilon)^{-s'}$ are decreasing for $t>0$, we apply (4.23) and (4.24) and obtain (4.22) just as in
the proof of Lemma 4.12.
\end{proof}


\begin{proof}[Proof of Theorem 4.3]
The pointwise estimates (4.22) imply that the positive solutions $w_{\varepsilon}\in C^{\infty}(\overline{D})$ to (4.12) constructed in Lemma 4.17 are uniformly bounded in $\varepsilon>0$ and their second derivative is locally bounded in $D$ for all $\varepsilon>0$. Thus, by choosing a subsequence, $w_{\varepsilon}$ converges to a limit $w$ as $\varepsilon\to 0$ locally uniformly in $D$ up to the first order. By the equation (4.12$)_1$, the second derivative also converges to $w$ locally uniformly in $D$, and the limit $w\in C^{2}(D)$ satisfies the equation (4.7$)_1$. Since the limit $w$ satisfies the pointwise estimates (4.22), $w$ is positive and uniquely extendable up to the boundary and vanishes on $\partial D$. Thus, $w\in C^{2}(D)\cap  C_0(\overline{D})$ is a positive solution to (4.7). The regularity $w\in C^{\infty}(D)$ follows by differentiating the equation (4.7$)_1$. The uniqueness follows from Lemma 4.13. 
\end{proof}

\subsection{Homogeneous solution construction}

We now construct ($-\alpha$)-homogeneous solutions to (1.1), (1.3), and (1.4) with singular profiles using the solutions to the elliptic
problem (4.1) constructed in Theorem 4.1.

\begin{prop}
Let $(\beta,c_1,c_2)$ satisfy (4.3). Let $w\in C^{\infty}(0,\pi/2)\cap C_{0}[0,\pi/2]$ be a positive solution to (4.1). Set $(\alpha,C_1,C_2)$ by (3.5). For $\beta=-1$ and $-3/2$, set $C_1$ and $C_2$ by arbitrary constants, respectively. Set $(\psi,\Pi,\rho)$ by (3.2) and (3.3), and extend them to $\mathbb{R}^{2}\backslash \{0\}$ by the symmetries (1.3) and (1.4). Then, $(u,p,\rho)\in C^{\infty}(\mathbb{R}^{2}\backslash \{x_1=0\}\cup \{x_2=0\})$. For $-1<\alpha<-1/2$, $(u,p,\rho)\in C(\mathbb{R}^{2})$. Moreover, the velocity field $u$ is ($-\alpha$)-H\"older continuous in the radial variable and ($2+2/(\alpha-1)$)-H\"older continuous in the angular variable.
\end{prop}

\begin{proof}
The function $w$ satisfies $w\in C^{\infty}(\mathbb{S}^{1}\backslash P)\cap C(\mathbb{S}^{1})$ for $P=\{0,\pm \pi/2,\pi\}$. Observe that on the first quadrant $\{x_1>0,\ x_2>0\}$, 
\begin{equation}
u=\frac{-1}{r^{\alpha}}\left( (\alpha-1)we_{\phi}+w'e_{r}  \right),\quad
\rho=\frac{C_2}{r^{2\alpha +1}}w^{2+\frac{3}{\alpha-1}},\quad \Pi=\frac{C_1}{r^{2\alpha}}w^{2+\frac{2}{\alpha-1}},\quad p=\Pi-\frac{1}{2}|u|^{2}-x_2\rho. 
\end{equation}
Thus $(u,p,\rho)\in C^{\infty}(\mathbb{S}^{1}\backslash P)$. This means $(u,p,\rho)\in C^{\infty}(\mathbb{R}^{2}\backslash \{x_1=0\}\cup \{x_2=0\})$.  

For $-1<\alpha<-1/2$, $w\in C^{1,2+2/\beta}(\mathbb{S}^{1})$ implies $u\in C^{2+2/(\alpha-1)}(\mathbb{S}^{1})$ and $\rho\in C^{2+3/(\alpha-1)}(\mathbb{S}^{1})$. By the identity in Remark 4.15,
\begin{align}
2r^{2\alpha}p=C-\frac{c_2}{\beta+3/2}\int_{\phi}^{\pi/4}\cos\tilde{\phi}w^{2+3/\beta}\dd \tilde{\phi},\quad 0\leq \phi\leq \frac{\pi}{2},
\end{align}\\
and hence $p\in C^{1}(\mathbb{S}^{1})$. We conclude that $(u,p,\rho)\in C(\mathbb{R}^{2})$. The H\"older continuity of $u$ follows from the representation (4.25$)_1$ and $(\alpha-1)we_{\phi}+w'e_{r} \in C^{2+2/(\alpha-1)}(\mathbb{S}^{1})$.
\end{proof}

\begin{prop}
Let $u$ be the velocity field constructed in
Proposition 4.19. The $C^{1}$-norm of $u$ diverges on the axis $\{x_1=0\}$. If $c_1\neq 0$, the $C^{1}$-norm of $u$ diverges on the axis $\{x_1=0\}\cup \{x_2=0\}$. If in addition $c_2=0$ and $0<\alpha<1$, $u$ diverges on the axes $\{x_1=0\}\cup \{x_2=0\}$. 
\end{prop}

\begin{proof}
By Remark 4.2, $w''$ diverges when $\phi=\pi/2$. Thus the $C^{1}$-norm of $u$ diverges on $\{x_1=0\}$. If $c_1\neq 0$, $w''$ also diverges on $\phi=0$ and the $C^{1}$-norm of $u$ diverges on $\{x_2=0\}$. For $c_1\neq 0$, $c_2=0$, and $0<\alpha<1$, $w'$ diverges on $\phi=0$ and $\pi/2$ by Remark 4.5. Thus $u$ diverges on $\{x_1=0\}\cup \{x_2=0\}$.
\end{proof}

\begin{lem}
The functions $(u,p,\rho)$ in Proposition 4.19 are ($-\alpha$)-homogeneous weak solutions to (1.1), (1.3), and (1.4) in $\mathbb{R}^{2}$ for $-1<\alpha< -1/2$ and ($-\alpha$)-homogeneous solutions to (1.1), (1.3), and (1.4) in $\mathbb{R}^{2}\backslash \{x_1=0\}\cup \{x_2=0\}$ for $-1/2\leq \alpha<1$.  
\end{lem}

\begin{proof}
Arguing as in the proof of Lemma 3.19, we see that the functions $(u,p,\rho)\in C^{\infty}(\mathbb{R}^{2}\backslash \{x_1=0\}\cup \{x_2=0\})$ satisfy the equation (1.1) in $\mathbb{R}^{2}\backslash \{x_1=0\}\cup \{x_2=0\}$. For $-1<\alpha<-1/2$, $(u,p,\rho)$ is continuous in $\mathbb{R}^{2}$ and the normal component of $u$ vanish on $\{x_1=0\}\cup\{x_2=0\}$. 

Let $\varphi\in C^{1}_{c}(\mathbb{R}^{2})$
be arbitrary. Since $(u,p,\rho)$ satisfies (1.1) in the first quadrant $G_1=\{x_1>0,\ x_2>0\}$, integrating by parts in $G_{1,\varepsilon}=\{x_1>\varepsilon,\ x_2>\varepsilon \}$ for $\varepsilon>0$, we find
\begin{align*}
\int_{G_{1,\varepsilon}}((u\cdot \nabla \varphi) \cdot u+p\nabla \cdot \varphi)\dd x
-\int_{\partial G_{1,\varepsilon}}\left( u\cdot n (u\cdot \varphi)+p  \varphi \cdot n\right) \dd H
=\int_{G_{1,\varepsilon}}\rho e_{2}\cdot \varphi\dd x.
\end{align*}
Here, $n$ is the unit outward normal vector field on $\partial G_{1,\varepsilon}$. Since $(u,p,\rho)\in C(\mathbb{R}^{2})$, letting $\varepsilon \to 0$ implies 
\begin{align*}
\int_{G_1}((u\cdot \nabla \varphi) \cdot u+p\nabla \cdot \varphi)\dd x
-\int_{\partial G_{1}} p  \varphi \cdot n \dd H
=\int_{G_1}\rho e_{2}\cdot \varphi\dd x.
\end{align*}
Applying the same argument to the second, third, and fourth quadrants and combining the results,
\begin{equation*}
\int_{\mathbb{R}^{2}}(u\cdot \nabla \varphi \cdot u+p\nabla \cdot \varphi)\dd x
=\int_{\mathbb{R}^{2}} \rho e_2\cdot \varphi \dd x,
\end{equation*}
for all $\varphi\in C^{1}_{c}(\mathbb{R}^{2})$. Thus $(u,p,\rho)$ is a weak solution in $\mathbb{R}^{2}$. 
\end{proof}

For $c_2=0$, we deduce properties of ($-\alpha$)-homogeneous solutions from the autonomous equation.

\begin{prop}
In Proposition 4.19, take $-2<\beta<0$ ($\beta\neq -1$), $c_1\neq 0$, and $c_2=0$ ($C_2=0$ for $\beta\neq -3/2$). Set $C_2=0$ for $\beta=-3/2$. Then, $(u,p)\in C^{\infty}(\mathbb{R}^{2}\backslash \{x_1=0\}\cup \{x_2=0\})$ and $\rho=0$. For $-1<\alpha<0$, $(u,p)\in C(\mathbb{R}^{2})$. Moreover, the velocity field $u$ is ($-\alpha$)-H\"older continuous for the radial variable and $(2+2/(\alpha-1))$-H\"older continuous for the angular variable. 
\end{prop}


\begin{proof}
By (4.25) and (4.26), $\rho=0$ and $p$ is constant on $\mathbb{S}^{1}$. Thus $(u,p)\in C^{\infty}(\mathbb{R}^{2}\backslash \{x_1=0\}\cup \{x_2=0\})$. For $-1<\alpha<0$, solutions of the autonomous equation in Theorem 4.3 satisfy $w\in C^{1,2+2/(\alpha-1)}(\mathbb{S}^{1})$. Thus $(u,p)\in C(\mathbb{R}^{2})$ and $u$ is ($2+2/(\alpha-1)$)-H\"older continuous in the angular variable.   
\end{proof}

\begin{proof}[Proof of Theorems 1.3, 1.6, and 1.8]
The existence of ($-\alpha$)-homogeneous solutions $(u,p,\rho)\in C^{\infty}(\mathbb{R}^{2}\backslash \{x_1=0\}\cup \{x_2=0\})\cap C(\mathbb{R}^{2})$ for $-1<\alpha<-1/2$ and $(u,p,\rho)\in C^{\infty}(\mathbb{R}^{2}\backslash \{x_1=0\}\cup \{x_2=0\})$ for $-1/2\leq \alpha<1$ follows
from Proposition 4.19 and Lemma 4.21. By Proposition 4.20, the $C^{1}$-norm of $u$ diverges at $\{x_1=0\}$. This proves the assertions (i) and (ii)
from Theorem 1.3.

We now apply Proposition 4.19 with $-2<\beta<0$ ($\beta\neq -3/2$), $c_1=0$, and $c_2\neq 0$. By (3.5), $C_1=0$ for $\beta\neq -1$. For $\beta=-1$, $C_1$ is arbitrary, and we take $C_1=0$. Then $\Pi=0$ by (4.25) and $(u,p,\rho)$ is a Pseudo-Beltrami solution. Thus, the existence of ($-\alpha$)-homogeneous solutions for $-1<\alpha<1$ ($\alpha\neq -1/2$) in (i) and (ii) from Theorem 1.8 hold.

Finally, to prove the assertions (i) and (ii)
from Theorem 1.6, we apply Proposition 4.19 with $-2<\beta<0$ ($\beta\neq -1$), $c_1\neq 0$, and $c_2= 0$. By (3.5), $C_2=0$ for $\beta\neq -3/2$. For $\beta=-3/2$, $C_2$ is arbitrary, and we take $C_2=0$. Then $\rho=0$ by (4.25), and $(u,p)$ is a 2D Euler solution. 
The existence of ($-\alpha$)-homogeneous solutions for $-1<\alpha<1$ ($\alpha\neq 0$) now follows from Propositions 4.20 and 4.22.
\end{proof}

\section{Desingularization of the ($-1$)-homogeneous vortex sheet}

We finally show the desingularization of the ($-1$)-homogeneous vortex sheet solution by ($-\alpha$)-homogeneous solutions with regular profiles in Theorem 1.10. We show the existence of the stream function $w=w_{\beta}$ of ($-\alpha$)-homogeneous solutions uniformly converging to the stream function of the ($-1$)-homogeneous vortex sheet solution $2\phi/\pi$ in $[0,\pi/2]$.

\begin{prop}
There exist solutions $w=w_{\beta}\in C^{2}(0,\pi/2]\cap C[0,\pi/2]$ of 
\begin{equation}
\begin{aligned}
-w''&=\beta^{2}w+\frac{c_{1,\beta}}{\beta}w^{1+\frac{2}{\beta}}+\frac{c_{2,\beta}}{\beta}\sin\phi w^{1+\frac{3}{\beta}},\\
w&>0,\ w'>0,\quad \phi\in (0,\pi/2),\\
w(0)&=0,\ w(\pi/2)=1,\ w'(\pi/2)=0,
\end{aligned}
\end{equation}
such that $w_{\beta}$ converges to $2\phi /\pi$ uniformly in $[0,\pi/2]$ as $\beta\to 0$.
\end{prop}

\begin{proof}
For $c_1\geq 0$, $c_2\geq 0$ satisfying $c_1\neq 0$ or $c_2\neq 0$, and $0<\beta<1$, there exists a positive solution $w\in C^{3}[0,\pi]\cap C_0[0,\pi]$ to (3.4) such that $w$ is even symmetric for $\phi=\pi/2$ and increasing in $(0,\pi/2)$ by Theorem 3.1. The function $w$ takes a maximum at $\phi=\pi/2$ and $w'(\pi/2)=0$. We normalize $w$ by $w(\pi/2)$ and denote it again by $w$. Then for some constants $c_{1,\beta},\ c_{2,\beta}\geq 0$, $w=w_{\beta}$ satisfies the equations (5.1).

By multiplying $2w'$ by $(5.1)_1$, and by using the conditions $(5.1)_3$ at $\phi=\pi/2$, we obtain the identity:  
\begin{equation}
\begin{aligned}
|w'|^{2}&=\beta^{2}\left(1-w^{2}\right)+\frac{c_{1,\beta}}{\beta+1}\left(1-w^{2+\frac{2}{\beta}}\right)
+\frac{c_{2,\beta}}{\beta+3/2}\left(1-(\sin\phi)w^{2+\frac{3}{\beta}}-f_{\beta}(\phi) \right),\quad \phi\in [0,\pi/2],\\
f_{\beta}(\phi)&=\int_{\phi}^{\pi/2}\cos\tilde{\phi} w^{2+\frac{3}{\beta}}\dd \tilde{\phi}.
\end{aligned}
\end{equation}
By integrating the equation $(5.1)_1$ in $(0,\pi/2)$, we obtain the trace condition of $w'$ at $\phi=0$:
\begin{align}
w'(0)&=\beta^{2}\int_{0}^{\pi/2}w\dd \phi+\frac{c_{1,\beta}}{\beta}\int_{0}^{\pi/2}w^{1+\frac{2}{\beta}}\dd \phi+\frac{c_{2,\beta}}{\beta}\int_{0}^{\pi/2}(\sin\phi)w^{1+\frac{3}{\beta}}\dd \phi.  
\end{align}


We show that the constants $c_{1,\beta}$ and $c_{2,\beta}$ are bounded as $\beta\to 0$. We use the Taylor expansion
\begin{align}
\int_{0}^{1}\frac{\dd \sigma}{\sqrt{1-\sigma^{s}}}=1+\sum_{n=1}^{\infty}\frac{(2n-1)!!}{(2n)!!(sn+1)},\quad s>0.
\end{align}
Since each term on the right-hand side of $(5.2)_1$ are non-negative, we integrate the inverse function of $w(\phi)$ in $(0,1)$ and obtain 
\begin{align}
\frac{\pi}{2}\sqrt{\frac{c_{1,\beta}}{\beta+1}}\leq 1+\sum_{n=1}^{\infty}\frac{(2n-1)!!}{(2n)!! (2n+1)}.
\end{align}
Thus, $c_{1,\beta}$ is uniformly bounded. 

We show that $c_{2,\beta}$ is uniformly bounded. We first show the inequality 
\begin{align}
1-(\sin\phi) w^{2+\frac{3}{\beta}}-f_{\beta}(\phi)\geq \left(1-f_{\beta}(0) \right)\left(1-w^{2+\frac{3}{\beta}}\right),\quad \phi\in [0,\pi/2].
\end{align}
The function $f_{\beta}$ is decreasing in $[0,\pi/2]$ by $f_{\beta}'(\phi)=-\cos\phi w^{2+3/\beta}<0$. By $(1-f_{\beta}-\sin\phi)'=-\cos\phi (1-w^{2+3/\beta})<0$, 
\begin{align*}
1-f_{\beta}(\phi)-\sin\phi\geq 1-f_{\beta}\left(\frac{\pi}{2}\right)-1=0.
\end{align*}
Thus 
\begin{equation*}
\begin{aligned}
1-\sin\phi w^{2+\frac{3}{\beta}}-f_{\beta}(\phi)\geq 1-f_{\beta}(\phi)-\left(1-f_{\beta}(\phi)\right)w^{2+\frac{3}{\beta}}
=\left(1-f_{\beta}(\phi)\right) \left(1-w^{2+\frac{3}{\beta}}\right) 
\geq \left(1-f_{\beta}(0)\right) \left(1-w^{2+\frac{3}{\beta}}\right),
\end{aligned}
\end{equation*}
and the inequality (5.6) holds.

By using the inequality (5.6) and the identity (5.2), we integrate the inverse function of $w(\phi)$ in $(0,1)$ and obtain the upper bound
\begin{align}
\frac{\pi}{2}\sqrt{\frac{c_{2,\beta}}{\beta+3/2}}\leq \frac{1}{\sqrt{1-f_{\beta}(0) }}\left(1+\sum_{n=1}^{\infty}\frac{(2n-1)!!}{(2n)!! (2n+1)}\right).
\end{align}
Suppose that $c_{2,\beta}$ diverges as $\beta\to 0$. Then $\beta w'(0)/c_{2,\beta}$ vanishes by the identity $(5.2)_1$. By the trace condition (5.3),   
\begin{align*}
\int_{0}^{\pi/2}\sin\phi \left(\limsup_{\beta\to 0}w^{1+\frac{3}{\beta}}\right) \dd \phi=0.
\end{align*}
Thus $\lim_{\beta\to 0}w^{1+3/\beta}=0$. This implies that $\lim_{\beta\to 0}f_{\beta}(0)=0$ and $c_{2,\beta}$ is bounded by (5.7). This is a contradiction. Thus $c_{2,\beta}$ is uniformly bounded for $0<\beta\leq 1$. 

We now show that $w=w_{\beta}$ subsequently converges to the limit $2\phi/\pi$. By $(5.1)_1$ and (5.2), $w'$ is decreasing and 
\begin{align*}
0\leq w'(\phi)\leq w'(0)=\beta^{2}+\frac{c_{1,\beta}}{\beta+1}+\frac{c_{2,\beta}}{\beta+3/2}(1-f_{\beta}(0) ),\quad \phi\in [0,\pi/2].
\end{align*}
Since $c_{1,\beta}$ and $c_{2,\beta}$ are bounded for $0<\beta\leq 1$, $w'$ is uniformly bounded in $[0,\pi/2]$. By choosing a subsequence, $w$ converges to a limit $w_0$ uniformly in $[0,\pi/2]$. By choosing a subsequence, we may assume that $c_{1,\beta} $ and $c_{2,\beta}$ converge to non-negative constants $c_{1,0}$ and $c_{2,0}$, respectively. If $c_{1,0}=0$ and $c_{2,0}=0$, letting $\beta\to 0$ to the above inequality implies that $w_0\in C[0,\pi/2]$ is a constant. This contradicts the boundary conditions $w_0(0)=0$ and $w_0(\pi/2)=1$. 

We may assume that $c_{1,0}\neq 0$ or $c_{2,0}\neq 0$. By letting $\beta\to 0$ to the identity (5.3), 
\begin{equation*}
c_{1,0}\int_{0}^{\pi/2}\left(\limsup_{\beta\to 0}w^{1+\frac{2}{\beta}}\right)\dd \phi+c_{2,0}\int_{0}^{\pi/2}\sin\phi \left(\limsup_{\beta\to 0} w^{1+\frac{3}{\beta}}\right)\dd \phi  
=0.  
\end{equation*}
By $w^{1+3/\beta}= w^{1+2/\beta}w^{1/\beta}\leq w^{1+2/\beta}$, $\lim_{\beta\to 0}w^{1+3/\beta}=0$ holds. In a similar way, $w^{2+3/\beta}$ and $w^{2+2/\beta}$ also converge to zero for each $\phi\in [0,\pi/2)$. Since $w$ is increasing, $w^{2+3/\beta}$ and $w^{2+2/\beta}$ converge to zero locally uniformly in $[0,\pi/2)$. This particularly implies that $\lim_{\beta\to 0}f_{\beta}(0)=0$. We take arbitrary $\delta\in (0,\pi/2)$. By the identity (5.2),

\begin{align*}
\sup_{[0,\pi/2-\delta]}\left| |w'|^{2}-\frac{c_{1,\beta}}{\beta+1}-\frac{c_{2,\beta}}{\beta+3/2}  \right|
\leq \beta^{2}+\frac{c_{1,\beta}}{\beta+1}\sup_{[0,\pi/2-\delta]}w^{2+\frac{2}{\beta}}
+\frac{c_{2,\beta}}{\beta+3/2}\left(\sup_{[0,\pi/2-\delta]}w^{2+\frac{3}{\beta}}
+f_{\beta}(0)
\right).
\end{align*}
The right-hand side converges to zero. Thus $w'$ converges to $\sqrt{c_{1,0}+2c_{2,0}/3 }$ locally uniformly in $[0,\pi/2)$ and the limit $w_0$ is a linear function. By the boundary conditions $w_0(0)=0$ and $w_0(\pi/2)=1$, we conclude that $w_0=2\phi/\pi$. 
\end{proof}


\begin{proof}[Proof of Theorem 1.10]
Let $(u_{\textrm{sh}}, p_{\textrm{sh}})$ be the vortex sheet solution in (1.9). For $\alpha=\beta+1$, $C_{1,\beta}=-c_\beta/(2\beta+2)$, and $C_{2,\beta}=c_{2,\beta}/(2\beta+3)$, we set $(u_\alpha,p_{\alpha})$ by (3.27). Then, 
\begin{align*}
-r^{\alpha} u_{\alpha}&=\beta w_{\beta}e_{\phi}+w'_{\beta}e_{r},\\
-2r^{2\alpha} p_{\alpha}&= |w_{\beta}'|^{2}+\beta^{2}w_{\beta}^{2}+\frac{c_{1,\beta}}{\beta+1}w^{2+\frac{2}{\beta}}+\frac{c_{2,\beta}}{\beta+3/2}\sin\phi w_{\beta}^{2+\frac{3}{\beta}},\\
r^{2\alpha+1}\rho_{\alpha}&=\frac{c_{2,\beta}}{2\beta+3}w_{\beta}^{2+\frac{3}{\beta}}.
\end{align*}
By Lemma 3.19, $(u_{\alpha},p_{\alpha})\in C^{2}(\mathbb{R}^{2}\backslash \{0\}) $ are $(-\alpha)$-homogeneous solution to the Boussinesq equations satisfying symmetry (1.4). Since $w_{\beta}'$ converges to $2/\pi$, and $w^{2+2/\beta}$ and $w^{2+3/\beta}$ converge to zero locally uniformly in $[0,\pi/2)$, respectively, $r^{\alpha}u_{\alpha}$ converges to $(4/\pi) r u_{\textrm{sh}}$, $p_{\alpha}$ converges to $(4/\pi)^{2} p_{\textrm{sh}}$, and $r^{2\alpha+1}\rho_{\alpha}$ converges to zero locally uniformly for $\phi\in [0,\pi/2)$. For $\phi=\pi/2$, 
\begin{align*}
-2r^{2\alpha }p_{\alpha}=\beta^{2}+\frac{c_{1,\beta}}{\beta+1}+\frac{c_{2,\beta}}{\beta+3/2}\to c_{1,0}+\frac{2}{3}c_{2,0}=\frac{4}{\pi^{2}}.
\end{align*}
Thus $r^{2\alpha } p_{\alpha}$ converges to $(4/\pi)^{2} r^{2 }p_{\textrm{sh}}$ for $\phi=\pi/2$. 

By symmetry for $x_1$ and $x_2$, $r^{\alpha}u_{\alpha}\to (4/\pi) r u_{\textrm{sh}}$, $r^{2\alpha}p_{\alpha}\to (4/\pi)^{2} p_{\textrm{sh}}$, and $r^{2\alpha}\rho_{\alpha}\to 0$ locally uniformly for $\phi \in \mathbb{S}^{1}\backslash \{\pm\pi/2\}$. The desired $(-\alpha)$-homogeneous solutions are obtained by dividing $u_\alpha$ by $4/\pi$ and $p_{\alpha}$ and $\rho_{\alpha}$ by $(4/\pi)^{2}$, respectively.
\end{proof}

\appendix 

\section{The second kind self-similarities in hydrodynamics}

We review existence results on backward self-similar solutions to the following 
fluid equations:
\begin{itemize}
\item gCLM equations
\item Burgers equation
\item Prandtl equations
\item Isentropic compressible Euler equations
\end{itemize}
Those equations partially share the essential features with the 2D inviscid Boussinesq equations in a half-plane, such as advection and stretching of vorticity, incompressible flows in a half-plane, density variation, etc., and possess similar scaling laws; see Table 1.

\clearpage

\begin{landscape}
\begin{table}[]
\vspace{60pt}
\small
\begin{tabular}{|c|c|c|c|c|c|c|}
\hline
Equations                                                                                & Variables & Unknowns      & Scaling                                                                                                                                                                                                                                   & Exponents                                                                                                                                   & Methods                                                                           & References                                                                                                                                                 \\ \hline
gCLM                                                                                     & $1$       & $\omega$      & $\omega(x,t)=\lambda^{1+\alpha}\omega(\lambda x,\lambda^{1+\alpha}t )$                                                                                                                                                                    & $-2\leq \alpha(a)\leq 2$                                                                                                                    & \begin{tabular}[c]{@{}c@{}}Explicit solutions\\ Fixed point theorems\end{tabular} & \begin{tabular}[c]{@{}c@{}}Elgindi \& Jeong \cite{EJ20c} \\ Zhen \cite{Z23} \\ Chen et al. \cite{CHH21} \\ Huang et al. \cite{HTW} \\ Jia \& \v{S}ver\'{a}k \cite{SverakVideo} \\ Chen \cite{Chen20} \\ Lushikov et al. \cite{LSS} \\ Huang et al. \cite{HQWW}\end{tabular} \\ \hline
Burgers                                                                                  & $n$       & $u$           & \begin{tabular}[c]{@{}c@{}}$u(x_1,x_2,t)=\lambda^{\alpha}u(\lambda x_1,\lambda^{\beta} x_2, \lambda^{1+\alpha}t)$\\ ($n=2$)\end{tabular}                                                                                                  & \begin{tabular}[c]{@{}c@{}}$\alpha=-\displaystyle\frac{1}{3}$\\ $\beta=\gamma=\displaystyle\frac{1}{3}$\end{tabular}                        & \begin{tabular}[c]{@{}c@{}}Explicit solutions\\ Taylor expansion\end{tabular}     & \begin{tabular}[c]{@{}c@{}}Collot et al. \cite{CGM22} \\ Buckmaster et al. \cite{BSV23}\end{tabular}                                                                                 \\ \hline
Prandtl                                                                                  & $2$       & $u=(u_1,u_2)$ & \begin{tabular}[c]{@{}c@{}}$u_1(x_1,x_2,t)=\lambda^{\alpha}u_1(\lambda x_1,\lambda^{\beta} x_2, \lambda^{1+\alpha}t )$\\ $u_2(x_1,x_2,t)=\lambda^{\alpha-1+\beta}u_2(\lambda x_1,\lambda^{\beta} x_2, \lambda^{1+\alpha}t )$\end{tabular} & \begin{tabular}[c]{@{}c@{}}$\alpha=-\displaystyle\frac{1}{3}$\\ $\beta=-\displaystyle\frac{1}{6}$, $-\displaystyle\frac{1}{3}$\end{tabular} & \begin{tabular}[c]{@{}c@{}}Volume-preserving \\ Crocco transform\end{tabular}     & Collot et al. \cite{CGM21}                                                                                                                                              \\ \hline
\begin{tabular}[c]{@{}c@{}}Isentropic \\ Euler\\ (spherically \\ symmetric)\end{tabular} & $1$       & $(u,\rho)$    & \begin{tabular}[c]{@{}c@{}}$u(x,t)=\lambda^{\alpha}u(\lambda x,\lambda^{1+\alpha}t )$\\ $\rho^{\gamma-1}(x,t)=\lambda^{2\alpha}\rho^{\gamma-1}(\lambda x, \lambda^{1+\alpha}t )$\end{tabular}                                             & $\{\alpha_n\}\subset (0,\alpha_{\peye})$                                                                                                    & Phase portrait                                                                    & \begin{tabular}[c]{@{}c@{}}Merle et al. \cite{MRRS}\\ Buckmaster et al. \cite{BCG}\end{tabular}                                                                                   \\ \hline
Boussinesq                                                                               & $2$       & $(u,p,\rho)$  & \begin{tabular}[c]{@{}c@{}}$u(x,t)=\lambda^{\alpha}u(\lambda x,\lambda^{1+\alpha}t )$\\ $p(x,t)=\lambda^{2\alpha}p (\lambda x,\lambda^{1+\alpha}t )$\\ $\rho(x,t)=\lambda^{2\alpha+1}\rho (\lambda x,\lambda^{1+\alpha}t )$\end{tabular}  & $\alpha\approx -0.617$                                                                                                                      & \begin{tabular}[c]{@{}c@{}}Neural network\\ (numerical analysis)\end{tabular}              & Wang et al. \cite{WLGB}                                                                                                                                                \\ \hline
\end{tabular}
\vspace{15pt}
\caption{Comparison with backward self-similar solutions to other quasi-linear fluid equations}
\end{table}
\end{landscape}

\subsection{One-dimensional vortex singularities}

A canonical 1D blow-up model for the 3D Euler equations is the gCLM equations
\begin{equation}
\omega_t+au\omega_x=\omega u_x,\quad u_x=H\omega,\quad x\in \mathbb{R}.
\end{equation}
Constantin et al. \cite{CLM85} introduced this equation for $a=0$ as a 1D model for vorticity of the 3D Euler equations incorporating the vortex stretching and the Biot--Savart law with the Hilbert transform $H$ and demonstrated that the equation exhibits a blow-up. Later, De Gregorio \cite{DeG} added the advection term $a=1$, and the continuum model for $0\leq a\leq 1$ is due to Okamoto et al. \cite{OSW08}. See also C\'{o}rdoba et al. \cite{CCF} and Castro and C\'{o}rdoba \cite{CC10} for the case $a=-1$. The self-similar equations of the gCLM equations form 
\begin{align*}
\omega+\frac{1}{\alpha+1}x\omega_x+au\omega_x=\omega u_x,\quad u_x=H\omega,\quad x\in \mathbb{R}.
\end{align*}
Elgindi and Jeong \cite[1.5, Proposition 2.1]{EJ20c} found the explicit solution to the CLM equations ($a=0$) with the scaling exponent $\alpha=0$:
\begin{align*}
\omega(x)=\frac{x}{1+x^{2}},\quad u(x)=-\arctan{x}.
\end{align*}
The work \cite{EJ20c} demonstrated the existence of smooth profile functions $\omega\in C^{\infty}$ for small $|a|$ with some exponent $\alpha\in (-1,1)$,  and also the existence of H\"older continuous profile functions $\omega\in C^{\beta}$ for all $a$ with some exponent $\alpha\in (-1,2\beta-1)$ and $\beta\in (0,1)$. See also Zhen \cite{Z23}.

Chen et al. \cite{CHH21} discovered the existence of compactly supported self-similar solutions to the De Gregorio model ($a=1$) with the scaling exponent $\alpha=-2$. See also Huang et al. \cite{HTW}. 

Chen \cite{Chen20} and Lushikov et al. \cite{LSS} found the explicit solution for $a=1/2$ with the scaling exponent $\alpha=2$:
\begin{align*}
\omega(x)=-\frac{2bx}{(x^{2}+b^{2})^{2}},\quad u(x)=\frac{x}{x^{2}+b^{2}},\quad b=\sqrt{\frac{3}{8}}.
\end{align*}
Lushikov et al. \cite{LSS} numerically found the threshold $a_c\approx 0.689$ for concentrating $\alpha>-1$ and expanding $\alpha<-1$ self-similarities. Huang et al. \cite[Theorem 2.4]{HQWW} demonstrated the existence of solutions for all $a\leq 1$ with some exponent $\alpha(a)\in [-2,2]$. The constructed solutions satisfy the asymptotic 
\begin{align*}
\omega(x)=O\left(\frac{1}{|x|^{\alpha(a)+1}} \right)\quad \textrm{as}\ |x|\to\infty,
\end{align*}
for $\alpha(a)>-1$ and compact support for $\alpha(a)<-1$. The solutions \cite{EJ20c}, \cite{Z23}, \cite{HTW}, and \cite{HQWW} are constructed via fixed-point theorems. Figure 5 summarizes the advection constant $a\leq 1$ and the scaling exponent $\alpha\in \mathbb{R}$.

\begin{figure}[h]
\includegraphics[scale=0.241]{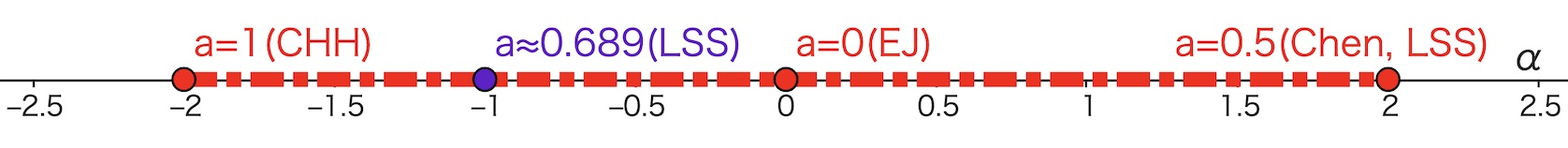}
\caption{The scaling exponent $\alpha\in \mathbb{R}$ of self-similar solutions to the gCLM equations. The red dots represent the scaling exponents of explicit solutions. The red dashed line represents the range where the scaling exponents $\alpha(a)$ of a family of solutions for $a\leq 1$\cite{HQWW} distribute.}
\end{figure}

\subsection{Multi-dimensional shock singularities}

A classic example of a nonlinear hyperbolic equation exhibiting a blow-up from smooth data is the 1D Burgers equation
\begin{align*}
u_t+uu_x=0,\quad x\in \mathbb{R}.
\end{align*}
The self-similar equation
\begin{align*}
\frac{\alpha}{\alpha+1}u+\frac{1}{\alpha+1}xu_x+uu_x=0,\quad x\in \mathbb{R},
\end{align*}
admits smooth solutions satisfying $u(x)=O(|x|^{-\alpha})$ as $|x|\to\infty$ for the scaling exponents 
\begin{align*}
\alpha=-\frac{1}{2i+1},\quad i\in \mathbb{N}.
\end{align*} 
See, e.g., \cite{EF09}. The solution for $\alpha=-1/3$ \cite{CSW} forming 
\begin{align*}
u_1(x)=\left(-\frac{x}{2}-\sqrt{\frac{x^{2}}{4} +\frac{1}{27}  }  \right)^{\frac{1}{3}}-\left(\frac{x}{2}+\sqrt{\frac{x^{2}}{4} +\frac{1}{27}  }  \right)^{\frac{1}{3}},
\end{align*}
is particularly important for the stability of blow-up solutions in the $1$D Burgers equation \cite{EF09} (This solution can be computed via the inverse equation $x=-u-u^{-1/\alpha}$ by setting $u=a^{1/3}-b^{1/3}$ and solving the quadratic equation $a-b=-x$ and $ab=1/27$). Non-smooth solutions exist for non-integers $i>0$ \cite{CGM22}.

Self-similar solutions depending on two variables involve an additional scaling exponent $\beta\in \mathbb{R}$; see Table 1. The associated self-similar equations form  
\begin{align*}
\frac{\alpha}{\alpha+1}u+\frac{1}{\alpha+1}xu_x+\frac{\beta}{\alpha+1}yu_y+uu_x=0,\quad (x,y)\in \mathbb{R}^{2}.
\end{align*}
Collot et al. \cite{CGM22} found the solution for $(\alpha,\beta)=(-1/3, 1/3)$:
\begin{align*}
u(x,y)=\frac{1}{\sqrt{F(y)}}u_1(F^{3/2}(y)x),\quad F(y)=\frac{1}{1+y^{2}},
\end{align*}
in the study of the Burgers equation with transverse viscosity $u_t+uu_x-u_{yy}=0$. More generally, the work \cite{CGM22} found the solutions $u(x,y)=F_k^{-1/2}(y)u_1(F_k^{3/2}(y)x )$, $F_k(y)=(1+y^{2k})^{-1}$ for $(\alpha,\beta)=(-1/3, 1/(3k))$ and $k\in \mathbb{N}$. 

Self-similar solutions depending on three variables involve two additional scaling exponents $\beta,\gamma\in \mathbb{R}$ and the associated self-similar equations form  
\begin{align*}
\frac{\alpha}{\alpha+1}u+\frac{1}{\alpha+1}xu_x+\frac{\beta}{\alpha+1}yu_y++\frac{\gamma}{\alpha+1}zu_z+uu_x=0,\quad (x,y,z)\in \mathbb{R}^{3}.
\end{align*}
The function $u(x,y,z)=F^{-1/2}(r)u_1(F^{3/2}(r)x )$ for $r=\sqrt{y^{2}+z^{2}}$ is a solution to this equation for $(\alpha,\beta,\gamma)=(-1/3,1/3,1/3)$ \cite[2.7]{BSV23}. Buckmaster et al. \cite[A.1]{BSV23} constructed a $10$D family of solutions to this equation for $(\alpha,\beta,\gamma)=(-1/3,1/3,1/3)$ by Taylor expansion and described shock formation of the 3D  isentropic compressible Euler equations, cf. \cite{BVS22}, \cite{BSV23b}. 

\subsection{Two-dimensional separation singularities}

The inviscid Prandtl equations model singularity formations of the Prandtl boundary layer equations: 
\begin{equation*}
\begin{aligned}
u_t+uu_x+vu_y&=0,  \\
u_x+v_y&=0,\quad (x,y)\in \mathbb{R}^{2}_{+},  \\
v(x,0,t)&=0,\quad x\in \mathbb{R}.
\end{aligned}
\end{equation*}
Self-similar solutions to these equations possess two scaling exponents $\alpha, \beta\in \mathbb{R}$ and the associated self-similar equation form 
\begin{align*}
\frac{\alpha}{\alpha+1}u+\frac{1}{\alpha+1}
x u_x
+\frac{\beta}{\alpha+1}yu_y
+uu_x+vu_y&=0,\\
u_x+v_y&=0,\quad (x,y)\in \mathbb{R}^{2}_{+},\\
v(x,0)&=0,\quad x\in \mathbb{R}.
\end{align*}
Collot et al. \cite[Lemma 3.1]{CGM22} found a new way to solve this equation for $\alpha\neq 0$ using a volume-preserving Crocco transform. Namely, $(u,v)$ solves the self-similar equation if and only if there exists a function $b(x,y)$ such that the mapping $(x,y)\longmapsto (a,b)$ for $a(x,y)=-u(x,y)$ is a volume-preserving diffeomorphism and the inverse function $x=x(a,b)$ solves the equation 
\begin{align*}
-\frac{\alpha}{\alpha+1}ax_a+\frac{\beta-\alpha+1}{\alpha+1}
b x_b=\frac{1}{\alpha+1}x-a.
\end{align*}
The work \cite[Lemma 3.1]{CGM22} found two explicit solutions for $\alpha=-1/3$:
\begin{align*}
&x=a+b^{2}+Ca^{3},\quad  \beta=-\frac{1}{6},\\
&x=a+a^{3}+\frac{1}{4}ab^{2},\quad \beta=-\frac{1}{3}.
\end{align*}
The case $\beta=-1/6$ is a generic boundary layer separation singularity relevant to Van-Dommelen--Shen singularity, and the case $\beta=-1/3$ is a degenerate singularity relevant to Burgers shock; see \cite{CGIM22} for a stability result in the Prandtl boundary layer equations.

\subsection{Imploding density singularities}

The isentropic compressible Euler equations with the ideal gas law is a system for gas dynamics involving the density:
\begin{equation*}
\begin{aligned}
u_t+u\cdot \nabla u+\nabla (\rho^{\gamma-1})&=0,\\
\rho_t +\nabla \cdot (u \rho)&=0,\\
\rho&>0,\quad x\in \mathbb{R}^{n},
\end{aligned}
\end{equation*}
with the adiabatic exponent $\gamma>1$. Its self-similar equations form 
\begin{align*}
\frac{\alpha}{\alpha+1}u+\frac{1}{\alpha+1}x\cdot \nabla u+u\cdot \nabla u+\nabla \rho^{\gamma-1}&=0,\\
\frac{2\alpha}{\alpha+1}\rho^{\gamma-1}+\frac{1}{\alpha+1}x\cdot \nabla \rho^{\gamma-1}+u\cdot \nabla \rho^{\gamma-1} +(\gamma-1)(\nabla \cdot u) \rho^{\gamma-1}&=0,\quad x\in \mathbb{R}^{n}.
\end{align*}
These equations can be expressed in the following simpler form by changing the variable $(u,\rho)(x)=(\tilde{u},\tilde{\rho})(x/(\alpha+1))$ and denoting $(\tilde{u},\tilde{\rho})$ again by $(u,\rho)$: 
\begin{align*}
\alpha u+(x+u)\cdot \nabla u+\nabla \rho^{\gamma-1}&=0,\\
2\alpha \rho^{\gamma-1}+(x+u)\cdot \nabla \rho^{\gamma-1} +(\gamma-1)(\nabla \cdot u) \rho^{\gamma-1}&=0,\quad x\in \mathbb{R}^{n}.
\end{align*}
Spherically symmetric solutions $u=f(r)x/|x|$ and $\rho=\rho(r)$ for $r=|x|$ satisfy the $1$D non-autonomous equations 
\begin{align*}
\alpha f+(r+f)f_r+(\rho^{\gamma-1})_{r}&=0,\\
2\alpha \rho^{\gamma-1}+(r+f)(\rho^{\gamma-1})_r+(\gamma-1)\left(f_r+\frac{n-1}{r}f \right) \rho^{\gamma-1}&=0,\quad r>0,
\end{align*}
and by the Emden transform, 
\begin{align*}
f=-r w(\log {r}),\quad \rho^{\gamma-1}=\frac{1}{\gamma-1}r^{2}\sigma^{2}(\log{r}), 
\end{align*}
this system is reduced to the autonomous equations for $(w,\sigma)$ with the constant $l=2/(\gamma-1)$:
\begin{align*}
(w-1)w_s+w^{2}-(\alpha+1)w+l (\sigma^{2}+\sigma\sigma_s)&=0,\\
\frac{1}{l}\sigma w_s+(w-1)\sigma_s+\left\{\left(\frac{n}{l}+1 \right)w-\alpha-1\right\}\sigma&=0,\quad s\in \mathbb{R}.
\end{align*}
With the polynomials and the constant 
\begin{align*}
\Delta &=(w-1)^{2}-\sigma^{2},\\
\Delta_1&=w(w-1)(w-\alpha-1)-n(w-w_e)\sigma^{2},\\
\Delta_2&= \frac{\sigma}{l}\left((l+n-1)w^{2}-\left(l+n+(l-1)(\alpha+1)\right)w+l(\alpha+1)+l\sigma^{2}  \right),\\
w_e&=\frac{l\alpha}{n},
\end{align*}
the system is expressed as follows:
\begin{align*}
\Delta \frac{dw}{ds}=-\Delta_1,\quad \Delta \frac{d\sigma}{ds}=-\Delta_2. 
\end{align*}
Merle et al. \cite{MRRS} investigated the phase portrait of this system with the parameters $(n,l,\alpha)$ in the $(\sigma,w)$ plane and demonstrated the existence of solutions for $n\geq 2$. In particular for $n=2$ and $3$, it is shown \cite[Theorem 1.3]{MRRS} that there exists a function $S_{\infty}(n,l)$ such that for all $l\in (0,\infty)\backslash \{n\}$ such that $S_{\infty}(n,l)\neq 0$, there exists a sequence $\{\alpha_n\}\subset (0,\alpha_{\peye})$ such that $\alpha_n\to \alpha_{\peye}$ for 
\begin{align*}
\alpha_{\peye}=\begin{cases}
&\displaystyle \frac{n+l}{l+\sqrt{n}}-1,\quad 0<l<n,\\
&\displaystyle \frac{n-1}{\left(1+\sqrt{l}\right)^{2}},\quad n<l,
\end{cases}
\end{align*}
and $\alpha=\alpha_n$ admits a smooth spherically symmetric profile function $(u,\rho)$ satisfying 
\begin{align*}
u(x)=O\left( \frac{1}{|x|^{\frac{\alpha}{\alpha+1}}}\right),\quad \rho(x)=O\left( \frac{1}{|x|^{\frac{2\alpha}{\alpha+1}}}\right)\quad \textrm{as}\ |x|\to\infty.
\end{align*}
The mass of this density $\rho$ is infinite. Buckmaster et al. \cite{BCG} demonstrated the existence of smooth spherically symmetric solutions for all adiabatic exponent $\gamma>1$ by studying phase portraits associated with Riemann invariants; see \cite{BCG2} for a review.

\begin{center}
\sc{Conflict of Interest}
\end{center}
The authors declare that they have no conflict of interest.

\begin{center}
\sc{Data availability}
\end{center}
Data sharing is not applicable to this article as no datasets were generated or analyzed during the current study.

\bibliographystyle{alpha}
\bibliography{ref}

\end{document}